\documentclass[12pt,a4paper]{amsart}
\usepackage{amsxtra,amssymb}

\calclayout

\newcommand{\+}{\nobreakdash-}
\renewcommand{\:}{\colon}
\renewcommand{\.}{\mskip.5\thinmuskip}

\DeclareMathOperator{\Hom}{Hom}
\DeclareMathOperator{\Ext}{Ext}
\DeclareMathOperator{\Tor}{Tor}
\DeclareMathOperator{\Spec}{Spec}
\DeclareMathOperator{\Ass}{Ass}
\DeclareMathOperator{\Ann}{Ann}

\DeclareMathOperator{\coker}{coker}
\DeclareMathOperator{\id}{id}
\DeclareMathOperator{\im}{im}
\DeclareMathOperator{\pd}{pd}

\newcommand{\Ab}{\mathsf{Ab}}

\newcommand{\rarrow}{\longrightarrow}
\newcommand{\larrow}{\longleftarrow}
\newcommand{\ot}{\otimes}

\newcommand{\Z}{\mathbb Z}
\newcommand{\Q}{\mathbb Q}

\newcommand{\A}{\mathsf A}
\newcommand{\B}{\mathsf B}
\newcommand{\F}{\mathsf F}
\newcommand{\C}{\mathsf C}
\newcommand{\D}{\mathsf D}
\newcommand{\sS}{\mathsf S}
\newcommand{\T}{\mathsf T}

\renewcommand{\b}{\mathsf b}

\newcommand{\cM}{\mathcal M}
\newcommand{\cK}{\mathcal K}

\newcommand{\m}{\mathfrak m}
\newcommand{\p}{\mathfrak p}
\newcommand{\q}{\mathfrak q}

\newcommand{\modl}{{\operatorname{\mathsf{--mod}}}}
\newcommand{\modr}{{\operatorname{\mathsf{mod--}}}}

\newcommand{\tors}{{\operatorname{\mathsf{-tors}}}}
\newcommand{\ctra}{{\operatorname{\mathsf{-ctra}}}}
\newcommand{\secmp}{{\operatorname{\mathsf{-secmp}}}}

\newcommand{\filt}{\operatorname{\mathsf{filt}}}
\newcommand{\fl}{\mathsf{fl}}
\newcommand{\vfl}{\mathsf{vfl}}
\renewcommand{\cot}{\mathsf{cot}}
\newcommand{\ctaa}{\mathsf{ctaa}}

\renewcommand{\div}{\mathrm{div}}
\newcommand{\red}{\mathrm{red}}

\newcommand{\boL}{\mathbb L}
\newcommand{\boR}{\mathbb R}

\newcommand{\bu}{{\text{\smaller\smaller$\scriptstyle\bullet$}}}
\newcommand{\lrarrow}{\.\relbar\joinrel\relbar\joinrel\rightarrow\.}
\newcommand{\llarrow}{\.\leftarrow\joinrel\relbar\joinrel\relbar\.}

\newcommand{\Section}[1]{\bigskip\section{#1}\medskip}

\theoremstyle{plain}
\newtheorem{thm}{Theorem}[section]

\newtheorem{lem}[thm]{Lemma}

\newtheorem{cor}[thm]{Corollary}
\theoremstyle{definition}
\newtheorem{ex}[thm]{Example}
\newtheorem{exs}[thm]{Examples}
\newtheorem{rem}[thm]{Remark}

\begin{document}

\title{Contraadjusted modules, contramodules, \\ and reduced
cotorsion modules}

\author{Leonid Positselski}

\address{Department of Mathematics, Faculty of Natural Sciences,
University of Haifa, Mount Carmel, Haifa 31905, Israel; and
\newline\indent Laboratory of Algebraic Geometry, National Research
University Higher School of Economics, Moscow 119048; and
\newline\indent Sector of Algebra and Number Theory, Institute for
Information Transmission Problems, Moscow 127051, Russia}

\email{posic@mccme.ru}

\begin{abstract}
 This paper is devoted to the more elementary aspects of
the contramodule story, and can be viewed as an extended introduction
to the more technically complicated~\cite{Pmgm}.
 Reduced cotorsion abelian groups form an abelian category, which is in
some sense covariantly dual to the category of torsion abelian groups.
 An abelian group is reduced cotorsion if and only if it is isomorphic
to a product of $p$\+contramodule abelian groups over prime numbers~$p$.
 Any $p$\+contraadjusted abelian group is $p$\+adically complete, and
any $p$\+adically separated and complete group is a $p$\+contramodule,
but the converse assertions are not true.
 In some form, these results hold for modules over arbitrary commutative
rings, while other formulations are applicable to modules over
one-dimensional Noetherian rings.
\end{abstract}

\maketitle

\tableofcontents

\allowdisplaybreaks

\section*{Introduction}
\medskip

 An abelian group $C$ is said to be \emph{cotorsion} if for any
torsion-free abelian group $F$ one has $\Ext^1_\Z(F,C)=0$.
 Considering the short exact sequence $F\rarrow\Q\ot_\Z F\rarrow
\Q/\Z\ot_\Z F$ and having in mind that the category of abelian groups
has homological dimension~$1$, one easily concludes that $C$ is
cotorsion if and only if $\Ext^1_\Z(\Q,C)=0$.

 An abelian group is said to be \emph{reduced} if it has no divisible
subgroups.
 Clearly, any abelian group $B$ has a unique maximal divisible subgroup
$B_\div\subset B$; the quotient group $B_\red=B/B_\div$ is the maximal
reduced quotient group of~$B$.
 The short exact sequence $B_\div\rarrow B\rarrow B_\red$ is
(noncanonically) split.
 For any element $b\in B_\div$, one can construct an abelian group
homomorphism $\Q\rarrow B$ taking $1$ to~$b$.
 So the subgroup $B_\div\subset B$ can be computed as the image of
the restriction map $\Hom_\Z(\Q,B)\rarrow B$ induced by the embedding
$\Z\rarrow\Q$.
 In particular, an abelian group $B$ is reduced if and only if
$\Hom_\Z(\Q,B)=0$.

 For any abelian category $\A$ and any object $U$ of projective
dimension~$\le 1$ in $\A$, the full subcategory $\C\subset\A$
consisting of all the objects $C$ satisfying $\Hom_\A(U,C)=0=
\Ext^1_\A(U,C)$ is closed under the kernels and cokernels of morphisms
between its objects, and also extensions and infinite products,
taken in~$\A$.
 It follows that the category $\C$ is abelian and its embedding
functor $\C\rarrow\A$ is exact.

 We have explained that \emph{the category of reduced cotorsion
abelian groups is abelian}.
 One can try to read this assertion between the lines of
the exposition in the book~\cite[\S54]{Fuc}.
 The feeling that this fact deserves to be emphasized was one of
the motivating impulses for writing this paper.
 In fact, the term ``co\+torsion group'' first appeared in
the paper~\cite{Harr}, where it meant what we now call
``reduced cotorsion''. 

 Notice that the category of torsion abelian groups, consisting of
all the groups $T$ satisfying $\Q\ot_\Z T=0$, has even stronger
properties: it contains all the subobjects and quotients of its
objects, in addition to extensions (and infinite direct sums) in
the category of abelian groups $\Ab$; so it is a Serre subcategory.
 That is because the group $\Q$ has flat dimension~$0$, but
projective dimension~$1$.

 Just as the category of torsion abelian groups, the category of
reduced cotorsion abelian groups splits into a Cartesian product
of categories indexed by the prime numbers.
 In fact, the functor right adjoint to the embedding of the full
subcategory of torsion abelian groups (i.~e., the maximal torsion 
subgroup functor) can be computed as $\Tor^1_\Z(\Q/\Z,-)$, and
the functor left adjoint to the embedding of the full subcategory
of reduced cotorsion abelian groups into $\Ab$ can be computed as
$\Ext^1_\Z(\Q/\Z,-)$ (cf.~\cite[Theorem~7.1]{Nun}).
 Decomposing $\Q/\Z$ into the direct sum of its $p$\+primary
components $\bigoplus_p\Q_p/\Z_p$, one concludes that an abelian
group $C$ is reduced cotorsion if and only if it can be (always
uniquely) presented as a product of abelian groups of the form
$\Ext^1_\Z(\Q_p/\Z_p,A)$, where $A\in\Ab$.
 Moreover, one can take $A=C$.

 Abelian groups of the form $\Ext^1_\Z(\Q_p/\Z_p,A)$ form an abelian
category which we call the \emph{category of $p$\+contramodule
abelian groups}.
 It can be described as the full subcategory of all abelian groups
$C$ satisfying $\Hom_\Z(\Z[p^{-1}],C)=0=\Ext^1_\Z(\Z[p^{-1}],C)$.
 Actually, both the abelian categories of $p$\+contramodule and
reduced cotorsion abelian groups lie in the intersection of
two natural classes of abelian categories (so there are two
independent ways to explain why they are abelian).
 Besides their above description as the $\Ext^{0,1}$\+orthogonality
classes, they can be also defined as the categories of
\emph{contramodules over the topological rings} of $p$\+adic integers
$\Z_p$ and finite integral ad\`eles $\prod_p\Z_p$, respectively.
 These are modules with infinite summation operations, as
introduced in~\cite[Remark~A.3]{Psemi} and~\cite{Pweak},
and overviewed in~\cite{Prev}.

 An abelian group $C$ is called \emph{$p$\+contraadjusted} if
$\Ext^1_\Z(\Z[p^{-1}],C)=0$.
 Similarly one defines $s$\+contraadjusted abelian groups for
any natural number~$s$; and a group is called \emph{contraadjusted}
if it is $s$\+contraadjusted for all $s\ge2$.
 Contraadjusted modules were introduced in connection with
contraherent cosheaves~\cite{Pcosh} and further
studied in~\cite{Sl,ST}.
 For any $p$\+contraadjusted abelian group $C$, the natural map
into the $p$\+adic completion $C\rarrow\varprojlim_n C/p^nC$ is
surjective.
 When this map is an isomorphism, the group $C$ is even
a $p$\+contramodule.
 The converse implications to both these assertions fail in general,
but they are true for $p$\+torsion-free abelian groups.

 Most of the results mentioned above extend to modules over
an arbitrary commutative ring $R$ with a fixed element $s\in R$,
or even with a fixed ideal $I\subset R$.
 The results based on the decomposition of the quotient module
$Q/R$, where $Q$ is the ring/field of fractions of $R$, into
a direct sum over the maximal ideals, are the main exception.
 These require $R$ to be a Noetherian ring of Krull dimension~$1$.

 The original motivation for our study of the cotorsion and
contraadjusted modules comes from algebraic geometry.
 Given a commutative ring $R$, the category of contraherent cosheaves
over the affine scheme $\Spec R$ is equivalent to the category of
contraadjusted $R$\+modules, while the category of locally cotorsion
contraherent cosheaves on $\Spec R$ is equivalent to the category
of cotorsion $R$\+modules~\cite{Pcosh}.
 One would like to have a supply of examples and easily computable
special (e.~g., low-dimensional) cases, just for developing
an intuition of what the contraherent cosheaves are.
 From this geometric point of view, cotorsion and contraadjusted
abelian groups are nothing but the (global cosections of) contraherent
cosheaves over $\Spec\Z$.

 The experience seems to teach that the locally cotorsion contraherent
cosheaves behave well on Noetherian schemes, while for schemes of
more general nature there are not enough of these, and one wishes to
consider arbitrary (i.~e., locally contraadjusted) contraherent
cosheaves.
 Hence the importance of cotorsion modules over Noetherian rings and
contraadjusted modules over both Noetherian and non-Noetherian rings,
from our geometric standpoint.
 Flat cotorsion modules, corresponding geometrically to projective
locally cotorsion contraherent cosheaves, over Noetherian rings were
classified by Enochs~\cite{En} (see~\cite[Section~5.1]{Pcosh}
for the nonaffine case).
 Arbitrary cotorsion modules over Noetherian rings of Krull
dimension~$1$ are described (and contraadjusted modules discussed)
in the present paper, based on the approaches originated by
Nunke~\cite{Nun} and Sl\'avik--Trlifaj~\cite{Sl,ST}.

 Concerning the answers that we seem to obtain, the following vague
analogy may be illuminating.
 Quasi-coherent sheaves over $\Spec\Z$ are the same thing as abelian
groups, and coherent sheaves correspond to finitely generated abelian
groups.
 Arbitrary abelian groups are hopelessly complicated, but the finitely
generated ones (or, say, finitely generated modules over a Dedekind
domain) can be classified.
 In fact, they are described as finite direct sums, where one summand
is a free group (respectively, a projective module) and the remaining
ones are torsion modules indexed by the prime numbers (sitting at
the closed points of the spectrum).

 Likewise, contraadjusted abelian groups may be too complicated to
describe, but one can say a lot about cotorsion groups or
cotorsion modules over Dedekind domains.
 In fact, the latter are decribed as the infinite products of
a divisible group (\,$=$~injective module) and contramodules
sitting at the prime numbers (closed points).
 So one obtains a geometric picture of locally cotorsion contraherent
cosheaves over smooth curves bearing some vague similarity to
the classification of coherent sheaves over such curves.
 For Noetherian rings of Krull dimension~$1$ (\,$=$~singular
curves) there is essentially the same description of reduced
cotorsion modules as the products of contramodules (sitting at
the points), while an injective direct summand turns into
a more complicated (noninjective) divisible submodule.

\medskip

\textbf{Acknowledgement.}
 This paper grew out of my visits to Prague for a conference in
September and then for a longer stay in November--December~2015.
 I~would like to express my gratitude to Jan Trlifaj and Alexander
Sl\'avik for stimulating discussions.
 In particular, Alexander's presentations influenced the present
work, and he told me about the reference~\cite{Nun}.
 I~learned from Jan about the $R$\+topology and
the references~\cite[\S54\+-55]{Fuc} and~\cite{FS}.
 A large part of this paper is but an extended writeup of my last
talk in Prague in December.

 Subsequently in Israel, I~had a conversation with Joseph Bernstein,
and I~want to thank him for pointing out that abelian
categories are rare and interesting.
 Last but not least, I~am grateful to Vladimir Hinich and
the Israeli Academy of Sciences for warm hospitality in offering me
a postdoctoral position at the University of Haifa, with excellent
working conditions and financial support.
 I also want to thank the anonymous referee for a number of helpful
remarks and suggestions.
 The author was supported by the ISF grant~\#\,446/15 in Israel and
by the Grant Agency of the Czech Republic under the grant
P201/12/G028 in Prague.

\Section{Orthogonality Classes as Abelian Categories}

 All rings and modules in this paper are unital.
 
 Let $R$ be an associative ring.
 We denote by $R\modl$ and $\modr R$ the categories of (arbitrarily
large) left and right $R$\+modules, respectively.

 Given an abelian category $\A$, a full subcategory $\C\subset\A$ is
said to be closed under subobjects (respectively, quotient objects)
if every subobject (resp., quotient object) of an object of the class
$\C$ in the category $\A$ also belongs to~$\C$.
 A full subcategory closed under subobjects, quotients, and extensions
is called a \emph{Serre subcategory}.
 For such subcategories $\C\subset\A$, an abelian quotient category
$\A/\C$ is defined.

 A full subcategory $\C\subset\A$ is said to be closed under the kernels
(respectively, cokernels), if for every morphism $f\:C\rarrow D$ in $\A$
between two objects $C$, $D\in\C$ the kernel $\ker_\A(f)$ (resp.,
the cokernel $\coker_\A(f)$) of the morphism~$f$ computed in
the category $\A$ belongs to~$\C$.
 A full subcategory $\C$ closed under the kernels, cokernels, and
finite direct sums in an abelian category $\A$ is also an abelian
category with an exact embedding functor $\C\rarrow\A$.

\begin{thm} \label{serre-subcategory}
\textup{(a)} Let $U$ be a projective left $R$\+module.
 Then the full subcategory\/ $\C\subset R\modl$ formed by all
the left $R$\+modules $C$ for which\/ $\Hom_R(U,C)=0$ is closed
under subobjects, quotient objects, extensions, and infinite
products in $R\modl$.
 In particular, $\C\subset R\modl$ is a Serre subcategory. \par
\textup{(b)} Let $U$ be a flat right $R$\+module.
 Then the full subcategory\/ $\T\subset R\modl$ formed by all
the left $R$\+modules $T$ for which $U\ot_R T=0$ is closed
under subobjects, quotient objects, extensions, and infinite
direct sums in $R\modl$.
 In particular, $\T\subset R\modl$ is a Serre subcategory.
\end{thm}

\begin{proof}
 More generally, for any abelian categories $\A$ and $\B$ and
any exact functor $F\:\A\rarrow\B$ the full subcategory
$\sS\subset\A$ formed by all the objects $S\in A$ for which
$F(S)=0$ is a Serre subcategory in~$\A$.
\end{proof}

 The following result can be found in~\cite[Proposition~1.1]{GL}.

\begin{thm} \label{ext-0-1-orthogonal}
\textup{(a)} Let $U$ be a left $R$\+module of projective
dimension~$\le 1$.
 Then the full subcategory\/ $\C\subset R\modl$ formed by all
the left $R$\+modules $C$ for which\/ $\Hom_R(U,C)=0=\Ext^1_R(U,C)$
is closed under the kernels, cokernels, extensions, and infinite
products in $R\modl$.
 In particular, $\C$ is an abelian category and
the embedding functor\/ $\C\rarrow R\modl$ is exact. \par
\textup{(b)} Let $U$ be a right $R$\+module of flat
dimension~$\le 1$.
 Then the full subcategory\/ $\T\subset R\modl$ formed by all
the left $R$\+modules $T$ for which $U\ot_R T=0=\Tor_1^R(U,T)$
is closed under the kernels, cokernels, extensions, and infinite
direct sums in $R\modl$.
 In particular, $\T$ is an abelian category and
the embedding functor\/ $\T\rarrow R\modl$ is exact.
\end{thm}

\begin{proof}
 Generally, let $\A$, $\B$ be abelian categories and $(F^0,F^1)$
be a cohomological $\delta$\+functor (of cohomological
dimension~$\le1$) from $\A$ to $\B$, that is a pair of functors
$F^0$, $F^1\:\A\rarrow B$ together with a $6$\+term exact sequence
\begin{multline*}
 0\lrarrow F^0(X)\lrarrow F^0(Y)\lrarrow F^0(Z) \\
 \lrarrow F^1(X)\lrarrow F^1(Y)\lrarrow F^1(Z)\lrarrow0
\end{multline*}
in $\B$ defined functorially for any short exact sequence
$0\rarrow X\rarrow Y\rarrow Z\rarrow0$ in~$\A$.
 Then the full subcategory $\C\subset\A$ formed by all the objects
$C$ for which $F^0(C)=0=F^1(C)$ is closed under the kernels,
cokernels, and extensions in~$\A$.

 Indeed, the class of all object $X\in\A$ such that $F^0(X)=0$
is closed under subobjects and extensions, while the class of all
$Z\in\A$ such that $F^1(Z)=0$ is closed under extensions and
quotients.
 Now let $f\:C\rarrow D$ be a morphism between two objects
$C$, $D\in\C$ and let $E$ be the image of~$f$.
 Then $F^0(E)=0$, as $E$ is a subobject in $D$; and
$F^1(E)=0$, because $E$ is a quotient of~$C$.
 Finally, let $K$ and $L$ denote the kernel and cokernel of
the morphism~$f$.
 Considering the long exact sequence of functor $F^*$ for
the short exact sequences $0\rarrow K\rarrow C\rarrow E\rarrow0$
and $0\rarrow E\rarrow D\rarrow L\rarrow0$ in the category $\A$,
one concludes that the objects $K$ and $L$ belong to~$\C$.
\end{proof}

\begin{rem}
 It is clear from the proof of Theorem~\ref{ext-0-1-orthogonal} that
a module $U$ of projective (respectively, flat) dimension~$\le1$ in
its formulation can be replaced by a two-term complex
$U^{-1}\overset u\rarrow U^0$ of projective (resp., flat) $R$\+modules.
 In other words, a module can be replaced with an object of
the derived category of modules with the similar restriction on
the projective/flat dimension.
 Then it is claimed that the full subcategory of all modules $C$
for which $\Hom_R(u,C)$ is an isomorphism (resp., all modules $T$
for which $u\ot_RT$ is an isomorphism) is closed under the kernels,
cokernels, and extensions in $R\modl$; so it is an abelian category.
 One just applies the above argument to the $\delta$\+functor
$(F^0,F^1)$ with $F^0(X)=\ker\Hom_R(u,X)$ and $F^1(X)=\coker
\Hom_R(u,X)$.
 Similarly, if $V^0\overset v\rarrow V^1$ is a two-term complex of
injective left $R$\+modules, then the left $R$\+modules $T$ for which
$\Hom_R(T,v)$ is an isomorphism form a full subcategory closed under
the kernels, cokernels, extensions, and infinite direct sums in
$R\modl$, hence an abelian category.
 (Cf.\ the discussion of infinite systems of nonhomogeneous linear
equations in the next section.)
\end{rem}

\Section{$s$-Contraadjustedness, $s$-Contramoduleness,
and $s$-Completeness}

 Let $R$ be a commutative ring and $s\in R$ be a fixed element.
 In this section, we will consider $R$\+modules $C$ with
certain conditions imposed on the action of $s$ in~$C$.
 The action of the rest of $R$ will be less important, and
almost just as well we could be talking about abelian groups
$C$ endowed with an endomorphism $s\:C\rarrow C$.

 An $R$\+module $C$ is called \emph{$s$\+torsion-free} if
the operator $s\:C\rarrow C$ is injective.
 An $R$\+module $C$ is called \emph{$s$\+divisible} if
the map $s\:C\rarrow C$ is surjective.

 We denote by $R[s^{-1}]$ the localization of $R$ with respect to
the multiplicative system $\{1$, $s$, $s^2$,~\dots,
$s^n$,~\dots~$\}\subset R$.
 An $R$\+module $C$ is said to be \emph{$s$\+contraadjusted} if
$\Ext^1_R(R[s^{-1}],C)=0$.
 An $R$\+module $C$ is called an \emph{$s$\+contramodule} (or an
\emph{$s$\+contra\-module $R$\+module}) if $\Hom_R(R[s^{-1}],C)=0
=\Ext^1_R(R[s^{-1}],C)$.

 Clearly, one has $\Hom_R(R[s^{-1}],C)=0$ if and only if there are
no nonzero $s$\+divisible $R$\+submodules in $C$, or which is
equivalent, if there are no nonzero abelian subgroups $D\subset C$ 
for which $sD=D$.
 The following lemma~\cite[Lemma~B.7.1]{Pweak},
\cite[Lemma~5.1]{ST}  explains what does the condition
$\Ext^1_R(R[s^{-1}],C)=0$ mean.

\begin{lem}  \label{ext-r-s-minus-1-computed}
\textup{(a)} An $R$\+module $C$ is $s$\+contraadjusted if and only if
for any sequence of elements~$a_0$, $a_1$, $a_2$,~\dots~$\in C$
the infinite system of nonhomogeneous linear equations
\begin{equation} \label{a-b-main-equation-system}
 b_n-sb_{n+1}=a_n, \qquad n\ge0
\end{equation}
has a solution~$b_0$, $b_1$, $b_2$,~\dots~$\in C$. \par
\textup{(b)} An $R$\+module $C$ is an $s$\+contramodule if and only if
for any sequence of elements~$a_0$, $a_1$, $a_2$,~\dots~$\in C$
the infinite system of nonhomogeneous linear
equations\/~\eqref{a-b-main-equation-system} has a \emph{unique}
solution~$b_0$, $b_1$, $b_2$,~\dots~$\in C$.
\end{lem}

\begin{proof}
 Computing $R[s^{-1}]\simeq\varinjlim_{n\ge0} R$ (with all the maps
$R\rarrow R$ being the multiplication with~$s$) using the telescope
construction for $\Z_{\ge0}$\+indexed inductive limits, one obtains
a free resolution of the form
$$
 0\lrarrow\bigoplus\nolimits_{n=0}^\infty Rf_n\lrarrow
 \bigoplus\nolimits_{n=0}^\infty Re_n\lrarrow R[s^{-1}]\lrarrow 0,
$$
with the differential taking the basis vector $f_n$ to $e_n-se_{n+1}$,
for the $R$\+module $R[s^{-1}]$.
 Applying the functor $\Hom_R({-},C)$, one computes the $R$\+modules
$\Hom_R(R[s^{-1}],C)$ and $\Ext^1_R(R[s^{-1}],C)$ as, respectively,
the kernel and the cokernel of the map
$$
 \prod\nolimits_{n=0}^\infty C\lrarrow\prod\nolimits_{n=0}^\infty C
$$
taking a sequence $(b_n)_{n=0}^\infty$ to the sequence
$(a_n=b_n-sb_{n+1})_{n=0}^\infty$.
\end{proof}

\begin{lem} \label{s-contraadjusted-closure-properties}
 The class of $s$\+contraadjusted $R$\+modules is closed under
quotients, extensions, and infinite products in $R\modl$.
\end{lem}

\begin{proof}
 See the proof of Theorem~\ref{ext-0-1-orthogonal}.
 Alternatively, the closedness with respect to quotients
can be easily deduced from Lemma~\ref{ext-r-s-minus-1-computed}(a).
\end{proof}

 An $R$\+module $C$ is said to be \emph{$s$\+adically complete}
(or simply \emph{$s$\+complete}) if the natural map from it to
its $s$\+adic completion
$$
 \lambda_{s,C}\:C\lrarrow\varprojlim\nolimits_{n\ge1} C/s^nC
$$
is surjective.
 The $R$\+module $C$ is called \emph{$s$\+adically separated}
(or $s$\+separated) if the map $\lambda_{s,C}$ is injective.
 Clearly, the kernel of $\lambda_{s,C}$ is the submodule $\bigcap_n s^nC
\subset C$.

 The following result can be found in~\cite[Lemma~4.4]{Sl}
and~\cite[Lemma~5.4]{ST}.

\begin{thm} \label{s-contraadjusted-complete}
\textup{(a)} Any $s$\+contraadjusted $R$\+module is $s$\+complete. 
(Hence, in particular, any $s$\+contramodule $R$\+module is
$s$\+complete.) \par
\textup{(b)} Any $s$\+torsion-free $s$\+complete $R$\+module is
$s$\+contraadjusted.
\end{thm}

\begin{proof}
 By the definition, $s$\+completeness of $C$ means that for any
sequence of elements $c_n\in C$, $n\ge1$ such that
$c_{n+1}-c_n\in s^nC$ for all $n\ge1$ there exists an element
$b\in C$ for which $b-c_n\in s^nC$ for all $n\ge1$.
 Set $a_0=c_1$; and pick elements $a_n\in C$ such that
$c_{n+1}=c_n+s^na_n$ for all $n\ge1$.
 Then we have $c_n=a_0+sa_1+\dotsb+s^{n-1}a_{n-1}$ for $n\ge1$.
 This shows that an $R$\+module $C$ is $s$\+complete if and only if
for every sequence of elements $a_n\in C$, $n\ge0$ there exists
an element $b\in C$ such that
$$
 b-a_0-sa_1-\dotsb-s^{n-1}a_{n-1}\in s^nC 
 \qquad\text{for all \ $n\ge1$}.
$$
 Set $b_0=b$.
 Now the latter condition on $C$ can be expressed as
the solvability of the system of linear equations
$$
 b_0-s^{n+1}b_{n+1} = a_0+sa_1+\dotsb+s^na_n, \qquad n\ge0
$$
in $b_0$, $b_1$, $b_2$,~\dots~$\in C$ for all
$a_0$, $a_1$, $a_2$~\dots~$\in C$.
 Finally, we can rewrite this system of equations equivalently as
\begin{equation} \label{a-b-bad-equation-system}
 s^n(b_n-sb_{n+1}) = s^na_n, \qquad n\ge0,
\end{equation}
because one always has $b_0-s^{n+1}b_{n+1}=
(b_0-sb_1)+s(b_1-sb_2)+\dotsb+s^n(b_n-sb_{n+1})$.

 We have shown that an $R$\+module $C$ is $s$\+complete if and only
if the system of nonhomogeneous linear
equations~\eqref{a-b-bad-equation-system} is solvable in
$(b_n\in C)_{n=0}^\infty$ for every sequence
$(a_n\in C)_{n=0}^\infty$.
 Now it is obvious that the system of
equations~\eqref{a-b-main-equation-system}
implies~\eqref{a-b-bad-equation-system}, and the converse
implication holds whenever $sc=0$ implies $c=0$ for $c\in C$.
 This proves both the assertions~(a) and~(b).
\end{proof}

 A generalization of Theorem~\ref{s-contraadjusted-complete}(a)
to the $I$\+adic completions for finitely generated nonprincipal ideals
$I$ will be given below in
Theorem~\ref{contraadjusted-for-ideal-generators-complete}.
 A precise criterion of contraadjustedness extending the one provided
by Theorem~\ref{s-contraadjusted-complete}(a\+b) to
the non-$s$-torsion-free module case can be found in
Corollary~\ref{s-contraadjusted-criterion}(b)
(see also Remark~\ref{bounded-torsion-remark}).

\begin{thm} \label{s-separ-complete-contra}
\textup{(a)} Any $s$\+separated $s$\+complete $R$\+module is
an $s$\+contramodule. \par
\textup{(b)} Any $s$\+torsion-free $s$\+contramodule $R$\+module
is $s$\+separated (and $s$\+complete).
\end{thm}

\begin{proof}
 From the proof of Theorem~\ref{s-contraadjusted-complete}
one concludes that an $R$\+module $C$ is $s$\+sepa\-rated
if and only if for any sequence $a_0$, $a_1$,
$a_2$,~\dots~$\in C$ the element $b_0\in C$ is uniquely determined
by the system of equations~\eqref{a-b-bad-equation-system}.
 In other words, for any two solutions $(b_n'\in C)_{n=0}^\infty$ and
$(b_n''\in C)_{n=0}^\infty$ of the system~\eqref{a-b-bad-equation-system}
with the given sequence $(a_n\in C)_{n=0}^\infty$, one should have
$b_0'=b_0''$.
 We will say that ``a solution of
the system~\eqref{a-b-bad-equation-system} is weakly unique''
if this condition is satisfied.

 Assuming that this is the case for all sequences
$(a_n\in C)_{n=0}^\infty$, one can see that a solution
of the system of equations~\eqref{a-b-main-equation-system} is unique
(in the conventional sense of the word) in the $R$\+module~$C$.
 Indeed, for every $n\ge0$ the system~\eqref{a-b-main-equation-system}
implies
$$
 s^i(b_{n+i}-sb_{n+i+1}) = s^ia_{n+i}
 \qquad \text{for all \ $i\ge0$},
$$
which uniquely determines the element~$b_n$.

 Now assume that a solution of~\eqref{a-b-bad-equation-system}
exists and is weakly
unique for all sequences $(a_n\in C)_{n=0}^\infty$.
 In this case, in order to solve
the system~\eqref{a-b-main-equation-system}, one for every
$n\ge0$ solves the auxiliary equation system
\begin{equation} \label{aux-equation-system}
 s^i(d^{(n)}_i-sd^{(n)}_{i+1}) = s^ia_{n+i}, \qquad i\ge0,
\end{equation}
in $d^{(n)}_i\in C$, \ $n$, $i\ge0$; and puts $b_n=d^{(n)}_0$.
 Then, for every $n\ge0$, the sequence
$e^{(n)}_i=sd^{(n+1)}_i+a_{n+i}$ provides another solution
of the system~\eqref{aux-equation-system}, because
\begin{multline*}
 s^i(e^{(n)}_i-se^{(n)}_{i+1}) = s^i(sd^{(n+1)}_i + a_{n+i} -
 s^2d^{(n+1)}_{i+1} - sa_{n+i+1}) \\
 = ss^i(d^{(n+1)}_i - sd^{(n+1)}_{i+1}) + s^i(a_{n+i}-sa_{n+i+1}) \\
 = ss^ia_{n+i+1} + s^i(a_{n+i}-sa_{n+i+1}) = s^ia_{n+i}.
\end{multline*}
 Hence, due to the weak uniqueness condition on the solutions
of~\eqref{a-b-bad-equation-system}, we have
$$
 b_n=d^{(n)}_0=e^{(n)}_0=sd^{(n+1)}_0+a_n=sb_{n+1}+a_n
$$
and a solution of the equations~\eqref{a-b-main-equation-system}
is obtained.
 This proves part~(a).

 To check~(b), it remains to recall that in an $s$\+torsion-free
$R$\+module $C$ the systems~\eqref{a-b-main-equation-system}
and~\eqref{a-b-bad-equation-system} are equivalent, so (existence
and) uniqueness of solutions of~\eqref{a-b-main-equation-system}
implies (existence and) uniquence of solutions
of~\eqref{a-b-bad-equation-system}.
\end{proof}

\begin{rem} \label{s-separ-complete-contra-easy-rem}
 The assertion of Theorem~\ref{s-separ-complete-contra}(b) is
well-known; see Corollary~\ref{flat-contra-noetherian-cor}(b)
below for a generalization to Noetherian rings.
 Theorem~\ref{s-separ-complete-contra}(a) is even more standard,
and easier proved with the standard methods.
 It suffices to say that the class of $s$\+contramodule $R$\+modules
is closed under kernels and infinite products in $R\modl$ by
Theorem~\ref{ext-0-1-orthogonal}(a); hence it is also closed under
projective limits of all diagrams.
 Any $R$\+module $D$ for which $s^nD=0$ for some $n\ge1$ is
an $s$\+contramodule (since the action of~$s$ in $\Ext^*_R(R[s^{-1}],D)$
is then simultaneously invertible and nilpotent); thus
the $R$\+module $\varprojlim_n C/s^nC$ is an $s$\+contramodule, too
(cf.\ Lemma~\ref{I-completion-s-contra} below).
 The more complicated argument above is presented for comparison
with the proof of Theorem~\ref{s-contraadjusted-complete}, and
also in order to illustrate the workings of the technique of
infinite systems of nonhomogeneous linear equations.
\end{rem}

 One can sum up the results of
Theorems~\ref{s-contraadjusted-complete}\+-\ref{s-separ-complete-contra}
in the following diagram:
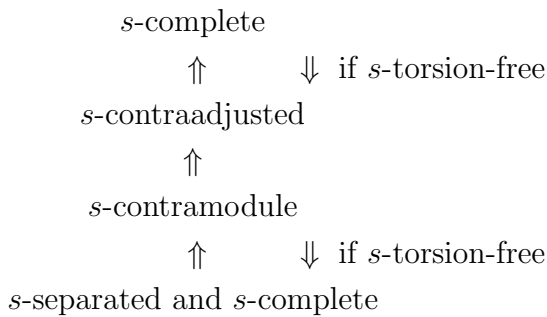
\begin{gather*}
\text{$s$-complete} \\*
\hphantom{\Downarrow\text{ if $s$-torsion-free}\qquad\quad}
\Uparrow\qquad\quad\Downarrow\text{ if $s$-torsion-free}
\\* \text{$s$-contraadjusted}
\\ \Uparrow \\ \text{$s$-contramodule} \\* 
\hphantom{\Downarrow\text{ if $s$-torsion-free}\qquad\quad}
\Uparrow\qquad\quad\Downarrow\text{ if $s$-torsion-free}
\\* \text{$s$-separated and $s$-complete}
\end{gather*}

\begin{rem}
 The two ``inner'' properties in the above diagram are defined in
terms of (the existence and/or uniqueness of solutions of)
the equation system~\eqref{a-b-main-equation-system}, while
the two ``outer'' properties translate into (the existence and/or
weak uniqueness of solutions of) the equation
system~\eqref{a-b-bad-equation-system}.
 The system~\eqref{a-b-bad-equation-system} does not look good,
though, and appears to be nothing more than an unsuccessful
na\"\i ve attempt to arrive at~\eqref{a-b-main-equation-system}.
 This is supposed to teach us that the $s$\+completeness and
$s$\+separatedness are not the right conditions to consider
unless $s$\+torsion\+freeness, or some weaker form of it guaranteeing
essential equivalence of~\eqref{a-b-main-equation-system}
and~\eqref{a-b-bad-equation-system}, is assumed.

 Indeed, it is known~\cite[Example~2.5]{Sim} that the class of
$s$\+separated and $s$\+complete $R$\+modules does \emph{not}
have good homological properties (of the kind listed
in Theorem~\ref{ext-0-1-orthogonal}(a) for the class of
$s$\+contramodules): it is not closed under cokernels (even under
the cokernels of injective morphisms) in $R\modl$, nor is it closed
under extensions.
 The category of $s$\+separated and $s$\+complete $R$\+module is
not abelian.
 Similarly, the counterexample below shows that the class of
$s$\+complete $R$\+modules is not closed under extensions
(cf.\ Lemma~\ref{s-contraadjusted-closure-properties}), even
though it is closed under quotients.

 Concerning the latter assertion, given a short exact sequence of
$R$\+modules $0\rarrow K\rarrow L\rarrow M\rarrow0$, the exact
sequence $0\rarrow K/(K\cap s^nL)\rarrow L/s^nL\rarrow M/s^nM\rarrow0$
with surjective maps between the modules in the projective system
$K/(K\cap s^nL)$ yields surjectivity of the map $\varprojlim_n L/s^nL
\rarrow\varprojlim_n M/s^nM$, so $s$\+completeness of $L$ implies
$s$\+completeness of~$M$.
\end{rem}

\begin{exs}  \label{s-contra-adjusted-counterex}
 (1) Let us start with reproducing the now-classical counterexample
of a nonseparated contramodule (see~\cite[Section~1.5]{Prev} and
the references therein).
 Set $R=\Z$ and choose a prime number~$p$.
 Let $C$ denote the subgroup in abelian group
$\prod_{n=0}^\infty\Z_p$ consisting of all sequences of $p$\+adic
integers $u_0$, $u_1$, $u_2$,~\dots, $u_n$,~\dots\ converging to zero
in the topology of~$\Z_p$.
 Let $D\subset C$ denote the group of all sequences of $p$\+adic
integers of the form $v_0$, $pv_1$, $p^2v_2$,~\dots, where
$v_n\in\Z_p$; and let $E\subset D$ be the subgroup of all sequences 
$u_n=p^nv_n$ such that $v_n\to 0$ in $\Z_p$ as $n\to\infty$.

 All the three groups $C$, $D$, $E$ are $p$\+contramodules; in fact,
they are $p$\+separated and $p$\+complete.
 Hence the quotient group $C/E$ is a $p$\+contramodule, too.
 However, it is not $p$\+separated; in fact, one has
$\bigcap_n p^n(C/E)=D/E\subset C/E$, so $\varprojlim_n (C/E)/p^n(C/E)
= C/D$.
 Furthermore, there is a short exact sequence $0\rarrow D/E\rarrow
C/E\rarrow C/D\rarrow 0$, where the groups $D/E\simeq
(\prod_{n=0}^\infty\Z_p)/C$ and $C/D\subset\prod_{n=0}^\infty\Z/p^n\Z$
are $p$\+separated and $p$\+complete, but the group $C/E$
is not~\cite{Sim}.

 Computing the cokernel of the embedding $E\rarrow C$ in
the category of $p$\+separated and $p$\+complete abelian groups
$\Ab_{p\secmp}$, one obtains the group $C/D$ and the kernel of
the morphism $C\rarrow C/D$ is $D$, while the kernel of
$E\rarrow C$ is, of course, zero, and the cokernel
of $0\rarrow E$ is~$E$.
 Thus the category $\Ab_{p\secmp}$ is not abelian.

\medskip

 (2) Now let us demonstrate an example of a $p$\+complete abelian
group that is not $p$\+contraadjusted.
 Choose a bijection $(\phi,\psi)\:\Z_{\ge0}\simeq
\Z_{\ge0}\times\Z_{\ge0}$, and assign to every eventually vanishing
sequence of $p$\+adic integers $w_0$, $w_1$, $w_2$,~\dots, $w_i$,~\dots\
the sequence $w_{\phi(0)}$, $pw_{\phi(1)}$, $p^2w_{\phi(2)}$,~\dots,
$p^nw_{\phi(n)}$,~\dots\ \
 This construction provides a homomorphism $\bigoplus_{n=0}^\infty\Z_p
\rarrow D$ whose composition with the projection
$D\rarrow D/E$ is still an injective map $\bigoplus_{n=0}^\infty\Z_p
\rarrow D/E$, because $w_{\phi(n)}\to0$ as $n\to\infty$ implies
$w_i=0$ for all~$i$.
 Denote the image of this embedding by $A\subset D/E$.
 Consider the surjective homomorphism $A\rarrow\Q_p$ taking
$(w_i)_{i=0}^\infty$ to $\sum_{i=0}^\infty\frac{w_i}{p^i}$, and extend it
to an abelian group homomorphism $f\:C/E\rarrow\Q_p$ in
an arbitrary way.
 Set $F=\ker(f)\subset C/E$, and denote by $G$ the kernel of
the composition of~$f$ with the projection $C\rarrow C/E$.

 Then from the exact sequence $0\rarrow F\rarrow C/E\rarrow \Q_p
\rarrow 0$ one concludes that $F/p^nF\simeq (C/E)/p^n(C/E)$, so
$\varprojlim_n F/p^nF=C/D$.
 The map $F\rarrow C/D$ is surjective, because the map
$D\rarrow C/G\simeq\Q_p$ is.
 Hence the group $F$ is $p$\+complete.
 Similarly, from the short exact sequence $0\rarrow G\rarrow C
\rarrow\Q_p\rarrow 0$ we get $G/p^nG\simeq C/p^nC$, so
$\varprojlim_n G/p^nG=C$ and $G$ is not $p$\+complete.
 Finally, we have a short exact sequence $0\rarrow D\rarrow G
\rarrow F\rarrow0$.
 Both the groups $D$ and $F$ are $p$\+complete, and $D$ is also
$p$\+separated.
 Thus $F$ is not $p$\+contraadjusted (for otherwise $G$ would
have to be $p$\+contraadjusted, too, but it is not even
$p$\+complete).
\end{exs}

\begin{rem} \label{reduced-separated-remark}
 In addition to Theorems~\ref{s-contraadjusted-complete}\+-%
\ref{s-separ-complete-contra}, it is instructive to compare
the two versions of the separatedness property.
 An $R$\+module is $s$\+separated if and only if it has no
\emph{elements} divisible by an arbitrary power of~$s$.
 This means weak uniqueness of solutions of the equation
system~\eqref{a-b-bad-equation-system}.
 An $R$\+module $C$ satisfies $\Hom_R(R[s^{-1}],C)=0$ if and
only if it has no $s$\+divisible \emph{submodules}.
 This is equivalent to uniqueness of solutions of the equation
system~\eqref{a-b-main-equation-system}.
 The former condition is ``na\"\i ve'' and the latter one
is its ``well-behaved version'', in that the class of
all $R$\+modules without $s$\+divisible submodules is closed
under extensions (see the proof of
Theorem~\ref{ext-0-1-orthogonal}), while the class of
$s$\+separated $R$\+modules is not (as the above
Example~\ref{s-contra-adjusted-counterex}\,(1) demonstrates).
 For $s$\+torsion-free modules, the two conditions are
equivalent.
\end{rem}

\Section{$s$-Power Infinite Summation Operations}
\label{s-power-summation-secn}

 Let $s$~be a formal symbol.
 We will say that an abelian group $C$ is endowed with
an \emph{$s$\+power infinite summation operation} if for every
sequence of elements $a_0$, $a_1$, $a_2$,~\dots~$\in C$
an element denoted formally by
$$
 \sum\nolimits_{n=0}^\infty s^na_n\in C
$$
is specified.
 The axioms of additivity
$$
 \sum\nolimits_{n=0}^\infty s^n(a_n+b_n) =
 \sum\nolimits_{n=0}^\infty s^na_n + \sum\nolimits_{n=0}^\infty s^nb_n
 \quad \text{for any $(a_n,b_n\in C)_{n=0}^\infty$},
$$
contraunitality
$$
 \sum\nolimits_{n=0}^\infty s^na_n = a_0
 \quad \text{when $a_1=a_2=a_3=\dotsb=0$ in $C$},
$$
and contraassociativity
$$
 \sum\nolimits_{i=0}^\infty
 s^i\left(\sum\nolimits_{j=0}^\infty s^ja_{ij}\right)=
 \sum\nolimits_{n=0}^\infty s^n\left(\sum\nolimits_{i+j=n} a_{ij}\right)
 \quad \text{for any $(a_{ij}\in C)_{i,j=0}^\infty$}
$$
have to be satisfied (cf.~\cite[Sections~1.3\+-1.4]{Prev}).

 Given an abelian group $C$ endowed with an $s$\+power infinite
summation operation, one defines an additive operator
(abelian group endomorphism) $s\:C\rarrow C$ by the rule
$$
 sa=\sum\nolimits_{i=0}^\infty s^ia_i, \quad
 \text{where $a_1=a$ and $a_i=0$ for $i\ne1$}.
$$
 It follows from the contraassociativity axiom that
the powers (iterated compositions) $s^n=s\circ\dotsb\circ s$ of
the endomorphism~$s$ can be obtained as
$$
 s^na=\sum\nolimits_{i=0}^\infty s^ia_i, \quad
 \text{where $a_n=a$ and $a_i=0$ for $i\ne n$}.
$$

\begin{exs}  \label{s-power-summation-examples}
 (1) For any abelian group $V$, the group of formal power series
$V[[z]]$ is naturally endowed with a $z$\+power infinite summation
operation.
 The group of $p$\+adic integers $\Z_p$ is naturally endowed with
a $p$\+power infinite summation operation.
 More generally, for any set $X$ the group $\prod_{x\in X}\Z_p$
is endowed with a $p$\+power infinite summation operation.
 The subgroup $C=\Z_p[[\Z_{\ge0}]]\subset\prod_{n=0}^\infty\Z_p$ of all
sequences of $p$\+adic integers converging to zero in the topology
of $\Z_p$ is preserved by the $p$\+power infinite sumation operation
in $\prod_{n=0}^\infty\Z_p$, hence also endowed with a $p$\+power
infinite summation operation.
 This generalizes to the case of the subgroup $\Z_p[[X]]\subset
\prod_{x\in X}\Z_p$ consisting of all the families of elements
$u_x\in\Z_p$, \ $x\in X$ such that for every $n\ge0$ one has
$u_x\in p^n\Z_p$ for all but a finite number of indices $x\in X$.
 In all these cases, the infinite sum can be computed as the limit
of finite partial sums in the $s$\+adic (i.~e., $z$\+adic or $p$\+adic,
resp.)\ topology of the group in question.

\medskip

 (2) For any two abelian groups $C$ and $D$ with $s$\+power infinite
summation operations and a group homomorphism $f\:C\rarrow D$
preserving the infinite summation operations, the groups
$\ker(f)$ and $\coker(f)$ inherit the $s$\+power infinite
summation operations of $C$ and~$D$.
 It follows that the category of abelian groups with $s$\+power
infinite summation operations is abelian.
 Furthermore, for any family of groups $C_\alpha$ with $s$\+power
infnite summation operations, there is a natural infinite
summation operation on the infinite product $\prod_\alpha C_\alpha$.

\medskip

 (3) The construction of Example~\ref{s-contra-adjusted-counterex}\,(1)
allows to show that there exists an abelian group $B$ with
an $s$\+power infinite summation operation and a sequence of elements
$b_0$, $b_1$, $b_2$,~\dots~$\in B$ such that $\sum_{n=0}^\infty s^nb_n
\ne0$ in $B$, but $s^nb_n=0$ for every $n\ge0$.
 Indeed, set $B=C/E$, and let $b_n=c_n+E$ be the coset of
the sequence $c_n=(u_0,u_1,u_2,\dots)$ with $u_0=0$, \,$u_1=0$,~\dots,
$u_n=1$, \,$u_{n+1}=0$,~\dots \
 Then $p^nc_n\in E$, but $\sum_{n=0}^\infty p^nc_n\notin E$, because
the sequence $u_0=1$, $u_1=p$, $u_2=p^2$,~\dots, $u_n=p^n$,~\dots\
does not have the form $u_n=p^nv_n$ with $v_n\to0$ in $\Z_p$
as $n\to\infty$.
 This counterexample shows that the $s$\+power infinite summation
operations, generally speaking, cannot be interpreted as any kind
of limit of finite partial sums.
\end{exs}

\begin{lem} \label{infsummation-nodivisible}
 Let $C$ be an abelian group endowed with an $s$\+power infinite
summation operation and $D\subset C$ be a subgroup for which $sD=D$.
 Then $D=0$.
\end{lem}

\begin{proof}
 Let $a_0$, $a_1$, $a_2$~\dots\ be a sequence of elements
in $C$ satisfying $a_n=sa_{n+1}$ for every $n\ge0$.
 Our aim is to show that $a_n=0$ for all $n\ge0$.
 The idea is to consider the element $\sum_{n=0}^\infty s^na_n$ and
perform the transformations
$$
 \sum\nolimits_{n=0}^\infty s^na_n=\sum\nolimits_{n=0}^\infty s^nsa_{n+1}=
 \sum\nolimits_{n=0}^\infty s^{n+1}a_{n+1}=\sum\nolimits_{n=1}^\infty s^na_n,
$$
implying $a_0=0$ (which is clearly sufficient).
 To do it more rigorously, consider the double-indexed array of
elements $a_{ij}=a_{i+1}$ for $i\ge0$, $j=1$ and $a_{ij}=0$ for other
values of $i$, $j$, and apply the contraassociativity axiom,
$$
 \sum\nolimits_{i=0}^\infty s^ia_i =
 \sum\nolimits_{i=0}^\infty s^i\left
 (\sum\nolimits_{j=0}^\infty s^ja_{ij}\right)=
 \sum\nolimits_{n=0}^\infty s^n\left(\sum\nolimits_{i+j=n} a_{ij}\right)
 = \sum\nolimits_{n=0}^\infty s^n a_n',
$$
where $a_0'=0$ and $a_n'=a_n$ for $n\ge1$.
 Using the additivity and contraunitality axioms allows to deduce
the desired equation $a_0=0$.
\end{proof}

\begin{thm} \label{s-power-contra-equivalence}
\textup{(a)} An $s$\+power infinite summation operation on
an abelian group $C$ is uniquely determined by the endomorphism
$s\:C\rarrow C$.
 In other words, given an abelian group $C$ with an additive operator
$s\:C\rarrow C$, there exists at most one $s$\+power infinite
summation operation structure on $C$ restricting to the prescribed
action of the operator~$s$ in~$C$. \par
\textup{(b)} Given two abelian groups $C$ and $D$ endowed with
$s$\+power infinite summation operations, an abelian group
homomorphism $f\:C\rarrow D$ preserves the $s$\+power infinite
summation operations if and only if it commutes with
the endomorphisms~$s$ on $C$ and $D$, i.~e., $fs=sf$. \par
\textup{(c)} An endomorphism $s\:C\rarrow C$ of an abelian group $C$
can be extended to an $s$\+power infinite summation operation on $C$
if and only if one has\/ $\Hom_{\Z[s]}(\Z[s,s^{-1}],C)\allowbreak
=0=\Ext^1_{\Z[s]}(\Z[s,s^{-1}],C)$.
\end{thm}

 In other words, the theorem says that the category of abelian
groups with $s$\+power infinite summation operations is equivalent
(if one wishes, even isomorphic) to the category of $s$\+contramodule
$\Z[s]$\+modules (cf.~\cite[Lemma~B.5.1]{Pweak}).

\begin{proof}
 Let $C$ be an abelian group endowed with an $s$\+power infinite
summation operation.
 We already know from Lemma~\ref{infsummation-nodivisible} that
$\Hom_{\Z[s]}(\Z[s,s^{-1}],C)=0$; let us check that
$\Ext_{\Z[s]}^1(\Z[s,s^{-1}],C)=0$.
 According to Lemma~\ref{ext-r-s-minus-1-computed}(a), we have
to check that for any sequence of elements $a_0$, $a_1$,
$a_2$,~\dots~$\in C$ the system of
equations~\eqref{a-b-main-equation-system} can be solved in~$C$.
 Put
\begin{equation} \label{equation-solving-by-summation}
 b_n=\sum\nolimits_{i=0}^\infty s^ia_{n+i}.
\end{equation}
 We claim that~\eqref{a-b-main-equation-system} is satisfied.
 Indeed,
\begin{align*}
 b_n-sb_{n+1} &= \sum\nolimits_{i=0}^\infty s^ia_{n+i} -
 s\sum\nolimits_{i=0}^\infty s^ia_{n+1+i} \\ &=
 \sum\nolimits_{i=0}^\infty s^ia_{n+i} -
 \sum\nolimits_{i=0}^\infty s^{i+1}a_{n+i+1} = 
 \sum\nolimits_{i=0}^\infty s^ia_{n+i} -
 \sum\nolimits_{i=1}^\infty s^ia_{n+i} = a_n.
\end{align*}
 To make a rigorous argument out of this calculation, one can,
similarly to the above proof of Lemma~\ref{infsummation-nodivisible},
apply the contraassociativity axiom to the array of elements
$a_{ij}=a_{n+j+1}$ for $i=1$, $j\ge0$ and $a_{ij}=0$ for other values
of $i$ and~$j$.

 Conversely, let $C$ be an $s$\+contramodule $\Z[s]$\+module; so
the system of equations~\eqref{a-b-main-equation-system} is
uniquely solvable in $C$ for any sequence of elements
$(a_n\in C)_{n=0}^\infty$.
 Given such a sequence, solve
the system~\eqref{a-b-main-equation-system} and put
\begin{equation} \label{summation-by-equation-solving}
 \sum\nolimits_{n=0}^\infty s^na_n = b_0.
\end{equation}
 The additivity and contraunitality axioms being pretty straighforward,
let us check the contraassociativity axiom for the infinite summation
operation so defined.
 For this purpose, one has to compute the sum over the area $i+j\ge n$,
\ $i$, $j\ge0$ as the sum of the sums over the rows.
 Let us start with solving the equations
$$
 c_{i,m}-sc_{i,m+1}=a_{i,m}, \qquad c_{i,m}\in C, \quad i,\,m\ge0;
$$
so, according to our definition, $c_{i,m}=\sum_{j=0}^\infty s^ja_{i,m+j}$
and, in particular, $c_{i,0}=\sum_{j=0}^\infty s^ja_{ij}$.
 Furthermore, solve the equations
$$
 d_n-sd_{n+1}=c_{n,0}, \qquad d_n\in C, \quad n\ge0;
$$
so $d_n=\sum_{i=0}^\infty s^ic_{n+i,0}$ and, in particular,
$d_0=\sum_{i=0}^\infty s^i\bigl(\sum_{j=0}^\infty s^ja_{ij}\bigr)$.
 Finally, set $e_n=c_{0,n}+c_{1,n-1}+\dotsb+c_{n-1,1}+d_n$; this is our
sum over the area $i+j\ge n$, \ $i$, $j\ge0$.
 In particular, by the definition, we have $e_0=d_0$.
 On the other hand,
\begin{setlength}{\multlinegap}{0pt}
\begin{multline*}
 e_n-se_{n+1} \\ = (c_{0,n}-sc_{0,n+1})+(c_{1,n-1}-sc_{1,n})+
 \dotsb+(c_{n-1,1}-sc_{n-1,2}) - sc_{n,1}
 + (d_n-sd_{n+1}) \\
 = a_{0,n} + a_{1,n-1} + \dotsb + a_{n-1,1} -sc_{n,1} + c_{n,0} 
 \\ = a_{0,n} + a_{1,n-1} + \dotsb + a_{n-1,1} + a_{n,0},
\end{multline*}
hence $e_0=\sum_{n=0}^\infty s^n\bigl(\sum_{i+j=n}a_{ij}\bigr)$
and we are done.
\end{setlength}

 We have shown that the infinite system of nonhomogeneous linear
equations~\eqref{a-b-main-equation-system} is solvable
by~\eqref{equation-solving-by-summation} in any abelian group $C$
with an $s$\+power infinite summation operation.
 Moreover, the system~\eqref{a-b-main-equation-system} is uniquely
solvable in $C$, as the related system of homogeneous linear
equations has no nonzero solutions according to (the proof of)
Lemma~\ref{infsummation-nodivisible}.
 Furthermore, solving the equations~\eqref{a-b-main-equation-system}
allows to recover the $s$\+power infinite summation operation in $C$
by the rule~\eqref{summation-by-equation-solving}.

 Therefore, the infinite summation operation structure with
the prescribed map $s\:C\rarrow C$ is unique when it exists.
 Furthermore, any abelian group homomorphism $f\:C\rarrow D$
commuting with the operators~$s$ takes solutions of
the system~\eqref{a-b-main-equation-system} in $C$ to similar solutions
in $D$, and consequently preserves the infinite summation
operations.
 All the assertions of the theorem are now proved.
\end{proof}

\Section{$[s,t]$-Power Infinite Summation Operations}
\label{[st]-power-summation-secn}

 Let us now have two formal symbols~$s$ and~$t$.
 We say that an abelian group $C$ is endowed with
an \emph{$[s,t]$\+power infinite summation operation} if for
every array of elements $a_{mn}\in C$, \ $m$, $n\ge0$,
an element denoted formally by
$$
 \sum_{n=0}^\infty s^mt^n a_{mn}\in C
$$
is defined.
 The axioms of additivity
$$
 \sum_{m,n=0}^\infty s^mt^n(a_{mn}+b_{mn}) =
 \sum_{m,n=0}^\infty s^mt^na_{mn} +
 \sum_{m,n=0}^\infty s^mt^nb_{mn}
 \quad \text{for any $(a_{mn},b_{mn}\in C)_{m,n=0}^\infty$},
$$
contraunitality
$$
 \sum_{m,n=0}^\infty s^mt^na_{mn} = a_{00}
 \quad \text{whenever $a_{mn}=0$ in $C$ for all $(m,n)\ne(0,0)$},
$$
and contraassociativity
$$
 \sum_{i,j=0}^\infty s^it^j
 \left(\sum_{k,l=0}^\infty s^kt^la_{ij,kl}\right)=
 \sum_{m,n=0}^\infty s^mt^n\left(\sum_{i+k=m}^{j+l=n} a_{ij,kl}\right)
 \quad \text{for any $(a_{ij,kl}\in C)_{i,j,k,l=0}^\infty$}
$$
are imposed.

\begin{thm} \label{[st]-power-summation-equivalence}
 The category of abelian groups $C$ endowed with an $[s,t]$\+power
infinite summation operation is isomorphic to the category of
abelian groups $C$ endowed with a pair of commuting endomorphisms
$s$, $t\:C\rarrow C$, \ $st=ts$, such that both the $s$\+power
and the $t$\+power infinite summation operations exist in~$C$.
\end{thm}

 In other words, an abelian group $C$ with an $[s,t]$\+power
infinite summation operation is the same thing as
a $\Z[s,t]$\+module satisfying
$$
 \Hom_{\Z[s,t]}(\Z[s,s^{-1},t]\oplus\Z[s,t,t^{-1}],\>P) = 0
 = \Ext^1_{\Z[s,t]}(\Z[s,s^{-1},t]\oplus\Z[s,t,t^{-1}],\>P)
$$
(cf.~\cite[Lemma~B.6.1 and Theorem~B.1.1]{Pweak}).

\begin{proof}
 Given an abelian group $C$ with an $[s,t]$\+power infinite summation
operation, one defines an $s$\+power infinite infinite summation
operation and a $t$\+power infinite summation operation on $C$ by
the obvious rules
\begin{equation} \label{forget-st-to-s-and-t}
\begin{alignedat}{2}
 &\sum\nolimits_{m=0}^\infty s^ma_m &&=
 \sum\nolimits_{m,n=0}^\infty s^mt^na_{m,n} \quad
 \text{if $a_{m,0}=a_m$ and $a_{m,n}=0$ for $n>0$}, \\
 &\sum\nolimits_{n=0}^\infty t^na_n &&=
 \sum\nolimits_{m,n=0}^\infty s^mt^na_{m,n} \quad
 \text{if $a_{0,n}=a_n$ and $a_{m,n}=0$ for $m>0$}.
\end{alignedat}
\end{equation}
 Specializing further to arrays $(a_{m,n})_{m,n=0}^\infty$ with the only
nonzero component $a_{1,0}$ or $a_{0,1}$, one obtains the additive
operators $s\:C\rarrow C$ and $t\:C\rarrow C$ (as explained in
Section~\ref{s-power-summation-secn}).
 Applying the contraassociativity axiom to the arrays $(a_{ij,kl})$
with the only nonzero component $a_{1,0,0,1}$ or $a_{0,1,1,0}$, one
proves that $st=ts$.

 Conversely, suppose that $s$\+power infinite summation operations
and $t$\+power infinite summation operations are defined in $C$,
and the operators $s$ and $t$ in $C$ commute.
 Then one can define the $[s,t]$\+power infinite summation operation
on $C$ as
\begin{equation} \label{recover-st-from-s-and-t}
 \sum\nolimits_{m,n=0}^\infty s^mt^na_{m,n} =
 \sum\nolimits_{m=0}^\infty s^m
 \left(\sum\nolimits_{n=0}^\infty t^n a_{m,n}\right).
\end{equation}
 This obviously satisfies the additivity and contraunitality axioms.
 Checking the contraassociativity axiom,
\begin{setlength}{\multlinegap}{0pt}
\begin{multline*}
 \sum\nolimits_{i,j=0}^\infty s^it^j
 \left(\sum\nolimits_{k,l=0}^\infty s^kt^la_{ij,kl}\right) =
 \sum\nolimits_{i=0}^\infty s^i\left(\sum\nolimits_{j=0}^\infty t^j\left(
 \sum\nolimits_{k=0}^\infty s^k\left(\sum\nolimits_{l=0}^\infty t^la_{ij,kl}
 \right)\right)\right) \\
 = \sum_{i=0}^\infty s^i\left(\sum\nolimits_{k=0}^\infty s^k\left(
 \sum\nolimits_{j=0}^\infty t^j\left(\sum\nolimits_{l=0}^\infty t^la_{ij,kl}
 \right)\right)\right) \\
 = \sum\nolimits_{m=0}^\infty s^m\left(\sum\nolimits_{i+k=m}
 \left(\sum\nolimits_{n=0}^\infty t^n\left(\sum\nolimits_{j+l=n}a_{ij,kl}
 \right)\right)\right) \\
 = \sum\nolimits_{m=0}^\infty s^m\left(\sum\nolimits_{n=0}^\infty t^n\left(
 \sum\nolimits_{i+k=m}^{j+l=n}a_{ij,kl}\right)\right)
 = \sum\nolimits_{m,n=0}^\infty s^mt^n\left(
 \sum\nolimits_{i+k=m}^{j+l=n}a_{ij,kl}\right),
\end{multline*}
reduces to showing that the two infinite summation operations commute
with each other, that is
\begin{equation} \label{inf-summations-commute}
 \sum\nolimits_{j=0}^\infty t^j
 \left(\sum\nolimits_{k=0}^\infty s^k a_{jk}\right) =
 \sum\nolimits_{k=0}^\infty s^k
 \left(\sum\nolimits_{j=0}^\infty t^j a_{jk}\right)
 \quad \text{for any $(a_{jk}\in C)_{j,k=0}^\infty$}.
\end{equation}

 To prove~\eqref{inf-summations-commute}, one can do a computation
with infinite systems of nonhomogeneous linear equations similar
to the one in the proof of Theorem~\ref{s-power-contra-equivalence},
based on considering the sums over the areas $j\ge n$, \ $k\ge m$.
 Alternatively, there is the following conceptual argument based on
the assertion of Theorem~\ref{s-power-contra-equivalence}(b).
 The $t$\+power infinite summation is a group homomorphism
$$
 \sum\nolimits_t=\left((a_j)_{j=0}^\infty\mapsto
 \sum\nolimits_{j=0}^\infty t^ja_j\right)\:
 \prod\nolimits_{j=0}^\infty C\lrarrow C.
$$
 There is a natural (termwise) $s$\+power infinite summation operation
in the group $\prod_{j=0}^\infty C$ (see
Example~\ref{s-power-summation-examples}\,(2)).
 The equation~\eqref{inf-summations-commute} says that $\sum_t$
is a morphism of groups with $s$\+power infinite summation
operations.
 According to Theorem~\ref{s-power-contra-equivalence}(b), it
suffices to show that $\sum_t$ commutes with the operators~$s$ in
$\prod_{j=0}^\infty C$ and $C$, which means the equation
\begin{equation} \label{t-summation-commutes-with-s}
 \sum\nolimits_{j=0}^\infty t^j(sa_j) =
 s \sum\nolimits_{j=0}^\infty t^j a_j
 \quad \text{for any $(a_j\in C)_{j=0}^\infty$}.
\end{equation}
 Now, the equation~\eqref{t-summation-commutes-with-s} is also
equivalent to the assertion that $s\:C\rarrow C$ is a morphism
of groups with $t$\+power infinite summation operations.
 Applying Theorem~\ref{s-power-contra-equivalence}(b) again,
we finally conclude that this follows from $st=ts$.
\end{setlength}

 We have constructed functors in both directions between the two
categories in question; it remains to check that their compositions
are the identity functors.
 Here the essential part is to check that, for an abelian group $C$
with an $[s,t]$\+power infinite summation operation and
its restrictions to an $s$\+power and a $t$\+power infinite
summation operations provided by~\eqref{forget-st-to-s-and-t},
the equation~\eqref{recover-st-from-s-and-t} holds.
 For this purpose, it suffices to apply the contraassociativity
axiom to the array $(a_{ij,kl})_{i,j,k,l=0}^\infty$ with
$a_{m,0,0,n}=a_{mn}$ for all $m$, $n\ge0$ and
$a_{ij,kl}=0$ when $j>0$ or $k>0$.
\end{proof}

 The following lemma, extending Lemma~\ref{infsummation-nodivisible}
to the case of two variables, belongs to the class of results known
as the ``contramodule Nakayama lemma''.
 A number of other versions can be found in
the literature~\cite[Corollary~0.3]{PSY2},
\cite[Lemma~A.2.1]{Psemi}, \cite[Lemma~1.3.1]{Pweak},
\cite[Lemma~D.1.2]{Pcosh}, \cite[Lemma~6.14]{PR}.

\begin{lem} \label{nakayama}
 Let $C$ be an abelian group endowed with an $[s,t]$\+power infinite
summation operation and $D\subset C$ a subgroup such that $sD+tD=D$.
 Then $D=0$.
\end{lem}

\begin{proof}
 Let $a_0\in D$ be an element.
 Choose a pair of elements $b'_0$ and $b''_0\in D$ such that
$a_0=sb'_0+tb''_0$.
 Set $a_{1,0}=b'_0$ and $a_{0,1}=b''_0$.
 Choose two pairs of elements $b'_{1,0}$, $b''_{1,0}$, $b'_{0,1}$,
$b''_{0,1}\in D$ such that $a_{1,0}=sb'_{1,0}+tb''_{1,0}$ and
$a_{0,1}=sb'_{0,1}+tb''_{0,1}$.
 Set $a_{2,0}=b'_{1,0}$, \ $a_{1,1}=b''_{1,0}+b'_{0,1}$, 
and $a_{0,2}=b''_{0,1}$, etc.
 Proceeding by induction in~$n\ge0$, we choose for every
$i$, $j\ge0$, \ $i+j=n$, a pair of elements $b'_{i,j}$ and
$b''_{i,j}\in D$ such that
$$
 a_{ij}=sb'_{ij}+tb''_{ij}.
$$
 Then we set $b''_{n+1,-1}=b'_{-1,n+1}=0$ and
$$
 a_{i,j}=b'_{i-1,j}+b''_{i,j-1}
 \quad\text{for all $i$, $j\ge0$, \,$i+j=n+1$}.
$$
 Now we have
\begin{multline*}
 \sum_{i,j=0}^\infty s^it^ja_{ij} = 
 \sum_{i,j=0}^\infty s^it^j(sb'_{ij}+tb''_{ij}) =
 \sum_{i\ge1,j\ge0} s^it^jb'_{i-1,j} +
 \sum_{i\ge0,j\ge1} s^it^jb''_{i,j-1} \\
 = \sum_{i\ge0,j\ge0}^{i+j>0} s^it^j(b'_{i-1,j}+b''_{i,j-1})
 = \sum_{i\ge0,j\ge0}^{i+j>0} s^it^ja_{ij},
\end{multline*}
hence $a_0=0$.
\end{proof}

\begin{rem} \label{many-variables-power-summations}
 Similarly to the above exposition in this section, one can define
the notion of an $[s_1,\dotsc,s_m]$\+power infinite summation
operation in an abelian group $C$ for any $m\ge1$.
 Then one can prove that the category of abelian groups with
$[s_1,\dotsc,s_m]$\+power infinite summation operations is
isomorphic to the category of abelian groups $C$ with $m$~pairwise
commuting endomorphisms $s_j\:C\rarrow C$, \ $1\le j\le m$ such
that for every~$j$ there exists an $s_j$\+power infinite summation
operation structure on $C$ restricting to the prescribed
endomorphism~$s_j$.
 Furthermore, if $C$ is an abelian group with
an $[s_1,\dotsc,s_m]$\+power infinite summation operation and
$D\subset C$ a subgroup for which $D\subset s_1D+\dotsb+s_mD$,
then $D=0$.
\end{rem}

\Section{$[s_1,\dotsc,s_m]$-Contraadjustedness, $I$-Contramoduleness,
\\* and $I$-Adic Completeness}

 One of the aims of this paper is to discuss the following theorem.

\begin{thm} \label{ideal-contramodule-thm}
 Let $R$ be a commutative ring, $I\subset R$ an ideal generated by
some elements $s_j\in R$, and $C$ an $R$\+module.
 Assume that $C$ is an $s_j$\+contramodule for every~$j$.
 Then $C$ is an $s$\+contramodule for every $s\in I$.
\end{thm}

\begin{rem}
 Notice that the analogue of the assertion of
Theorem~\ref{ideal-contramodule-thm} is \emph{not} true for
$s$\+contraadjusted modules.
 E.~g., any $R$\+module is $1$\+contraadjusted, but it does not
have to be $s$\+contraadjusted for other elements $s\in R$.
 Or, to give another example, the abelian group $\Z[\frac{1}{2},
\frac{1}{3}]$ is $2$\+contraadjusted and $3$\+contraadjusted,
but not $5$\+contraadjusted (as one can readily check using
Theorem~\ref{s-contraadjusted-complete}).
\end{rem}

 Two proofs of Theorem~\ref{ideal-contramodule-thm} are given
in this paper (cf.~\cite[Section~2]{Pmgm}).
 One of them, based on the technique of infinite summation operations
developed in
Sections~\ref{s-power-summation-secn}\+-\ref{[st]-power-summation-secn},
is presented immediately below.
 The other one is postponed to
Sections~\ref{functor-delta-s-first-secn}\+-%
\ref{functor-delta-I-second-secn},
because it uses an explicit construction of the functor $\Delta$
left adjoint to the embedding of the category of contramodules
into $R\modl$, which will be introduced there.

 Yet another proof can be found in~\cite[Theorem~5 and
Lemma~7\,(1)]{Yek}).

\begin{proof}[First proof of Theorem~\ref{ideal-contramodule-thm}]
 Given a commutative ring $R$ and an $R$\+module $C$, denote by
$I_C$ the set of all elements $s\in R$ for which $C$ is
an $s$\+contramodule.
 We will show that $I_C$ is an ideal in~$R$.

\begin{lem} \label{rs-infinite-summation}
 Let $C$ be an abelian group and $r$, $s\:C\rarrow C$ be two
commuting endomorphisms of~$C$.
 Assume that $C$ admits an $s$\+power infinite summation operation.
 Then an $(rs)$\+power infinite summation operation also exists in~$C$.
\end{lem}

\begin{proof}
 Set
$$
 \sum\nolimits_{n=0}^\infty (rs)^n a_n =
 \sum\nolimits_{n=0}^\infty s^n(r^na_n)
 \quad\text{for any $(a_n\in C)_{n=0}^\infty$}.
$$
 One will have to use Theorem~\ref{s-power-contra-equivalence}(b)
(the commutativity of $r$ with the $s$\+power infinite summation
operation) in order to check the contraassociativity axiom for
the $(rs)$\+power infinite summation operation so defined.
\end{proof}

\begin{lem} \label{s+t-infinite-summation}
 Let $C$ be an abelian group and $s$, $t\:C\rarrow C$ be two
commuting endomorphisms of it.
 Assume that $C$ admits an $s$\+power infinite summation operation
and a $t$\+power infinite summation operation.
 Then $C$ also admits an $(s+t)$\+power infinite summation operation.
\end{lem}

\begin{proof}
 According to Theorem~\ref{[st]-power-summation-equivalence},
there is an $[s,t]$\+power infinite summation operation in~$C$.
 So we can put
$$
 \sum\nolimits_{n=0}^\infty (s+t)^na_n =
 \sum\nolimits_{i,j=0}^\infty s^it^j
 \left({\textstyle\binom{i+j}{i}a_{i+j}}\right)
 \quad\text{for any $(a_n\in C)_{n=0}^\infty$}.
$$
 We leave it to the reader to check the axioms.
\end{proof}

 In view of Theorem~\ref{s-power-contra-equivalence}(c), it follows
from Lemmas~\ref{rs-infinite-summation}\+-\ref{s+t-infinite-summation}
that $I_C$ is an ideal in~$R$.
 The first proof of Theorem~\ref{ideal-contramodule-thm} is finished.
\end{proof}

\begin{rem} \label{I-C-radical-ideal}
 Notice that the property of an $R$\+module $C$ to be
an $s$\+contramodule for some element $s\in R$ does not depend on
the $R$\+module structure on $C$, but only on the abelian group $C$
with the endomorphism~$s$, as it is clear from
Lemma~\ref{ext-r-s-minus-1-computed}.
 So we can choose the ring $R$ as we find convenient.
 Furthermore, for any $R$\+module $C$ the ideal $I_C\subset R$ is
a \emph{radical} ideal in $R$, that is, for any $s\in R$ and $n\ge1$
such that $s^n\in I_C$ one has $s\in I_C$.
 In other words, if $s\:C\rarrow C$ is an endomorphism of
an abelian group $C$ such that $C$ admits an $s^n$\+power
infinite summation operation, then $C$ also admits an $s$\+power
infinite summation operation.
 Indeed, the two related localizations of the ring $R$ coincide,
$R[s^{-1}]=R[(s^n)^{-1}]$, so the $s$\+contramodule and
$s^n$\+contramodule properties are equivalent.
\end{rem}

 Let $R$ be a commutative ring and $I$ an ideal in~$R$.
 We denote the $I$\+adic completion functor by
$$
 C\longmapsto\Lambda_I(C)=\varprojlim\nolimits_{n\ge1}C/I^nC.
$$
 An $R$\+module $C$ is called \emph{$I$\+adically complete} if
the natural map
$$
 \lambda_{I,C}\:C\lrarrow\Lambda_I(C)
$$
is surjective.
 The $R$\+module $C$ is called \emph{$I$\+adically separated} if
the map $\lambda_{I,C}$ is injective, i.~e., if the intersection
$\ker(\lambda_{I,C})=\bigcap_{n\ge1}I^nC$ vanishes.

 The following result generalizes
Theorem~\ref{s-contraadjusted-complete}(a) (by replacing a principal
ideal with a finitely generated one) and improves
upon~\cite[Theorem~10]{Yek} (by removing the irrelevant separatedness
assumption).

\begin{thm} \label{contraadjusted-for-ideal-generators-complete}
 Let $R$ be a commutative ring and $I\subset R$ be the ideal
generated by a finite set of elements $s_1$,~\dots, $s_m\in R$.
 Assume that an $R$\+module $C$ is $s_j$\+contraadjusted for every
$j=1$,~\dots,~$m$.
 Then the $R$\+module $C$ is $I$\+adically complete.
\end{thm}

\begin{proof}
 The idea is that $I$-adic completeness can be thought of in
the language of $[s_1,\dotsc,s_m]$\+power infinite sums and such
infinite sums obtained in terms of solutions of
the equations~\eqref{a-b-main-equation-system}, even if those
solutions are not unique and the infinite sums accordingly only
ambiguously defined.
 Specifically, let $(c_n\in C)_{n\ge1}$ be a sequence of elements
satisfying $c_{n+1}-c_n\in I^nC$ for all $n\ge1$.
 Then there exist elements $a_{n_1,\dotsc,n_m}\in C$, \
$n_1$,~\dots, $n_m\ge0$, \ $n_1+\dots+n_m\ge1$ such that
$$
 c_{n+1}-c_n =\sum\nolimits_{n_1\ge0,\dotsc,n_m\ge0}^{n_1+\dotsb+n_m=n}
 s_1^{n_1}\dotsm s_m^{n_m} a_{n_1,\dotsc,n_m}
 \quad\text{for all $n\ge1$}.
$$
 Set $a_{0,\dotsc,0}=c_1$.
 Applying Lemma~\ref{ext-r-s-minus-1-computed}(a) to the element
$s_m\in R$ and the $R$\+module $C$, for every
$n_1$,~\dots, $n_{m-1}\ge0$ choose elements
$b_{n_1,\dotsc,n_{m-1};k}^{(1)}\in C$, \ $k\ge0$ such that
$$
 b_{n_1,\dotsc,n_{m-1};k}^{(1)}-sb_{n_1,\dotsc,n_{m-1};k+1}^{(1)} = 
 a_{n_1,\dotsc,n_{m-1},k}
 \quad\text{for all $n_1$,~\dots, $n_{m-1}\ge0$, $k\ge0$}.
$$
 Proceeding by decreasing induction in~$j\le m-1$ and applying
Lemma~\ref{ext-r-s-minus-1-computed}(a) to $s_j\in R$ and
the module $C$, for every $n_1$,~\dots,~$n_{j-1}\ge0$ choose elements
$b_{n_1,\dotsc,n_{j-1};k}^{(m-j+1)}\in C$, \ $k\ge0$ such that
$$
 b_{n_1,\dotsc,n_{j-1};k}^{(m-j+1)}-sb_{n_1,\dotsc,n_{j-1};k+1}^{(m-j+1)} = 
 b_{n_1,\dotsc,n_{m-1},k;0}^{(m-j)}
 \quad\text{for all $n_1$,~\dots, $n_{j-1}\ge0$, $k\ge0$}.
$$
 Eventually, for $j=2$ we will obtain
$$
 b_{n_1;k}^{(m-1)}-sb_{n_1;k+1}^{(m-1)} = 
 b_{n_1,k;0}^{(m-2)}
 \quad\text{for all $n_1\ge0$, $k\ge0$},
$$
and for $j=1$,
$$
 b_k^{(m)}-sb_{k+1}^{(m)} =  b_{k;0}^{(m-1)}
 \quad\text{for all $k\ge0$}.
$$

 Set $b=b_0^{(m)}$.
 We claim that $b-c_n\in I^nC$ for all $n\ge1$; so the element $b\in C$
is a preimage of the element $c\in\varprojlim_n C/I^nC$ represented
by $(c_n)_{n\ge1}$ under the natural map $\lambda_{I,C}\:C\rarrow
\varprojlim_nC/I^nC$.
 Indeed,
\begin{align*}
 b_0^{(m)} &= b_{0;0}^{(m-1)}+sb_1^{(m)} =
 b_{0;0}^{(m-1)}+sb_{1;0}^{(m-1)} + s^2b_2^{(m)} = \dotsb = 
 \sum_{n_1=0}^{n-1}s^{n_1}b_{n_1;0}^{(m-1)} + s^nb_n^{(m)} \\
 &= \sum_{n_1=0}^{n-1}s^{n_1}\left(\sum_{n_2=0}^{n-n_1-1}
 s^{n_2}b_{n_1,n_2;0}^{(m-2)}\right) +
 s^n\sum_{n_1=0}^{n-1}b_{n_1;n-n_1}^{(m-1)}
 + s^nb_n^{(m)} = \dotsb \\
 &= \sum_{n_1\ge0,\dotsc,n_{m-1}\ge0}^{n_1+\dotsb+n_{m-1}<n}
 s^{n_1+\dotsb+n_{m-1}}b_{n_1,\dotsc,n_{m-1};0}^{(1)} +
 s^n \sum_{j=2}^m\left(
 \sum_{n_1\ge0,\dotsc,n_{m-j}\ge0,k\ge1}^{n_1+\dotsc+n_{m-j}+k=n}
 b_{n_1,\dotsc,n_{m-j};k}^{(j)}\right) \\
 &= \sum_{n_1\ge0,\dotsc,n_m\ge0}^{n_1+\dotsb+n_m<n}
 s^{n_1+\dotsb+n_m}a_{n_1,\dotsc,n_m}+
 s^n \sum_{j=1}^m\left(
 \sum_{n_1\ge0,\dotsc,n_{m-j}\ge0,k\ge1}^{n_1+\dotsc+n_{m-j}+k=n}
 b_{n_1,\dotsc,n_{m-j};k}^{(j)}\right) \\
 &= c_n + s^n \sum_{j=1}^m\left(
 \sum_{n_1\ge0,\dotsc,n_{m-j}\ge0,k\ge1}^{n_1+\dotsc+n_{m-j}+k=n}
 b_{n_1,\dotsc,n_{m-j};k}^{(j)}\right).
\end{align*}

 To explain the last step without long formulas, one can argue as
follows.
 Reducing solutions of the equation
systems~\eqref{a-b-main-equation-system} in $C$
modulo $I^nC$, one obtains solutions of the same equation
systems~\eqref{a-b-main-equation-system} in $C/I^nC$.
 For any $s\in I$, the equation
systems~\eqref{a-b-main-equation-system} are uniquely solvable in
$C/I^nC$ (e.~g., because $C/I^nC$ is $s$\+separated and
$s$\+complete, so Theorem~\ref{s-separ-complete-contra}(a) applies;
or for the reasons explained in
Remark~\ref{s-separ-complete-contra-easy-rem}).
 Furthermore, the solutions of~\eqref{a-b-main-equation-system}
in $C/I^nC$ can be expressed in
the form~\eqref{equation-solving-by-summation}
(where the sum is actually finite, as the action of~$s$
is nilpotent).
 Therefore, we have
\begin{align*}
 b_0^{(m)} &\equiv \sum\nolimits_{n_1=0}^{n-1}s^{n_1}b_{n_1;0}^{(m-1)}
 \equiv\sum\nolimits_{n_1,n_2=0}^{n-1}s^{n_1+n_2}b_{n_1,n_2;0}^{(m-2)}
 \equiv\dotsb \\ 
 &\equiv \sum\nolimits_{n_1,\dotsc,n_{m-1}=0}^{n-1}
 s^{n_1+\dotsb+n_{m-1}}b_{n_1,\dotsc,n_{m-1};0}^{(1)}\equiv
 \sum\nolimits_{n_1,\dotsc,n_m=0}^{n-1}
 s^{n_1+\dotsb+n_m}a_{n_1,\dotsc,n_m} \\
 &\equiv c_n \pmod {I^nC}.  \qedhere
\end{align*}
\end{proof}

 For a stronger version of
Theorem~\ref{contraadjusted-for-ideal-generators-complete},
see Remark~\ref{contraadj-for-ideal-gens-adjunction-surjective}
below.

\begin{lem} \label{I-completion-s-contra}
 Let $R$ be a commutative ring, $I\subset R$ an ideal, and
$C$ an $R$\+module.
 Then the $R$\+module $\Lambda_I(C)=\varprojlim_{n\ge1} C/I^nC$ is
an $s$\+contramodule for every element $s\in I$.
 In particular, any $I$\+adically separated and complete $R$\+module
is an $s$\+contramodule.
\end{lem}

\begin{proof}
 Any $R$\+module in which $s$~acts nilpotently is
an $s$\+contramodule, and the projective limit of any diagram
of $s$\+contramodule $R$\+modules is an $s$\+contramodule
$R$\+module (see Remark~\ref{s-separ-complete-contra-easy-rem}).
 Alternatively, one can define an $s$\+power infinite summation
operation in $\varprojlim_n C/I^nC$ (see
Theorem~\ref{s-power-contra-equivalence}(c)) as the limit of
finite partial sums in the topology of the projective limit (of
discrete groups $C/I^nC$) on $\varprojlim_n C/I^nC$, where
the kernels of the projection maps $\varprojlim_i C/I^iC\rarrow
C/I^nC$ form a base of neighborhoods of zero.
\end{proof}

 Let us denote the full subcategory of $I$\+adically separated and
complete $R$\+modules by $R\modl_{I\secmp}\subset R\modl$.

\begin{thm} \label{lambda-adjoint}
 Let $R$ be a commutative ring and $I\subset R$ a finitely
generated ideal.
 Then the functor of $I$\+adic completion\/ $\Lambda_I\:R\modl\rarrow
R\modl_{I\secmp}$ is left adjoint to the fully faithful embedding
functor $R\modl_{I\secmp}\rarrow R\modl$.
\end{thm}

\begin{proof}
 The formulation of the theorem tacitly includes the claim that
the functor $\Lambda_I$ takes values in $R\modl_{I\secmp}$, i.~e.,
the $I$\+adic completion of any $R$\+module is $I$\+adically
separated and complete.
 This is the result of~\cite[Corollary~3.6]{Yek0}, which is
\emph{not} true without the assumption that $I$ is finitely
generated~\cite[Example~1.8]{Yek0}.
 The point is that the projective limit topology is, by
the definition, always complete on
$\Lambda_I(C)=\varprojlim_n C/I^nC$, but the $I$\+adic topology
on $\Lambda_I(C)$ can differ from the projective limit topology
(see~\cite[Example~1.13 and the preceding discussion]{Yek0}).

 The two topologies are the same when the ideal $I$ is finitely
generated.
 A more precise claim, which is of key importance here, is that
the submodule $I^n\Lambda_I(C)$ coincides with the kernel $K_n$
of the natural projection $\Lambda_I(C)\rarrow C/I^nC$.
 One obviously has $I^n\Lambda_I(C)\subset K_n$, so $\Lambda_I(C)$
is always $I$\+adically separated.
 Furthermore, when $I$ is finitely generated, combining the results
of Lemma~\ref{I-completion-s-contra} and
Theorem~\ref{contraadjusted-for-ideal-generators-complete}
proves that $\Lambda_I(C)$ is $I$\+adically complete.
 Then one can apply~\cite[Theorem~1.5]{Yek0} in order to deduce
the assertion that $K_n=I^n\Lambda_I(C)$.

 An explicit proof of the equation $K_n=I^n\Lambda_I(C)$ based on
the infinite summation operation technique would look as
follows.
 Let $c_i\in C$, $i\ge1$ be a sequence of elements such that
$c_{i+1}-c_i \in I^iC$ for $i\ge1$ and $c_i=0$ for $i\le n$.
 Let $s_1$,~\dots, $s_m$ be some set of generators of the ideal~$I$.
 Arguing as in the proof of
Theorem~\ref{contraadjusted-for-ideal-generators-complete},
there are elements $a_{i_1,\dotsc,i_m}\in C$, \ $i_1$,~\dots, $i_m\ge0$
such that $a_{0,\dotsc,0}=c_1$ and
$$
 c_{i+1}-c_i = \sum\nolimits_{i_1\ge0,\dotsc,i_m\ge0}^{i_1+\dotsb+i_m=i}
 s_1^{i_1}\dotsm s_m^{i_m}a_{i_1,\dotsc,i_m}
 \quad\text{for all $i\ge1$}.
$$
 In addition, we can assume that $a_{i_1,\dotsc,i_m}=0$ whenever
$i_1+\dotsb+i_m<n$.
 Following Lemma~\ref{I-completion-s-contra} and
Remark~\ref{many-variables-power-summations}, or just using
directly the limits of finite partial sums in the projective limit
topology on $\varprojlim_i C/I^iC$, we have
an $[s_1,\dotsc,s_m]$\+power infinite summation operation
in~$\Lambda_I(C)$.
 Hence the element $b\in\varprojlim_i C/I^iC$ represented by
the sequence $(c_i)_{i=1}^\infty$ can be expressed as
$$
 b = \sum\nolimits_{i_1,\dotsc,i_m=0}^\infty s_1^{i_1}\dotsm s_m^{i_m}
 a_{i_1,\dotsc,i_m},
$$ 
where the images of the elements $a_{i_1,\dotsc,i_m}$ in
$\varprojlim_i C/I^iC$ are
denoted for simplicity also by $a_{i_1,\dotsc,i_m}$.
 Finally, it is a standard exercise in combinatorics to rewrite
this infinite sum as a finite sum of $[s_1,\dotsc,s_j]$\+power
infinite sums, $1\le j\le m$, each of which is divisible by
a certain monomial of degree~$n$ in $s_1$,~\dots,~$s_m$.

 Now we can show that the two functors are adjoint.
 Let $D=\varprojlim_n D/I^nD$ be an $I$\+adically separated
and complete $R$\+module, and let $C$ be an arbitrary $R$\+module.
 Then $R$\+module homomorphisms $C\rarrow D$ correspond bijectively
to morphisms of projective systems of $R$\+modules
$(C/I^nC)_{n\ge1}\rarrow (D/I^nD)_{n\ge1}$, since any $R$\+module
homomorphism $C\rarrow D/I^nD$ factorizes through the surjection
$C\rarrow C/I^nC$ in a unique way.
 In particular, $R$\+module homomorphisms $\Lambda_I(C)\rarrow D$
correspond bijectively to morphisms of projective systems
$(\Lambda_I(C)/I^n\Lambda_I(C))_n\rarrow (D/I^nD)_n$.
 But we have just seen that $\Lambda_I(C)/I^n\Lambda_I(C) = 
\Lambda_I(C)/K_n = C/I^nC$.
 Hence the desired bijection $\Hom_R(\Lambda_I(C),D)=\Hom_R(C,D)$.
\end{proof}

\begin{rem}
 The category $R\modl_{I\secmp}$ is ``na\"\i ve'' and not well-behaved
(see Example~\ref{s-contra-adjusted-counterex}\,(1)); and so
the functor $\Lambda_I$ is ``na\"\i ve'' and not well-behaved
outside of the classical situation of finitely-generated modules
over Noetherian rings (when the Artin--Rees lemma is applicable).
 In fact, for $R=\Z$ and $I=p\Z$, applying the $p$\+adic completion
functor $\Lambda_p$ to the embedding $D/E\rarrow C/E$ from
Example~\ref{s-contra-adjusted-counterex}\,(1) produces the zero
map $D/E\rarrow C/D$, while applying the functor $\Lambda_p$ to
the short exact sequence $0\rarrow E\rarrow C\rarrow C/E\rarrow0$
one obtains the sequence $0\rarrow E\rarrow C\rarrow C/D\rarrow 0$,
which is not exact at the middle term~\cite[Example~3.20]{Yek0}.
 So, as pointed out in~\cite[Proposition~1.2]{Yek0}, preserving
surjections seems to be the only good property of
the functor~$\Lambda_I$ in general.
 When the ideal $I$~is finitely generated, one can also say that
$\Lambda_I$ preserves infinite products and is a reflector onto
its image (hence idempotent); but that is about it.

 The explanation is that the functor $\Lambda_I$ is constructed
by composing the right exact functor of reduction modulo $I^n$ with
the left exact functor of projective limit.
 Composing functors that are exact from different sides is not
generally a good procedure.
 Composing the derived functors between derived categories (and then
taking the degree~$0$ cohomology if needed) is advisable instead.
 In the situation at hand, this would mean applying the functor
$\Lambda_I$ to a flat resolution of a given $R$\+module~\cite{PSY},
which works well for many, but not all, finitely generated
ideals~$I$ (cf.\ the counterexample in~\cite[Example~2.6]{Pmgm}).
 The ``well-behaved version'' of the functor $\Lambda_I$, denoted
by $\Delta_I$, will be discussed in the next two sections.
\end{rem}

\Section{$s$-Torsion Modules and the Functor $\Gamma_s$, \\*
$s$-Contramodules and the Functor $\Delta_s$}
\label{functor-delta-s-first-secn}

 The fundamental idea of covariant duality between torsion modules
and contramodules, which lurked beneath the surface of our exposition
before, becomes explicitly utilized in this and the next section.
 Let $R$ be a commutative ring and $s\in R$ an element.
 An $R$\+module $M$ is said to be \emph{$s$\+torsion} if for each
element $x\in M$ there exists an integer $n\ge1$ such that $s^nx=0$
in~$M$.
 Equivalently, one can say that an $R$\+module $M$ is $s$\+torsion
if and only if $R[s^{-1}]\ot_RM=0$.

 The following lemma provides one of the simplest illustrations of
the torsion module-contramodule duality.

\begin{lem} \label{tensor-hom-torsion-contra-duality-lem}
\textup{(a)} Let $N$ and $M$ be $R$\+modules.
 Then the tensor product $R$\+module $N\ot_R M$ is $s$\+torsion
whenever either $N$, or $M$ is $s$\+torsion. \par
\textup{(b)} Let $M$ and $C$ be $R$\+modules.
 Then the\/ $\Hom$ $R$\+module\/ $\Hom_R(M,C)$ is an $s$\+contramodule
whenever either $M$ is $s$\+torsion, or $C$ is an $s$\+contramodule.
\end{lem}

\begin{proof}
 Part~(a): one has
$$
 R[s^{-1}]\ot_R(N\ot_RM)=(R[s^{-1}]\ot_RN)\ot_RM,
$$
hence $R[s^{-1}]\ot_RN=0$ implies $R[s^{-1}]\ot_R(N\ot_RM)=0$.

 Part~(b): denoting the derived functor of $R$\+module homomorphisms,
viewed as a functor between the derived categories of $R$\+modules,
by $\boR\Hom_R$, one has
\begin{align*}
 \Hom_R(R[s^{-1}],\Hom_R(M,C)) &=
 H^0(\boR\Hom_R(R[s^{-1}],\.\boR\Hom_R(M,C))), \\
 \Ext^1_R(R[s^{-1}],\Hom_R(M,C)) &\subset
 H^1(\boR\Hom_R(R[s^{-1}],\.\boR\Hom_R(M,C))).
\end{align*}
 Hence from
\begin{align*}
 \boR\Hom_R(R[s^{-1}],\.\boR\Hom_R(M,C)) &=
 \boR\Hom_R(R[s^{-1}]\ot_R M,\>C) \\
 &= \boR\Hom_R(M,\.\boR\Hom_R(R[s^{-1}],C))
\end{align*}
we can conclude that either $R[s^{-1}]\ot_RM=0$ or
$\Ext^*_R(R[s^{-1}],C)=0$ is sufficient to imply
$\Ext^*_R(R[s^{-1}],\Hom_R(M,C))=0$. 

 Alternatively, one can define an $s$\+power infinite summation
operation on the $R$\+module $\Hom_R(M,C)$ by the rules
$$
 \left(\sum\nolimits_{n=0}^\infty s^nf_n\right)(x) =
 \sum\nolimits_{n=0}^\infty f_n(s^nx)
$$
for all $f_n\in\Hom_R(M,C)$, \ $x\in M$ if $M$ is $s$\+torsion
(so the sum in the right-hand side is finite, because $s^nx=0$
for large enough~$n$), or
$$
 \left(\sum\nolimits_{n=0}^\infty s^nf_n\right)(x) =
 \sum\nolimits_{n=0}^\infty s^n (f_n(x))
$$
if $C$ is an $s$\+contramodule (so the sum in the right-hand side
refers to the $s$\+power infinite summation operation on~$C$).
\end{proof}

 The next lemma is a generalization of the previous one that will be
useful in Section~\ref{krull-dim-1-secn}.
 In its context, the derived tensor product $N^\bu\ot_R^\boL M^\bu$ of
unbounded complexes is computed, as usually in the unbounded derived
category of modules, using homotopy flat resolutions, while
the $R$\+modules of homomorphisms in the unbounded derived category
$\Hom_{\D(R\modl)}(M^\bu,C^\bu[n])=H^n(\boR\Hom_R(M^\bu,C^\bu))$
can be computed using homotopy projective or homotopy injective
resolutions.

\begin{lem} \label{tor-ext-torsion-contra-lem}
\textup{(a)} Let $N^\bu$ and $M^\bu$ be two complexes of $R$\+modules
such that either the modules $H^n(N)$ are $s$\+torsion for all
$n\in\Z$, or the modules $H^n(M)$ are $s$\+torsion for all $n\in\Z$.
 Then the derived tensor product modules $H^n(N^\bu\ot_R^\boL M^\bu)$
are $s$\+torsion for all $n\in\Z$. \par
\textup{(b)} Let $M^\bu$ and $C^\bu$ be two complexes of $R$\+modules
such that either the modules $H^n(M)$ are $s$\+torsion for all $n\in\Z$,
or the modules $H^n(C)$ are $s$\+contramodules for all $n\in\Z$.
 Then the $R$\+modules\/ $\Hom_{\D(R\modl)}(M^\bu,C^\bu[n])$ are
$s$\+contramodules for all $n\in\Z$.
\end{lem}

\begin{proof}
 Part~(a): the $R$\+module $R[s^{-1}]$ is flat, so given a complex
of $R$\+modules $L^\bu$, the $R$\+modules $H^n(L^\bu)$ are $s$\+torsion
for all $n\in\Z$ if and only if the complex $R[s^{-1}]\ot_RL^\bu$
is acyclic.
 Now one has
$$
 R[s^{-1}]\ot_R(N^\bu\ot_R^\boL M^\bu)=
 (R[s^{-1}]\ot_RN^\bu)\ot_R^\boL M^\bu,
$$
hence acyclicity of the complex $R[s^{-1}]\ot_RN^\bu$ implies
acyclicity of the complex $R[s^{-1}]\ot_R(N^\bu\ot_R^\boL M^\bu)$.

 Part~(b): the $R$\+module $R[s^{-1}]$ has projective dimension
at most~$1$, so for any complex of $R$\+modules $B^\bu$ there are
short exact sequences
\begin{multline*}
 0\lrarrow\Ext^1_R(R[s^{-1}],H^{n-1}(B^\bu))\lrarrow \\
 H^n(\boR\Hom_R(R[s^{-1}],B^\bu))
 \lrarrow\Hom_R(R[s^{-1}],H^n(B^\bu))\lrarrow0.
\end{multline*}
 Therefore, the $R$\+modules $H^n(B^\bu)$ are $s$\+contramodules
for all $n\in\Z$ if and only if $\boR\Hom_R(R[s^{-1}],B^\bu)=0$
in $\D(R\modl)$.
 Now one has
\begin{align*}
 \boR\Hom_R(R[s^{-1}],\.\boR\Hom_R(M^\bu,C^\bu)) &=
 \boR\Hom_R(R[s^{-1}]\ot_R M^\bu,\>C^\bu) \\
 &= \boR\Hom_R(M^\bu,\.\boR\Hom_R(R[s^{-1}],C^\bu)),
\end{align*}
hence acyclicity of one of the complexes $R[s^{-1}]\ot_RM^\bu$ or
$\boR\Hom_R(R[s^{-1}],C^\bu)$ implies acyclicity of the complex
$\boR\Hom_R(R[s^{-1}],\.\boR\Hom_R(M^\bu,C^\bu))$.

 Alternatively, for bounded complexes $N^\bu$, $M^\bu$, and $C^\bu$
one can deduce Lemma~\ref{tor-ext-torsion-contra-lem} from
Lemma~\ref{tensor-hom-torsion-contra-duality-lem}.
 For example, let us explain the first case in part~(b).
 Since the class of $s$\+contramodule $R$\+modules is closed
the kernels, cokernels, and extensions, the question reduces to proving
that the $R$\+modules $\Ext^n_R(M,C)$ are $s$\+contramodules for every
$s$\+torsion $R$\+module $M$ and every $R$\+module~$C$.
 Here it suffices to choose an injective $R$\+module resolution
for $C$ and apply Lemma~\ref{tensor-hom-torsion-contra-duality-lem}(b).
\end{proof}

 Let $R\modl_{s\tors}\subset R\modl$ denote the full subcategory of
$s$\+torsion $R$\+modules in $R\modl$.
 We recall from Theorem~\ref{serre-subcategory}(b) that
$R\modl_{s\tors}$ is an abelian category, and, in fact, even
a Serre subcategory in $R\modl$.
 Denote by $\Gamma_s(M)\subset M$ the maximal $s$\+torsion submodule
of an $R$\+module~$M$.
 Then the functor $\Gamma_s$ is right adjoint to the fully faithful
embedding functor $R\modl_{s\tors}\rarrow R\modl$.

 We start from computing the functor $\Gamma_s$ and then proceed to
construct the dual-analogous functor $\Delta_s$.

\begin{lem} \label{gamma-s-lemma}
 For any $R$\+module $M$, the following $R$\+modules are naturally
isomorphic to each other:
\begin{enumerate}
\renewcommand{\theenumi}{\roman{enumi}}
\item the $R$\+module\/ $\Gamma_s(M)$;
\item the kernel of the $R$\+module morphism
$$
 \psi_s^M\:\bigoplus\nolimits_{n\ge0} M\lrarrow
 \bigoplus\nolimits_{n\ge1}M
$$
taking an eventually vanishing sequence $x_0$, $x_1$,
$x_2$,~\dots~$\in M$ to the sequence
$$
 (y_1,y_2,y_3,\dotsc)=\psi_s^M(x_0,x_1,x_2,\dotsc), \quad
 y_n=x_n-sx_{n-1}, \quad n\ge1;
$$
\item assuming that $s$~is not a zero-divisor in~$R$, the $R$\+module\/
$\Tor^R_1(R[s^{-1}]/R,\.M)$.
\end{enumerate}
\end{lem} 

\begin{proof}
 (i)${}\simeq{}$(ii): The equations $y_n=0$, $n\ge1$ mean that
$x_n=sx_{n-1}$ for $n\ge1$, that is $x_n=s^nx_0$.
 The submodule formed by all the elements $x_0\in M$ for which
the sequence $s^nx_0\in M$, \ $n\ge1$ is eventually vanishing
coincides with $\Gamma_s(M)\subset M$.

 (i)${}\simeq{}$(iii): When $s$ is a nonzero-divisor in $R$,
the $R$\+module $R[s^{-1}]/R$ has a two-term flat left resolution
$R\rarrow R[s^{-1}]$.
 Hence the $R$\+module $\Tor^R_1(R[s^{-1}]/R,\.M)$ is computed as
the kernel of the morphism $(R\to R[s^{-1}])\ot_R M =
(M\to R[s^{-1}]\ot_RM)$, which obviously coincides with~$\Gamma_s(M)$.

 (ii)${}\simeq{}$(iii): When $s$ is a nonzero-divisor, the two-term
complex
\begin{equation} \label{s-telescope-morphism}
 \psi_s^R\:\bigoplus\nolimits_{n\ge0}R\lrarrow
 \bigoplus\nolimits_{n\ge1}R
\end{equation}
is a projective left resolution of the $R$\+module $R[s^{-1}]/R$
(cf.\ the proof of Lemma~\ref{ext-r-s-minus-1-computed}).
 Therefore, the $R$\+module $\Tor^R_1(R[s^{-1}]/R,\.M)$ is computed as
the kernel of the morphism $\psi_s^R\ot_R M=\psi_s^M$.
\end{proof}

 For any commutative ring $R$ and element $s\in R$, we denote by
$R\modl_{s\ctra}\subset R\modl$ the full subcategory of
$s$\+contramodule $R$\+modules.
 According to Theorem~\ref{ext-0-1-orthogonal}(a), the category
$R\modl_{s\ctra}$ is abelian and the embedding functor
$R\modl_{s\ctra}\allowbreak\rarrow R\modl$ is exact.

\begin{thm} \label{delta-s-theorem}
 For any $R$\+module $C$, the following $R$\+modules are naturally
isomorphic to each other:
\begin{enumerate}
\renewcommand{\theenumi}{\roman{enumi}}
\item the cokernel of the $R$\+module morphism
\begin{equation} \label{phi-s-morphism}
 \phi^s_C\:\prod\nolimits_{n\ge1}C\lrarrow\prod\nolimits_{n\ge0}C
\end{equation}
taking a sequence $c_1$, $c_2$, $c_3$,~\dots~$\in C$, \ $c_0=0$ to
the sequence
$$
 (b_0,b_1,b_2,\dots)=\phi^s_C(c_1,c_2,c_3,\dots),
 \quad b_n=c_n-sc_{n+1}, \quad n\ge0;
$$
\item the cokernel of the endomorphism of the $R$\+module
$C[[z]]$ of formal power series in one variable~$z$ with
coefficients in~$C$
$$
 (z-s)\:C[[z]]\lrarrow C[[z]]
$$
which is the difference of the endomorphism of multiplication
by~$z$ (coming from the formal power series structure) and
the endomorphism~$s$ (induced by the endomorphism~$s$ on
the coefficient module~$C$);
\item assuming that $s$ is not a zero-divisor in $R$,
the $R$\+module $\Ext^1_R(R[s^{-1}]/R,\.C)$.
\end{enumerate}
 Denote the $R$\+module produced by either of the constructions
(i)\+-(iii) by\/~$\Delta_s(C)$.
 Then the functor\/ $\Delta_s\:R\modl\rarrow R\modl_{s\ctra}$ is
left adjoint to the fully faithful embedding functor
$R\modl_{s\ctra}\rarrow R\modl$.
\end{thm}

\begin{proof}
(i)${}\simeq{}$(ii): The obvious isomorphisms $\prod_{n\ge1}C
\simeq C[[z]]\simeq\prod_{n\ge0}C$ identify the morphism~$\phi^s_C$
with the morphism~$z-s$.

(i)${}\simeq{}$(iii): When $s$ is a nonzero-divisor,
the $R$\+module $\Ext_R^1(R[s^{-1}]/R,\.C)$ can be computed
as the cokernel of the morphism $\Hom_R(\psi_s^R,C)$
(see~\eqref{s-telescope-morphism}), which is readily identified
with the morphism~$\phi^s_C$.

 Before proving that the two functors are adjoint
(cf.~\cite[the proof of Proposition~2.1]{Pmgm}), we have to check
that $\Delta_s$ takes values in $R\modl_{s\ctra}$, i.~e.,
the $R$\+module $\Delta_s(C)$ is an $s$\+contramodule for every
$R$\+module~$C$.
 This claim only depends on the action of~$s$ and not on the rest
of the $R$\+module structures on our modules, so we can assume
that $R=\Z[s]$.
 Furthermore, the functor $\Delta_s$ is right exact by construction
(as the cokernel of a morphism is a right exact functor and
the functor of infinite product of abelian groups is exact).
 So it can be computed using an initial fragment of a free
$\Z[s]$\+module resolution of the module~$C$.

 In other words, the functor $\Delta_s$ preserves cokernels.
 Let us present the $\Z[s]$\+module $C$ as the cokernel of
a morphism of free $\Z[s]$\+modules $u\:A'[s]\rarrow A''[s]$, where
$A'$ and $A''$ are some (free) abelian groups.
 Then $\Delta_s(C)$ is the cokernel of the morphism
$\Delta_s(u)\:\Delta_s(A'[s])\rarrow\Delta_s(A''[s])$.
 Since the class of $s$\+contramodules is closed under the cokernels,
it suffices to check that the $\Z[s]$\+module $\Delta_s(A[s])$
is an $s$\+contramodule for any abelian group~$A$.

 This is accomplished by an explicit computation; and
the construction~(ii) seems to be the most convenient one.
 We leave it to the reader to compute that the quotient module
$A[s][[z]]/(z-s)$ is naturally isomorphic to the $\Z[s]$\+module
of formal power series $A[[s]]$; so $\Delta_s(A[s])=A[[s]]$.
 That is clearly an $s$\+contramodule.

 Now let $D$ be an $s$\+contramodule $R$\+module and $C$ an arbitrary
$R$\+module.
 By the definition, the adjunction morphism $\delta_{s,C}\:C\rarrow
\Delta_s(C)$ is induced by the embedding of $C$ into $\prod_{n\ge0}C$
as the $(n=0)$-indexed factor of the product.
 To any $R$\+module morphism $g\:\Delta_s(C)\rarrow D$ one assigns
the composition $g\delta_{s,C}\:C\rarrow\Delta_s(C)\rarrow D$.
 We have to check that this construction defines an isomorphism
of the Hom modules $\Hom_R(\Delta_s(C),D)\simeq\Hom_R(C,D)$.

 Let $f\:C\rarrow D$ be an $R$\+module morphism; we need to show
that it factorizes through the morphism~$\delta_{s,C}$.
 Define a morphism $g'\:\prod_{n\ge0}C\rarrow D$ by
the rule
$$
 g'(b_0,b_1,b_2,\dotsc)=\sum\nolimits_{n=0}^\infty s^nf(b_n).
$$
 Then for any sequence $(c_n\in C)_{n\ge1}$, \ $c_0=0$ we have
\begin{multline*}
 g'(\phi^s_C((c_n)_{n=1}^\infty)) = g'((c_n-sc_{n+1})_{n=0}^\infty)=
 \sum\nolimits_{n=0}^\infty s^n(f(c_n-sc_{n+1})) \\
 =\sum\nolimits_{n=0}^\infty s^nf(c_n)-
 \sum\nolimits_{n=0}^\infty s^n(sf(c_{n+1}))= f(c_0)=0,
\end{multline*}
so the morphism~$g'$ factorizes through the cokernel of~$\phi_C^s$
and induces a morphism $g\:\Delta_s(C)\rarrow D$.
 Using the contraunitality axiom for infinite summation operations,
one can easily check that $g\delta_{s,C}=f$.

 Finally, let $h\:\Delta_s(C)\rarrow D$ be an $R$\+module morphism
for which $h\delta_{s,C}=0$.
 Denote the composition of~$h$ with the projection
$\prod_{n=0}^\infty C\rarrow\Delta_s(C)$ by
$h'\:\prod_{n=0}^\infty C\rarrow D$.
 Given any sequence $(b_n\in C)_{n=0}^\infty$, set
$$
 d_n=h'((b_{n+i})_{i=0}^\infty)=h'(b_n,b_{n+1},b_{n+2},\dotsc)\in D.
$$
 Then we have
\begin{setlength}{\multlinegap}{0pt}
\begin{multline*}
 d_n-sd_{n+1} =
 h'(b_n-sb_{n+1},\>b_{n+1}-sb_{n+2},\>b_{n+2}-sb_{n+3},\>\dotsc) \\ 
 =h'(-sb_{n+1},\>b_{n+1}-sb_{n+2},\>b_{n+2}-sb_{n+3},\>\dotsc)
 =h'(\phi_C^s(b_{n+1},b_{n+2},b_{n+3},\dots))=0
\end{multline*}
for every $n\ge0$, since $h'(b,0,0,\dots)=0$ for any $b\in C$
by assumption.
 Since $D$ has no nonzero $s$\+divisible submodules, we conclude
that $d_n=0$ for all $n\ge0$, and in particular $d_0=0$.
 Thus $h'((b_n)_{n=0}^\infty)=0$ for any sequence
$(b_n\in C)_{n=0}^\infty$ and $h=0$.
\end{setlength}
\end{proof}

\begin{rem} \label{gamma-delta-zero-divisor-remark}
 The two-term complex~\eqref{s-telescope-morphism} plays a central
role in the constructions of
Lemma~\ref{gamma-s-lemma} and Theorem~\ref{delta-s-theorem}.
 Let us denote it by $T^\bu(R;s)$ and place in the cohomological
degrees~$0$ and~$1$, so that $T^0(R;s)=\bigoplus_{n=0}^\infty R$ and
$T^1(R;s)=\bigoplus_{n=1}^\infty R$.
 The complex $T^\bu(R;s)[1]$ is quasi-isomorphic to the two-term
complex $R\rarrow R[s^{-1}]$ (where the term $R$ sits in
the cohomological degree~$-1$ and the term $R[s^{-1}]$ sits in
the cohomological degree~$0$).
 The quasi-isomorphism is provided by the map taking an eventually
vanishing sequence $x_0$, $x_1$, $x_2$,~\dots~$\in R$ to the element
$x_0\in R$ and an eventually vanishing sequence $y_1$, $y_2$,
$y_3$,~\dots~$\in R$ to the element
$-\sum_{n=1}^\infty y_n/s^n\in R[s^{-1}]$.
 Hence the assumption that $s$ is not a zero-divisor in $R$ can be
removed from the constructions of Lemma~\ref{gamma-s-lemma}(iii)
and Theorem~\ref{delta-s-theorem}(iii) by saying that
$\Gamma_s(M)=H^{-1}((R\to R[s^{-1}])\ot_RM)$ and $\Delta_s(C)=
\Hom_{\D^\b(R\modl)}((R\to R[s^{-1}]),\.C[1])$ for any commutative
ring $R$, an element $s\in R$, and $R$\+modules $M$ and~$C$.
\end{rem}

\begin{rem}
 The following observations, suggested to the author by the anonymous
referee, shed some additional light on
Lemma~\ref{gamma-s-lemma} and Theorem~\ref{delta-s-theorem}.
 Pick a derived category symbol $\star=\b$, $+$, $-$, or~$\varnothing$,
and consider the related (bounded or unbounded) derived category of
$R$\+modules $\D^\star(R\modl)$.
 The restriction-of-scalars functor $\D^\star(R[s^{-1}]\modl)\rarrow
\D^\star(R\modl)$ is a fully faithful embedding whose essential image
is the full subcategory in $\D^\star(R\modl)$ formed by all
the complexes in whose cohomology modules $s$~acts by automorphism.
 The functor $\D^\star(R[s^{-1}]\modl)\rarrow\D^\star(R\modl)$ has
adjoints on both sides, the left adjoint being the extension of
scalars $M^\bu\longmapsto R[s^{-1}]\ot_RM^\bu$ and the right adjoint
being the coextension of scalars $C^\bu\longmapsto
\boR\Hom_R(R[s^{-1}],C^\bu)$.
 The kernel of this extension-of-scalars functor is the full subcategory
$\D^\star_{s\tors}(R\modl)\subset\D^\star(R\modl)$ of complexes of
$R$\+modules with $s$\+torsion cohomology modules, while the kernel
of the coextension of scalars is the full subcategory
$\D^\star_{s\ctra}(R\modl)\subset\D^\star(R\modl)$ of complexes of
$R$\+modules with $s$\+contramodule cohomology modules
(cf.\ the proof of Lemma~\ref{tor-ext-torsion-contra-lem}).

 The embedding functor $\D^\star_{s\tors}(R\modl)\rarrow\D^\star(R\modl)$
has a right adjoint, which can be computed as the functor
$(R\to R[s^{-1}])[-1]\ot_R{-}$, while the embedding functor
$\D^\star_{s\ctra}(R\modl)\rarrow\D^\star(R\modl)$ has a left adjoint,
which is computed as the functor $\Hom_R(T^\bu(R;s),{-})$
(see~\cite[Section~3]{Pmgm} or~\cite[Section~4]{PMat}).
 The functor $\Gamma_s$ is the composition $R\modl\rarrow
\D^\star(R\modl)\rarrow\D^\star_{s\tors}(R\modl)\rarrow R\modl_{s\tors}$
of the functor $(R\to R[s^{-1}])[-1]\ot_R\nobreak{-}$ with the embedding
$R\modl\rarrow\D^\star(R\modl)$ and the degree-zero cohomology functor
$\D^\star_{s\tors}(R\modl)\rarrow R\modl_{s\tors}$, while the functor
$\Delta_s$ is the similar composition $R\modl\rarrow\D^\star(R\modl)
\rarrow\D^\star_{s\ctra}(R\modl)\rarrow R\modl_{s\ctra}$ of the functor
$\Hom_R(T^\bu(R;s),{-})$ with the degree-zero cohomology functor
$\D^\star_{s\ctra}(R\modl)\rarrow R\modl_{s\ctra}$.

 It follows that both the triangulated categories
$\D^\star_{s\tors}(R\modl)$ and $\D^\star_{s\ctra}(R\modl)$ are equivalent
to the quotient category $\D^\star(R\modl)/\D^\star(R[s^{-1}]\modl)$.
 The mutually inverse equivalences between $\D^\star_{s\tors}(R\modl)$
and $\D^\star_{s\ctra}(R\modl)$ are given by the restrictions of
the functors $(R\to R[s^{-1}])[-1]\ot_R{-}$ and
$\Hom_R(T^\bu(R;s),{-})$.
 The resulting two t\+structures on the triangulated category
$\D^\star_{s\tors}(R\modl)\simeq\D^\star_{s\ctra}(R\modl)$ are connected
by tilting with respect to torsion pairs (in the sense of
an appropriate generalization of~\cite[Section~1.2]{HRS}).
 The torsion pair in the abelian category $R\modl_{s\tors}$ consists
of the classes of $s$\+divisible $s$\+torsion $R$\+modules and
$s$\+reduced $s$\+torsion $R$\+modules.
 The torsion pair in the abelian category $R\modl_{s\ctra}$ consists
of the classes of $s$\+special $s$\+contramodule $R$\+modules and
$s$\+torsion-free $s$\+contramodule
$R$\+modules~\cite[Section~5]{PMat}.
\end{rem}

 Given a commutative ring $R$, an element $t\in R$, and an $R$\+module
$M$, we denote by ${}_tM\subset M$ the submodule of all elements
annihilated by~$t$ in~$M$.

 For any element $s$ in a commutative ring $R$, one can assign to
every $R$\+module $C$ two projective systems of $R$\+modules.
 One of them is formed by the $R$\+modules $C/s^nC$ and the natural
surjective morphisms between them.
 The other one consists of the $R$\+modules ${}_{s^n}C$ and
the multiplication maps $s\:{}_{s^{n+1}}C\rarrow{}_{s^n}C$.

 Let us denote by $\Lambda_s$ the $s$\+completion functor
$C\longmapsto\Lambda_s(C)=\varprojlim_{n\ge1}C/s^nC$.
 The following lemma provides a comparison between the functors
$\Lambda_s$ and~$\Delta_s$.

\begin{lem} \label{delta-s-lambda-s-short-sequence}
 Let $R$ be a commutative ring and $s\in R$ be an element.
 Then for any $R$\+module $C$ there is a natural short exact sequence
of $R$\+modules
$$
 0\lrarrow\varprojlim\nolimits^1_{n\ge1}\.{}_{s^n}C\lrarrow\Delta_s(C)
 \lrarrow\Lambda_s(C)\lrarrow 0.
$$
\end{lem}

\begin{proof}
 Denote by $T_n^\bu(R;s)$, \,$n\ge1$ the subcomplex
\begin{equation} \label{s-telescope-subcomplex}
 \bigoplus\nolimits_{i=0}^{n-1} R\lrarrow
 \bigoplus\nolimits_{i=1}^n R
\end{equation}
of the complex~\eqref{s-telescope-morphism}.
 As in Remark~\ref{gamma-delta-zero-divisor-remark}, 
\,$T_n^\bu(R;s)$ is viewed as a complex concentrated in
the cohomological degrees~$0$ and~$1$.
 The complex $T_n^\bu(R;s)$ is homotopy equivalent to the two-term
complex of $R$\+modules $R\overset{s^n}\rarrow R$.
 The complexes $\Hom_R(T_n^\bu(R;s),C)$ form a projective system
with termwise surjective morphisms of complexes, hence there is
a short exact sequence of homology modules
\begin{setlength}{\multlinegap}{0pt}
\begin{multline*}
 0\lrarrow \varprojlim\nolimits_n^1H_1(\Hom_R(T_n^\bu(R;s),C)) \\
 \lrarrow H_0(\varprojlim\nolimits_n\Hom_R(T_n^\bu(R;s),C))\lrarrow
 \varprojlim\nolimits_n H_0(\Hom_R(T_n^\bu(R;s),C))\lrarrow0.
\end{multline*}
 Furthermore, we have
$$
 \varprojlim\nolimits_{n\ge1}\Hom_R(T_n^\bu(R;s),C)=
  \Hom_R(T^\bu(R;s),C).
$$
 It remains to recall that $H_0(\Hom_R(T^\bu(R;s),C))=\Delta_s(C)$
by Theorem~\ref{delta-s-theorem}(i), while
$H_1(\Hom_R(T^\bu_n(R;s),C))=H_1(\Hom_R((R\overset{s^n}\to R),\.C))
=\.{}_{s^n}C$ and $H_0(\Hom_R(T^\bu_n(R;s),\allowbreak C))=
H_0(\Hom_R((R\overset{s^n}\to R),\.C))=C/s^nC$.
 (See the proof of Lemma~\ref{delta-I-lambda-I} below for a more
detailed discussion.)
\end{setlength}
\end{proof}

\begin{cor} \label{delta-s-to-lambda-s}
 Let $R$ be a commutative ring and $s\in R$ an element.  Then \par
\textup{(a)} for any $R$\+module $C$, there is a natural surjective
$R$\+module morphism $\Delta_s(C)\rarrow\Lambda_s(C)$; \par
\textup{(b)} for any $s$\+torsion-free $R$\+module $C$, the morphism
$\Delta_s(C)\rarrow\Lambda_s(C)$ is an isomorphism.
\end{cor}

\begin{proof}
 Both assertions follow immediately from
Lemma~\ref{delta-s-lambda-s-short-sequence}.
 Alternatively, one can deduce part~(a) from
Theorem~\ref{s-contraadjusted-complete}(a) and
part~(b) from Theorem~\ref{s-contraadjusted-complete}(b),
as we will now explain.

 Part~(a): by Theorem~\ref{s-separ-complete-contra}(a), the category of
$s$\+separated and $s$\+complete $R$\+modules $R\modl_{s\secmp}$
is contained in the category of $s$\+contramodule $R$\+modules
$R\modl_{s\ctra}$, which in turn is contained in the category $R\modl$.
 By Theorem~\ref{lambda-adjoint}, the functor $\Lambda_s$ is left
adjoint to the embedding functor $R\modl_{s\secmp}\rarrow R\modl$,
while by Theorem~\ref{delta-s-theorem}, the functor $\Delta_s$ is left
adjoint to the embedding functor $R\modl_{s\ctra}\rarrow R\modl$.
 It follows that there is a natural transformation $\Delta_s\rarrow
\Lambda_s$ forming a commutative triangle with the adjunction morphisms
$$
 C\lrarrow\Delta_s(C)\lrarrow\Lambda_s(C)
$$
for any $R$\+module~$C$. 
 Furthermore, we have $\Lambda_s(\Delta_s(C))=\Lambda_s(C)$.
 Since $\Delta_s(C)$ is an $s$\+contramodule, it follows from
Theorem~\ref{s-contraadjusted-complete}(a) that $\Delta_s(C)$
is $s$\+complete, hence the morphism
$\Delta_s(C)\rarrow\Lambda_s(C)$ is surjective.

 Part~(b): in view of the proof of part~(a), it suffices to check that
the $R$\+module $\Delta_s(C)$ is $s$\+torsion-free for every
$s$\+torsion-free $R$\+module~$C$.
 According to Remark~\ref{gamma-delta-zero-divisor-remark}, we have
$\Delta_s(B)=\Hom_{\D^\b(R\modl)}((R\to R[s^{-1}]),\.B[1])$ for any
$R$\+module~$B$.
 Hence from the short exact sequence $0\rarrow C\rarrow C\rarrow
C/sC\rarrow0$ we obtain a long exact sequence
\begin{multline*}
 \dotsb\lrarrow\Hom_{\D^\b(R\modl)}((R\to R[s^{-1}]),\.C/sC) \\ \lrarrow
 \Delta_s(C)\overset s\lrarrow\Delta_s(C)\lrarrow\Delta_s(C/sC)
 \lrarrow0.
\end{multline*}
 It remains to notice that $\Hom_{\D^\b(R\modl)}((R\to R[s^{-1}]),\.D)=
\Hom_R(R[s^{-1}]/R,D)$ for any $R$\+module $D$, and the right-hand side
vanishes when $sD=0$.
\end{proof}

\begin{rem} \label{lambda-s-torsion-free}
 One can also show that the natural morphism $\Delta_s(C)\rarrow
\Lambda_s(C)$ is an isomorphism whenever the $R$\+module
$\Lambda_s(C)$ is $s$\+torsion-free.
 Indeed, both $\Delta_s(C)$ and $\Lambda_s(C)$ are $s$\+contramodules,
so the kernel $K$ of the morphism in question is an $s$\+contramodule,
too.
 Furthermore, the morphism $\Delta_s(C)/s\Delta_s(C)\rarrow
\Lambda_s(C)/s\Lambda_s(C)$ is always an isomorphism, because for
any $R$\+module $M$ with $sM=0$ one has
\begin{alignat*}{3}
 &\Hom_R(\Delta_s(C)/s\Delta_s(C),\.M) &&=
 \Hom_R(\Delta_s(C),M) &&=\Hom_R(C,M) \\
 &\Hom_R(\Lambda_s(C)/s\Lambda_s(C),\.M) &&=
 \Hom_R(\Lambda_s(C),M) &&=\Hom_R(C,M).
\end{alignat*}
 Assuming that $\Lambda_s(C)$ is $s$\+torsion-free, from the exact
sequence $0\rarrow K\rarrow\Delta_s(C)\rarrow\Lambda_s(C)\rarrow0$
we get an exact sequence $0\rarrow K/sK\rarrow\Delta_s(C)/s\Delta_s(C)
\allowbreak\rarrow\Lambda_s(C)/s\Lambda_s(C)\rarrow0$.
 Hence $K=sK$ and it follows that $K=0$.
 (See Lemma~\ref{tor-lambda-nakayama} below for a generalization.)

 More generally, an $R$\+module $C$ is said to have \emph{bounded
$s$\+torsion} if there exists $m\ge1$ such that $s^nc=0$ implies
$s^mc=0$ for all $n\ge1$ and $c\in C$.
 One can observe that $\varprojlim_{n\ge1}^1\.{}_{s^n}C=0$ whenever
the $s$\+torsion in $C$ is bounded.
 By Lemma~\ref{delta-s-lambda-s-short-sequence}, the natural morphism
$\Delta_s(C)\rarrow\Lambda_s(C)$ is an isomorphism in this case.
 Hence it follows that the morphism $\Delta_s(C)\rarrow\Lambda_s(C)$
is an isomorphism whenever the $R$\+module $\Delta_s(C)$ has
bounded $s$\+torsion.

 Conversely, assume that the $R$\+module $\Lambda_s(C)$ has bounded
$s$\+torsion.
 Then, for every $n\ge1$, from the exact sequence $0\rarrow K
\rarrow\Delta_s(C)\rarrow\Lambda_s(C)\rarrow0$ we get an exact
sequence ${}_{s^n}\Lambda_s(C)\rarrow K/s^nK\rarrow\Delta_s(C)/s^n
\Delta_s(C)\rarrow\Lambda_s(C)/s^n\Lambda_s(C)\rarrow0$,
where ${}_tM$ denotes the submodule annihilated by an element $t\in R$
in an $R$\+module~$M$.
 For any $n\ge k$, we have a morphism of such exact sequences
with the maps $M/s^nM\rarrow M/s^kM$ being the natural surjections
and the map ${}_{s^n}\Lambda_s(C)\rarrow{}_{s^k}\Lambda_s(C)$ being
the multiplication with $s^{n-k}$.
 Since the map $\Delta_s(C)/s^n\Delta_s(C)\allowbreak\rarrow
\Lambda_s(C)/s^n\Lambda_s(C)$ is an isomorphism, the map
${}_{s^n}\Lambda_s(C)\rarrow K/s^nK$ is surjective.
 So is the map $K/s^nK\rarrow K/sK$.
 Now if the $s$\+torsion in $\Lambda_s(C)$ is bounded, then the map
${}_{s^n}\Lambda_s(C)\rarrow{}_s\.\Lambda_s(C)$ vanishes for $n$ large
enough, and it follows from the commutativity of the diagram that
$K/sK=0$.
 Hence $K=0$.
\end{rem}

 Now we deduce a corollary providing a ``non-na\"\i ve''
version of the contraadjustedness criterion from
Theorem~\ref{s-contraadjusted-complete}.

\begin{cor} \label{s-contraadjusted-criterion}
 Let $R$ be a commutative ring, $s\in R$ an element, and $C$
an $R$\+module.  Then \par
\textup{(a)} an $R$\+module $C$ has no $s$\+divisible submodules if
and only if the adjunction morphism $C\rarrow\Delta_s(C)$ is injective;
\par
\textup{(b)} an $R$\+module $C$ is $s$\+contraadjusted if and only if
the adjunction morphism $C\rarrow\Delta_s(C)$ is surjective.
\end{cor}

\begin{proof}
 First let us consider the case when $s$ is not a zero-divisor in~$R$.
 Then from the short exact sequence
$$
 0\lrarrow R\lrarrow R[s^{-1}]\lrarrow R[s^{-1}]/R\lrarrow0
$$
we obtain the long exact sequence of $\Ext$ groups
\begin{multline*}
 0\lrarrow\Hom_R(R[s^{-1}]/R,C)\lrarrow\Hom_R(R[s^{-1}],C) \\*
 \lrarrow C\lrarrow\Ext^1(R[s^{-1}]/R,C)\lrarrow
 \Ext^1_R(R[s^{-1}],C)\lrarrow0
\end{multline*}
 Recalling that $\Ext^1_R(R[s^{-1}]/R,C)=\Delta_s(C)$ by
Theorem~\ref{delta-s-theorem}(iii), we obtain both
the assertions (a) and~(b).
 In the general case, it suffices to notice that the properties
in question do not depend on the $R$\+module structure on $C$,
but only on the action of the element~$s$; so a ring $R$ can be
replaced with the polynomial ring~$\Z[s]$.

 Alternatively, the argument in the general case can be similar,
except that one has to use the cone in lieu of the cokernel.
 Denote by $(R\to R[s^{-1}])\in\D^\b(R\modl)$ the derived category
object represented by the two-term complex $R\rarrow R[s^{-1}]$, which
is quasi-isomorphic to the complex~\eqref{s-telescope-morphism}.
 Placing the complex $R\rarrow R[s^{-1}]$ in the cohomological
degrees~$-1$ and~$0$, one has a distinguished triangle
$$
 R\lrarrow R[s^{-1}]\lrarrow (R\to R[s^{-1}])\lrarrow R[1]
$$
in $\D^\b(R\modl)$.
 Hence the related long exact sequence of triangulated $\Hom$
\begin{multline*}
 \dotsb\lrarrow\Hom_{\D^\b(R\modl)}(R[s^{-1}],C)\lrarrow
 \Hom_{\D^\b(R\modl)}(R,C) \\ \lrarrow
 \Hom_{\D^\b(R\modl)}((R\to R[s^{-1}]),\>C[1]) \\ \lrarrow
 \Hom_{\D^\b(R\modl)}(R[s^{-1}],C[1])\lrarrow\Hom_{\D^\b(R\modl)}(R,C[1])
 \lrarrow\dotsb,
\end{multline*}
with $\Hom_{\D^\b(R\modl)}(R[s^{-1}],C)=\Hom_R(R[s^{-1}],C)$, \
$\Hom_{\D^\b(R\modl)}(R,C)=\Hom_R(R,\allowbreak C)=C$, \
$\Hom_{\D^\b(R\modl)}(R[s^{-1}],C[1])=\Ext^1_R(R[s^{-1}],C)$, and
$\Hom_{\D^\b(R\modl)}(R,C[1])\allowbreak=\Ext^1_R(R,C)=0$.
 Finally, notice the isomorphism
$$
 \Hom_{\D^\b(R\modl)}((R\to R[s^{-1}]),\>C[1]) = \Delta_s(C)
$$
from Remark~\ref{gamma-delta-zero-divisor-remark}.
 Now exactness of the long sequence implies both~(a) and~(b).
\end{proof}

\begin{rem} \label{bounded-torsion-remark}
 Combining Corollaries~\ref{delta-s-to-lambda-s}
and~\ref{s-contraadjusted-criterion}(b), one can obtain a new
proof of Theorem~\ref{s-contraadjusted-complete}.
 Similarly, Theorem~\ref{s-separ-complete-contra}(b) follows
from Corollary~\ref{delta-s-to-lambda-s}.

 More generally, the morphism $\Delta_s(C)\rarrow\Lambda_s(C)$ is
an isomorphism when the $s$\+torsion in $C$ is bounded (see
Remark~\ref{lambda-s-torsion-free}).
 This allows to weaken the ``$s$\+torsion-free'' assumption to
``bounded $s$\+torsion'' in the assertions of
Theorems~\ref{s-contraadjusted-complete}(b)
and~\ref{s-separ-complete-contra}(b).
 Similarly, it follows from
Corollary~\ref{s-contraadjusted-criterion}(a) that any $R$\+module
with bounded $s$\+torsion and without $s$\+divisible submodules is
$s$\+separated (cf.\ Remark~\ref{reduced-separated-remark}).
\end{rem}

\Section{$I$-Torsion Modules and the Functor $\Gamma_I$, \\*
$I$-Contramodules and the Functor $\Delta_I$}
\label{functor-delta-I-second-secn}

 Let $I$ be an ideal in a commutative ring~$R$.
 An $R$\+module $M$ is said to be \emph{$I$\+torsion} if it is
$s$\+torsion for every $s\in I$.
 Clearly, it suffices to check the latter condition for some set
of generators of the ideal~$I$: if $I$ is generated by some
elements~$s_j$ and $M$ is $s_j$\+torsion for every~$j$, then $M$
is $I$\+torsion.
 The dual-analogous contramodule version of this observation is
Theorem~\ref{ideal-contramodule-thm} (to be proved again below
in this section).

 Let $R\modl_{I\tors}\subset R\modl$ denote the full subcategory
of $I$\+torsion $R$\+modules.
 Denote by $\Gamma_I(M)$ the maximal $I$\+torsion submodule of
an $R$\+module~$M$.
 Then the functor $\Gamma_I$ is right adjoint to the fully faithful
embedding $R\modl_{I\tors}\rarrow R\modl$.

 As the category $R\modl_{I\tors}$ is abelian and its embedding functor
$R\modl_{I\tors}\rarrow R\modl$ is exact, the functor $\Gamma_I$ (viewed
either as a functor $R\modl\rarrow R\modl_{I\tors}$ or as a functor
$R\modl\rarrow R\modl$) is left exact.
 It also preserves infinite direct sums; and the full subcategory
$R\modl_{I\tors}\subset R\modl$ is closed under subobjects, quotients,
extensions, and infinite direct sums in $R\modl$.

 We recall the notation $T^\bu(R;s)$ for
the complex~\eqref{s-telescope-morphism}
(see Remark~\ref{gamma-delta-zero-divisor-remark}).
 Furthermore, we set
$$
 T^\bu(R;s_1,\dotsc,s_m) = T^\bu(R;s_1)\ot_R\dotsb\ot_R T^\bu(R;s_m).
$$
 Hence $T^\bu(R;s_1,\dotsc,s_m)$ is a complex of countably-generated
free $R$\+modules concentrated in the cohomological
degrees~$0$,~\dots,~$m$.
 Since the complex $T^\bu(R;s_j)$ is quasi-isomorphic to the two-term
complex $R\rarrow R[s_j^{-1}]$, the complex $T^\bu(R;s_1,\dotsc,s_m)$
is quasi-isomorphic to the complex
\begin{equation} \label{augmented-cech}
 \check C\sptilde(R;s_1,\dotsc,s_m)=(R\rarrow R[s_1^{-1}])
 \ot_R\ot\dotsb\ot_R(R\rarrow R[s_m^{-1}]),
\end{equation}
which looks explicitly as
$$
 R\lrarrow\bigoplus\nolimits_{j=1}^m R[s_j^{-1}]\lrarrow
 \bigoplus\nolimits_{j'<j''} R[s_{j'}^{-1},s_{j''}^{-1}]
 \lrarrow\dotsb\lrarrow R[s_1^{-1},\dotsc,s_m^{-1}].
$$
 For the reasons discussed below in Remark~\ref{cech-remark}, we call
the complex $\check C\sptilde(R;s_1,\dotsc,s_m)$ the \emph{augmented
\v Cech complex} of the ring $R$ with the elements $s_1$,~\dots,~$s_m$.
 This is a complex of flat (in fact, even very flat,
see~\cite{Pcosh,ST} for the definition) $R$\+modules concentrated
in the cohomological degrees~$0$,~\dots,~$m$.

\begin{lem}  \label{gamma-I-lemma}
 For any $R$\+module $M$, the following $R$\+modules are naturally
isomorphic to each other:
\begin{enumerate}
\renewcommand{\theenumi}{\roman{enumi}}
\item the maximal $I$\+torsion submodule\/ $\Gamma_I(M)$;
\item the maximal $s_1$-, $s_2$-,~\dots, and $s_m$\+torsion
submodule\/ $\Gamma_{s_m}\dotsm\Gamma_{s_2}\Gamma_{s_1}(M)$;
\item the cohomology module $H^0(T^\bu(R;s_1,\dotsc,s_m)\ot_R M)$.
\end{enumerate}
\end{lem}

\begin{proof}
 (i) and~(ii) are clearly the same submodule in~$M$.

 (i)${}\simeq{}$(iii): The complex $T^\bu(R;s_1,\dotsc,s_m)\ot_RM$
is quasi-isomorphic to the complex
$\check C\sptilde(R;s_1,\dotsc,s_m)\ot_R M$, which starts as
$M\rarrow \bigoplus_{j=1}^mM[s_j^{-1}]$.
 The kernel of the latter morphism coincides with $\Gamma_I(M)$.

 (ii)${}\simeq{}$(iii): For any complex of $R$\+modules $K^\bu$
concentrated in the cohomological degrees~$\ge0$ and any
complex of flat $R$\+modules $T^\bu$ concentrated in
the cohomological degrees~$\ge0$, one has
$$
 H^0(T^\bu\ot_R K^\bu)\simeq H^0(T^\bu\ot_R H^0(K^\bu)).
$$
 By Lemma~\ref{gamma-s-lemma}\,(i)$\simeq$(ii), we have
$H^0(T^\bu(R;s_j)\ot_RM)\simeq \Gamma_{s_j}(M)$.
 Hence
\begin{multline*}
 H^0\bigl(T^\bu(R;s_1,s_2)\ot_RM\bigr) \simeq
 H^0\bigl(T^\bu(R;s_2)\ot_R\bigl(T^\bu(R;s_1)\ot_RM\bigr)\bigr) \\
 \simeq H^0\bigl(T^\bu(R;s_2)\ot_R
 H^0\bigl(T^\bu(R;s_1)\ot_R M\bigr)\bigr) \\ \simeq
 H^0\bigl(T^\bu(R;s_2)\ot_R \Gamma_{s_1}(M)\bigr)\simeq
 \Gamma_{s_2}\Gamma_{s_1}(M),
\end{multline*}
etc.
 The argument finishes by induction on~$j$.
\end{proof}

 Let $I\subset R$ be the ideal generated by a finite
set of elements $s_1$,~\dots,~$s_m$.
 As our aim is to prove Theorem~\ref{ideal-contramodule-thm}
rather than just use it, let us introduce the temporary notation
$R\modl_{[s_1,\dotsc,s_m]\ctra}\subset R\modl$ for the full
subcategory of all $R$\+modules that are $s_j$\+contramodules
for every $1\le j\le m$.
 (When our proof is finished, we will switch to the permanent notation
$R\modl_{[s_1,\dotsc,s_m]\ctra}=R\modl_{I\ctra}$.)

 According to Theorem~\ref{ext-0-1-orthogonal}(a), the category
$R\modl_{[s_1,\dotsc,s_m]\ctra}$ is abelian and its embedding functor
$R\modl_{[s_1,\dotsc,s_m]\ctra}\rarrow R\modl$ is exact.
 The full subcategory $R\modl_{[s_1,\dotsc,s_m]\ctra}$ is also closed
under infinite products in $R\modl$.

\begin{thm} \label{delta-s1-sm-theorem}
 For any $R$\+module $C$, the following $R$\+modules are
naturally isomorphic to each other:
\begin{enumerate}
\renewcommand{\theenumi}{\roman{enumi}}
\item the module\/ $\Delta_{s_m}\dotsm\Delta_{s_2}\Delta_{s_1}(C)$;
\item the quotient module
$$
 C[[z_1,\dotsc,z_m]]\Big/
 \sum\nolimits_{j=1}^m(z_j-s_j)C[[z_1,\dotsc,z_m]]
$$
of the module of formal power series $C[[z_1,\dotsc,z_m]]$ in~$n$
variables $z_1$,~\dots, $z_m$ with coefficients in $C$ by the sum of
the images of the operators
$$
 z_1-s_1,\, \ z_2-s_2, \ \dotsc, \ z_m-s_m\:
 C[[z_1,\dotsc,z_m]]\rarrow C[[z_1,\dotsc,z_m]];
$$
\item the homology module $H_0(\Hom_R(T^\bu(R;s_1,\dotsc,s_m),C))$.
\end{enumerate}
 Denote the $R$\+module produced by either of
the constructions~(i)\+-(iii) by\/~$\Delta_{s_1,\dotsc,s_m}(C)$.
 Then the functor\/ $\Delta_{s_1,\dotsc,s_m}\:R\modl\rarrow
R\modl_{[s_1,\dotsc,s_m]\ctra}$ is left adjoint to the fully faithful
embedding functor $R\modl_{[s_1,\dotsc,s_m]\ctra}\rarrow R\modl$.
\end{thm}

\begin{proof}
 (i)${}\simeq{}$(ii): The natural isomorphism is easily constructed
using the observation that the functor assigning to an abelian
group $A$ the group $A[[z]]$ is exact and, in particular,
preserves cokernels.

 (i)${}\simeq{}$(iii): For any complex of $R$\+modules $K_\bu$
concentrated in the homological degrees~$\ge0$ and any complex
of projective $R$\+modules $T^\bu$ concentrated in
the cohomological degrees~$\ge0$, one has
$$
 H_0(\Hom_R(T^\bu,K_\bu))\simeq H_0(\Hom_R(T^\bu,H_0(K_\bu))).
$$
 Hence
\begin{multline*}
 H_0\bigl(\Hom_R\bigl(T^\bu(R;s_1,s_2),C\bigr)\bigr)\simeq
 H_0\bigl(\Hom_R\bigl(T^\bu(R;s_2),\Hom_R\bigl(T^\bu(R;s_1),C
 \bigr)\bigr)\bigr) \\ \simeq
 H_0\bigl(\Hom_R\bigl(T^\bu(R;s_2),
 H_0\bigl(\Hom_R\bigl(T^\bu(R;s_1),C\bigr)\bigr)\bigr)\bigr) \\
 \simeq H_0\bigl(\Hom_R\bigl(T^\bu(R,s_2),\Delta_{s_1}(C)\bigr)\bigr)
 \simeq \Delta_{s_2}\Delta_{s_1}(C),
\end{multline*}
etc.

 It remains to show that the functor $\Delta_{s_1,\dotsc,s_m}$ is left
adjoint to the embedding functor $R\modl_{[s_1,\dotsc,s_m]\ctra}\rarrow
R\modl$ (cf.~\cite[proof of Proposition~2.1]{Pmgm}).
 The key observation is that, for any two elements $s$ and $t\in R$,
the functor $\Delta_t$ takes $s$\+contramodules to
$s$\+contramodules.
 Indeed, the class of $s$\+contramodules is closed under
the infinite products and cokernels in $R\modl$, hence
the cokernel of the morphism~\eqref{phi-s-morphism}
is an $s$\+contramodule whenever the $R$\+module $C$~is.
 It follows that, for any $R$\+module $C$, the $R$\+module
$\Delta_{s_m}\dotsm\Delta_{s_2}\Delta_{s_1}(C)$ is
an $s_j$\+contramodule for every $1\le j\le m$.

 Now for an arbitrary $R$\+module $C$ and an $R$\+module $D$ from
the subcategory $R\modl_{[s_1,\dotsc,s_m]\ctra}\subset R\modl$
one has
\begin{multline*}
 \Hom_R(\Delta_m\dotsm\Delta_2\Delta_1(C),D)\simeq
 \Hom_R(\Delta_{m-1}\dotsm\Delta_2\Delta_1(C),D) \\ \simeq
 \dotsb\simeq\Hom_R(\Delta_2\Delta_1(C),D)\simeq
 \Hom_R(\Delta_1(C),D)\simeq\Hom_R(C,D)
\end{multline*}
due to the adjointness properties of the functors~$\Delta_{s_j}$.
\end{proof}

 For any ideal $I\subset R$, one denotes by $\sqrt I \subset R$
the radical of the ideal $I$, i.~e., the ideal formed by all
the elements $s\in R$ for which there exists $n\ge1$ such that
$s^n\in I$.
 The following result is due to Porta, Shaul, and
Yekutieli~\cite[Theorem~6.1]{PSY}.

\begin{thm}  \label{s-t-homotopy-equivalent}
 Let $s_1$,~\dots, $s_m$ and $t_1$,~\dots, $t_k$ be two finite sets
of elements in a commutative ring~$R$.
 Denote by $I=(s_1,\dotsc,s_m)$ and $J=(t_1,\dotsc,t_k)\subset R$
the ideals generated by the first and the second set, respectively.
 Suppose that $\sqrt I = \sqrt J$ in~$R$.
 Then the two complexes of $R$\+modules $T^\bu(R;s_1,\dotsc,s_m)$ and
$T^\bu(R;t_1,\dotsc,t_k)$ are homotopy equivalent.
\end{thm}

\begin{proof}[Sketch of proof\/~\cite{PSY}]
 According to Lemma~\ref{gamma-I-lemma}\,(i)$\simeq$(iii), for any
$R$\+module $M$ we have
$$
 H^0(T^\bu(R;s_1,\dotsc,s_m)\ot_RM)\simeq \Gamma_I(M)
 \simeq H^0(T^\bu(R;t_1,\dotsc,t_k)\ot_RM).
$$
 Therefore, the cohomology of the complexes
$T^\bu(R;s_1,\dotsc,s_m)\ot_RM$ and $T^\bu(R;t_1,\dotsc,\allowbreak
t_k)\ot_RM$ form cohomological $\delta$\+functors of the argument
$M\in R\modl$ taking values in $R\modl$.
 A sequence of elements $s_1$,~\dots, $s_m$ in a commutative ring $R$
is called \emph{weakly proregular}, if the cohomology
$H^*(T^\bu(R;s_1,\dotsc,s_m)\ot_RM)$ computes the right derived functor
$\boR\Gamma_I^*(M)$ of the functor $\Gamma_I\:R\modl\rarrow R\modl$.
 In other words, it means that one should have
$$
 H^q(T^\bu(R;s_1,\dotsc,s_m)\ot_RK)=0 \quad\text{for all $q>0$}
$$
when $K$ is an injective $R$\+module.
 (After the theorem will have been proved, it will follow that
the weak proregularity is a property of the ideal $I$ and
even $\sqrt{I}$ rather than of the generating sequence of elements.)

 The fact that the two ideals $\sqrt{I}$ and $\sqrt{J}$ coincide in
$R$ can be expressed as a finite system of equations on a finite
set of elements involved (including the elements $s_j$, $t_i$
and the elements used as the coefficients of the expressions of
powers of~$s_j$ as linear combinations of~$t_i$ and vice versa).
 Denote by $R'\subset R$ the subring generated by this finite set
of elements over~$\Z$.
 Then we have
\begin{align*}
 T^\bu(R;s_1,\dotsc,s_m)&=R\ot_{R'}T^\bu(R';s_1,\dotsc,s_m), \\
 T^\bu(R;t_1,\dotsc,t_k)&=R\ot_{R'}T^\bu(R';t_1,\dotsc,t_k).
\end{align*}
 Hence it suffices to prove that the complexes of $R'$\+modules
$T^\bu(R';s_1,\dotsc,s_m)$ and $T^\bu(R';t_1,\dotsc,t_k)$ are
homotopy equivalent.
 This reduces the assertion to be proved to the case of
a Noetherian ring~$R$.

 According to~\cite[Theorem~4.34]{PSY}, any finite sequence of
elements in a Noetherian ring is weakly proregular.
 Hence both complexes $T^\bu(R;s_1,\dotsc,s_m)$ and
$T^\bu(R;t_1,\dotsc,t_k)$ compute the derived functor
$\boR^*\Gamma_I(R)$ for the $R$\+module~$R$.
 Now let $K^\bu$ be an injective $R$\+module resolution of
the module~$R$; then we have a chain of quasi-isomorphisms
\begin{alignat*}{2}
 T^\bu(R;s_1,\dotsc,s_m)&\lrarrow T^\bu(K^\bu;s_1,\dotsc,s_m)
 &&\llarrow \Gamma_I(K^\bu) \\ 
 T^\bu(R;t_1,\dotsc,t_k)&\lrarrow T^\bu(K^\bu;t_1,\dotsc,t_k)
 &&\llarrow \Gamma_I(K^\bu)
\end{alignat*}
connecting $T^\bu(R;s_1,\dotsc,s_m)$ with $T^\bu(R;t_1,\dotsc,t_k)$
(where the bicomplexes are presumed to have been replaced with
their total complexes).
 Finally, two finite complexes of free $R$\+modules are homotopy
equivalent whenever they are quasi-isomorphic.
\end{proof}

\begin{proof}[Second proof of Theorem~\ref{ideal-contramodule-thm}]
 Clearly, it suffices to consider the case of a finitely generated
ideal.
 Let $R$ be a commutative ring, and let $s_1$,~\dots, $s_m$ and
$t_1$,~\dots, $t_m$ be two sequences of elements in $R$ generating
ideals $I$ and $J$ such that $\sqrt I=\sqrt J$.
 Then, by Theorem~\ref{delta-s1-sm-theorem}, the full subcategories
$R\modl_{[s_1,\dotsc,s_m]\ctra}$ and $R\modl_{[t_1,\dotsc,t_k]\ctra}\subset
R\modl$ are the essential images of their reflector functors
$\Delta_{s_1,\dotsc,s_m}$ and $\Delta_{t_1,\dotsc,t_k}$.
 Furthermore, for any $R$\+module $C$, the $R$\+modules
$\Delta_{s_1,\dotsc,s_m}(C)$ and $\Delta_{t_1,\dotsc,t_k}(C)$
can be computed as the homology modules
$$
 H_0(\Hom_R(T^\bu(R;s_1,\dotsc,s_m),\.C))
 \quad\text{and}\quad
 H_0(\Hom_R(T^\bu(R;t_1,\dotsc,t_k),\.C)).
$$
 Finally, according to Theorem~\ref{s-t-homotopy-equivalent},
the two complexes of $R$\+modules $T^\bu(R;s_1,\dotsc,s_m)$
and $T^\bu(R;s_1,\dotsc,s_m)$ are homotopy equivalent.
 Thus the functors $\Delta_{s_1,\dotsc,s_m}$ and
$\Delta_{t_1,\dotsc,t_k}$ are isomorphic, and it follows that
the two subcategories $R\modl_{[s_1,\dotsc,s_m]\ctra}$ and
$R\modl_{[t_1,\dotsc,t_k]\ctra}$ in $R\modl$ coincide.
\end{proof}

 Now we can set $R\modl_{I\ctra}=R\modl_{[s_1,\dotsc,s_m]\ctra}$
and $\Delta_I=\Delta_{s_1,\dotsc,s_m}$.

\begin{rem} \label{cech-remark}
 One can avoid the reduction to Noetherian rings in the proof
of Theorem~\ref{s-t-homotopy-equivalent} using the following
geometric argument instead.
 Let $X=\Spec R$ denote the affine scheme with the ring of
functions~$R$.
 Then the category of quasi-coherent sheaves on $X$ is equivalent
to the category of $R$\+modules; let us denote by $\cM$
the quasi-coherent sheaf corresponding to the $R$\+module $M$.
 For any element $s\in R$, we have the related principal affine
open subscheme $U_s\subset X$.
 Given a sequence $s_1$,~\dots, $s_m\in R$ generating an ideal
$I\subset R$, consider the open subscheme $U=\bigcup_{j=1}^m U_{s_j}
\subset X$.
 Then the open subscheme $U\subset X$ depends only on the radical
$\sqrt{I}\subset R$ of the ideal $I$ and not on the generating
sequence $s_1$,~\dots, $s_m$ itself.

 The complex of $R$\+modules $T^\bu(R;s_1,\dotsc,s_m)$ is
quasi-isomorphic to the complex $\check C\sptilde(R;s_1,\dotsc,s_m)$
\eqref{augmented-cech}.
 Let $\check C(R;s_1,\dotsc,s_m)$ denote the nonaugmented version
of the same complex
$$
 \bigoplus\nolimits_{j=1}^m R[s_j^{-1}]\lrarrow
 \bigoplus\nolimits_{j'<j'} R[s_{j'}^{-1},s_{j''}^{-1}]
 \lrarrow\dotsb\lrarrow R[s_1^{-1},\dotsc,s_m^{-1}].
$$
 Then $\check C(R;s_1,\dotsc,s_m)\ot_RM$ is the \v Cech complex 
computing the quasi-coherent sheaf cohomology $H^*(U,\cM|_U)$
in terms of the affine covering $U=\bigcup_{j=1}^mU_{s_j}$, while
the complex $\check C\sptilde(R;s_1,\dotsc,s_m)\ot_RM$ computes
the cohomology of the cone of the restriction morphism between
(the complexes representing) $M=H^*(X,\cM)$ and $H^*(U,\cM|_U)$.
 It follows that the cohomology of both the complexes
$\check C(R;s_1,\dotsc,s_m)\ot_RM$ and
$\check C\sptilde(R;s_1,\dotsc,s_m)\ot_RM$ depend only on the open
subscheme $U\subset X$ and not on the generating sequence of
the ideal.
 Thus, given our two sequences $s_1$,~\dots, $s_m$ and
$t_1$,~\dots, $t_k$, the complexes $T^\bu(R;s_1,\dotsc,s_m)$
and $T^\bu(R;t_1,\dotsc,t_k)$ are connected by a chain
of quasi-isomorphisms
$$
 T^\bu(R;s_1,\dotsc,s_m)\llarrow T^\bu(R;s_1,\dotsc,s_m,t_1,\dotsc,t_k)
 \lrarrow T^\bu(R;t_1,\dotsc,t_k).
$$

 Weak proregularity of ideals in Noetherian rings is explained,
in the same geometric terms, by the facts that injective
quasi-coherent sheaves on Noetherian schemes are flasque and
their restrictions to open subschemes remain
injective~\cite[\S\,II.7]{Hart}.
 Hence, denoting by $\cK^\bu$ the complex of quasi-coherent
sheaves on $X$ corresponding to an injective $R$\+module resolution
$K^\bu$ of a given module $M$, one can compute $H^*(U,\cM|_U)$ as
the cohomology of the complex of sections $H^*(\cK^\bu(U))$.
 Furthermore, the morphism of complexes $\cK^\bu(X)\rarrow\cK^\bu(U)$
is sujective, so one can replace the cone with the kernel, which
leads one to the complex $\Gamma_I(K^\bu)$ \cite[Section~1]{Pmgm}.

 This discussion is meant to suggest that, generally speaking,
it is not a good idea to use the underived global sections of
the restrictions of injective quasi-coherent sheaves to open
subschemes in lieu of the cohomology of quasi-coherent sheaves on
such open subschemes.
 Therefore, outside of the weakly proregular case, it is the derived
functor $\boR\Gamma_I$ that appears to be ``naive'', and
the cohomology of the complex $T^\bu(R;s_1,\dotsc,s_m)\ot_RM$
or $\check C\sptilde(R;s_1,\dotsc,s_m)\ot_RM$ is its
``well-behaved replacement''.
 Similarly, the homology of the complex
$\Hom_R(T^\bu(R;s_1,\dotsc,s_m),C)$ is preferable to the derived
functor $\boL\Delta_I$ outside of the weakly proregular
case~\cite[Section~3]{Pmgm}, while in the weakly proregular case
they coincide~\cite[Lemma~2.7]{Pmgm}.
\end{rem}

 Recall the notation $T_n^\bu(R;s)\subset T^\bu(R;s)$, \,$n\ge1$,
from the proof of Lemma~\ref{delta-s-lambda-s-short-sequence}.
 The complex $T_n^\bu(R;s)$ is quasi-isomorphic to the two-term
complex $R\overset{s^n}\rarrow R$.
 Set~\cite[Section~5]{PSY}
$$
 T_n^\bu(R;s_1,\dotsc,s_m)=T_n^\bu(R;s_1)\ot_R\dotsb\ot_R T^\bu(R;s_m).
$$
 Then $T_n^\bu(R;s_1,\dotsc,s_m)$ is a subcomplex in
$T^\bu(R;s_1,\dotsc,s_m)$ and one has
$$
 T^\bu(R;s_1,\dotsc,s_m) = \varinjlim\nolimits_{n\ge1}
 T^\bu_n(R;s_1,\dotsc,s_m).
$$

\begin{lem} \label{delta-I-lambda-I}
 Let $R$ be a commutative ring and $I\subset R$ be an ideal generated
by a sequence of elements $s_1$,~\dots, $s_m\in R$.
 Then for any $R$\+module $C$ there is a natural short exact
sequence of $R$\+modules
$$
 0\lrarrow\varprojlim\nolimits_{n\ge1}^1
 H_1(\Hom_R(T_n^\bu(R;s_1,\dotsc,s_m),\.C))\lrarrow\Delta_I(C)
 \lrarrow \Lambda_I(C)\lrarrow0.
$$
\end{lem}

\begin{proof}
 The complexes $\Hom_R(T_n^\bu(R;s_1,\dotsc,s_m),C)$ form a projective
system of complexes and termwise surjective morphisms between them
indexed by the integers $n\ge1$.
 Furthermore, we have $H_0(\Hom_R(T_n^\bu(R;s_1,\dotsc,s_m),C))\simeq
C/(s_1^n,\dotsc,s_m^n)C$, where $(s_1^n,\dotsc,s_m^n)$ denotes the ideal
in $R$ generated by the elements $s_1^n$,~\dots, $s_m^n\in R$.
 Hence
$$
 \varprojlim\nolimits_n H_0(\Hom_R(T_n^\bu(R;s_1,\dotsc,s_m),\.C))
 \simeq \Lambda_I(C).
$$
 On the other hand, we have
$$
 \varprojlim\nolimits_n\Hom_R(T_n^\bu(R;s_1,\dotsc,s_m),C)=
 \Hom_R(T^\bu(R;s_1,\dotsc,s_m),C),
$$
so
$$
 H_0\bigl(\varprojlim_n\nolimits
 \Hom_R(T_n^\bu(R;s_1,\dotsc,s_m),\.C)\bigr)
 \simeq \Delta_I(C). 
$$
 Finally, since the projective limit functor $\varprojlim_{n\ge1}$ has
cohomological dimension~$1$, for any projective system of complexes
of $R$\+modules and termwise surjective morphisms between them
$K_1^\bu\larrow K_2^\bu\larrow K_3^\bu\larrow\dotsb$,
the hypercohomology spectral sequence reduces to
``universal coefficient'' short exact sequences
$$
 0\lrarrow\varprojlim\nolimits_{n\ge1}^1 H^{q-1}(K_n^\bu)\lrarrow
 H^q\bigl(\varprojlim\nolimits_{n\ge1} K_n^\bu\bigr)\lrarrow
 \varprojlim\nolimits_{n\ge1} H^q(K_n^\bu)\lrarrow0,
$$
where $\varprojlim^1_{n\ge1}$ denotes the (first) derived functor of
projective limit.
\end{proof}

\begin{rem} \label{contraadj-for-ideal-gens-adjunction-surjective}
 As a stronger version of
Theorem~\ref{contraadjusted-for-ideal-generators-complete}, one can
prove that the adjunction morphism $C\rarrow\Delta_I(C)$ is surjective
whenever the $R$\+module $C$ is $s_j$\+contraadjusted for every
$j=1$,~\dots,~$m$.
 Indeed, by Corollary~\ref{s-contraadjusted-criterion}(b)
the morphism $C\rarrow\Delta_{s_1}(C)$ is surjective.
 Since the class of $s$\+contraadjusted $R$\+modules is closed under
quotients for any $s\in R$, it follows that the $s_1$\+contramodule
$R$\+module $\Delta_{s_1}(C)$ is $s_j$\+contraadjusted for every
$j=2$,~\dots,~$m$.
 Applying Corollary~\ref{s-contraadjusted-criterion}(b) again,
we see that the morphism $\Delta_{s_1}(C)\rarrow\Delta_{s_2}
\Delta_{s_1}(C)$ is surjective, etc.
 Hence the morphism $C\rarrow\Delta_{s_m}\dotsm\Delta_{s_1}(C)$ is
surjective.
 (Cf.~\cite[Section~C.2]{Pcosh}.)
\end{rem}

\begin{rem}
 Given an arbitrary (not necessarily finitely generated) ideal $I$ in
a commutative ring $R$, one denotes by $R\modl_{I\ctra}$
the full subcategory in $R\modl$ consisting of all the $R$\+modules
that are $s$\+contramodules for every $s\in I$.
 In this generality, one can show, using category-theoretic techniques,
that the embedding functor $R\modl_{I\ctra}\rarrow R\modl$ has
a left adjoint functor $\Delta_I\:R\modl\rarrow R\modl_{I\ctra}$
\cite[Examples~4.1\,(2\+-3)]{PR}.
\end{rem}

\Section{Covers, Envelopes, and Cotorsion Theories}
\label{covers-envelopes-secn}

 This section consists almost entirely of the definitions.
 Its aim is to supply preliminary material for the remaining part of
the paper, including the more general theoretical discussion
in Sections~\ref{relatively-cotorsion-secn}\+-\ref{flat-cotorsion-secn}
and, most importantly, the examples considered
in Sections~\ref{abelian-groups-secn}\+-\ref{krull-dim-1-secn}.
 All the proofs in this section are omitted and replaced with
references to the book~\cite{Xu} and
the papers~\cite{Sal,ET,BBE,Ba,Sto,PR}.

 Let $R$ be an associative ring.
 Given two full subcategories $\F$ and $\C\subset R\modl$, one denotes
by $\F^\perp$ and ${}^\perp\C$ the full subcategories
\begin{align*}
 \F^\perp &= \{\,C\in R\modl\mid
 \Ext^1_R(F,C)=0 \text{ for all $F\in\F$}\,\}, \\
 {}^\perp\C &= \{\,F\in R\modl\mid
 \Ext^1_R(F,C)=0 \text{ for all $C\in\C$}\,\}.
\end{align*}
 A \emph{cotorsion theory} (or \emph{cotorsion pair}) in $R\modl$
is a pair of full subcategories $\F$, $\C\subset R\modl$ such that
$\F^\perp=\C$ and ${}^\perp\C=\F$.
 Given a class of modules $\sS\subset R\modl$, one construct
a cotorsion theory by setting $\C=\sS^\perp$ and $\F={}^\perp\C$.
 The cotorsion theory $(\F,\C)=({}^\perp(\sS^\perp),\,\sS^\perp)$
is said to be \emph{generated} by~$\sS$.

\begin{ex} \label{flat-cotorsion-example}
 Let $R\modl_\fl\subset R\modl$ denote the class of all flat left
$R$\+modules.
 A left $R$\+module $C$ is said to be \emph{cotorsion} if it belongs
to $R\modl_\fl^\perp$, i.~e., one has $\Ext^1_R(F,C)=0$ for every
flat left $R$\+module~$F$.
 The class of all cotorsion left $R$\+modules is denoted by
$R\modl_\cot\subset R\modl$.
 The pair of full subcategories $(R\modl_\fl,\>R\modl_\cot)$
is called the \emph{flat cotorsion pair/theory} in the category
of left $R$\+modules.
 This is the classical example of a cotorsion theory.

 Notice that one actually needs to prove that $(R\modl_\fl,
\>R\modl_\cot)$ is a cotorsion theory, i.~e., that any left
$R$\+module belonging to ${}^\perp R\modl_\cot$ is flat.
 This is the result of~\cite[Lemma~3.4.1]{Xu}, which can be also
obtained from Theorems~\ref{eklof-trlifaj}\+-%
\ref{flat-cotorsion-theory-complete} below.
\end{ex}

 One says that a cotorsion theory $(\F,\C)$ has \emph{enough
projectives} if every $R$\+module $M$ can be included in a short
exact sequence
\begin{equation} \label{special-precover}
 0\lrarrow C'\lrarrow F\lrarrow M\lrarrow0, \qquad
 F\in\F,\,\ C'\in\C,
\end{equation}
and that $(\F,\C)$ has \emph{enough injectives} if every
for every $R$\+module there exists a short exact sequence
\begin{equation} \label{special-preenvelope}
 0\lrarrow M\lrarrow C\lrarrow F'\lrarrow 0, \qquad
 F'\in\F,\,\ C\in\C.
\end{equation}
 Examples of exact
sequences~(\ref{special-precover}\+-\ref{special-preenvelope})
in the flat cotorsion theory will be provided in
Sections~\ref{abelian-groups-secn} and~\ref{krull-dim-1-secn}.

\begin{rem}
 For any cotorsion theory $(\F,\C)$ in $R\modl$, the full subcategories
$\F$ and $\C\subset R\modl$ are closed under extensions in $R\modl$.
 Therefore, they inherit the exact category structures from
the abelian category $R\modl$~\cite{Bueh}.
 The intersection $\F\cap\C$ coincides with the class of injective
objects in the exact category $\F$ and with the class of projective
objects in the exact category~$\C$.

 If an $R$\+module $M$ in a short exact
sequence~\eqref{special-precover} belongs to $\C$,
then the $R$\+module $F$, being an extension of $M\in\C$ and
$C'\in\C$, belongs to $\F\cap\C$.
 So, when a cotorsion theory $(\F,\C)$ has enough projectives,
the exact category $\C$ has enough projective objects.
 Building up short exact sequences~\eqref{special-precover},
one can then construct a projective resolution of a given
object $M$ in the exact category~$\C$.

 Similarly, if an $R$\+module $M$ in a short exact
sequence~\eqref{special-preenvelope} belongs to $\F$, then
the $R$\+module $C$ belongs to $\F\cap\C$.
 So if a cotorsion theory $(\F,\C)$ has enough injectives then
the exact category $\F$ has enough injective objects.
 Building up short exact sequences~\eqref{special-preenvelope},
one can construct an injective resolution of a given object
$M\in\F$.

 In particularly, Theorem~\ref{flat-cotorsion-theory-complete} below
implies that there are enough injective objects in the exact
category of flat left $R$\+modules and enough projective objects in
the exact category of cotorsion left $R$\+modules.
 Similarly, it will follow from
Corollary~\ref{very-flat-cotorsion-theory} that the exact category of
contraadjusted modules over a commutative ring has enough projective
objects.
 These are important observations for the theory of contraherent
cosheaves~\cite[Section~4.4]{Pcosh}.
\end{rem}

 The following lemma is due to Salce~\cite{Sal}
(see also~\cite[Lemma~5.20]{GT}, \cite[second paragraph of
the proof of Theorem~10]{ET}, or~\cite[Lemma~1.1.3]{Pcosh}).

\begin{lem}  \label{salce-lemma}
 A cotorsion theory in $R\modl$ has enough projectives if and only if
it has enough injectives. \qed
\end{lem}

 A cotorsion theory having enough projectives/injectives is called
\emph{complete}.

 Given a class of modules $\sS\subset R\modl$, one says that
an $R$\+module $F$ is a \emph{transfinitely iterated extension}
(in the sense of the inductive limit) of modules from $\sS$ if
there exists an ordinal~$\alpha$ and an increasing chain of
submodules $F_i\subset F$ indexed by the ordinals $i\le\alpha$
such that $F_0=0$, \ $F_\alpha=F$, \ $F_i=\bigcup_{j<i}F_j$
for all limit ordinals $i\le\alpha$, and $F_{i+1}/F_i\in\sS$
for all $i<\alpha$.
 The class of all modules representable as transfinitely iterated
extensions of modules from $\sS$ is denoted by $\filt(\sS)$.

 The next result is known as the Eklof lemma~\cite[Lemma~1]{ET}.

\begin{lem} \label{eklof-lemma}
 For any class of modules\/ $\sS\subset R\modl$, one has\/
$\sS^\perp=\filt(\sS)^\perp$. \qed
\end{lem}

 The following important existence theorem is due to
Eklof and Trlifaj~\cite[Theorem~10]{ET}.

\begin{thm} \label{eklof-trlifaj}
\textup{(a)} Any cotorsion theory\/ $(\F,\C)$ generated by a set
(rather than a proper class) of modules\/ $\sS\subset R\modl$
is complete. \par
\textup{(b)} If the $R$\+module $R$ belongs to $\sS$, then
the class\/ $\F$ consists precisely of all the direct summands of
the modules belonging to $\filt(\sS)$. \qed
\end{thm}

 A class of modules $\F\subset R\modl$ is called \emph{deconstructible}
if there exists a set of modules $\sS\subset R\modl$ such that
$\F=\filt(\sS)$.
 The following corollary of Theorem~\ref{eklof-trlifaj} is due to
Enochs~\cite[Proposition~2]{BBE}.

\begin{thm} \label{flat-cotorsion-theory-complete}
 For any associative ring $R$, the class of all flat $R$\+modules
is deconstructible.
 Therefore, the flat cotorsion theory in $R\modl$ is complete. \qed
\end{thm}

 According to Lemma~\ref{eklof-lemma}, in any cotorsion theory in
$R\modl$ the class $\F$ is closed under transfinitely iterated
extensions in the sense of the inductive limit; this includes finitely
iterated extensions and infinite direct sums.
 The class $\C$ is closed under extensions and the infinite products
(and, in fact, even under transfinitely iterated extensions in
the sense of the projective limit~\cite[Proposition~18]{ET},
\cite[Lemma~4.5]{PR}, cf.\ Lemma~\ref{dual-eklof} below).
 Both the classes $\F$ and $\C$ are closed under direct summands.

\begin{lem} \label{hereditary-cotorsion}
 Let\/ $(\F,\C)$ be a cotorsion theory in $R\modl$.
 Then the following four conditions are equivalent:
\begin{enumerate}
\renewcommand{\theenumi}{\roman{enumi}}
\item $\Ext^2_R(F,C)=0$ for all $F\in\F$ and $C\in\C$;
\item $\Ext^n_R(F,C)=0$ for all $F\in\F$, \,$C\in\C$, and $n\ge2$;
\item the class\/ $\F$ is closed under the kernels of surjective
morphisms;
\item the class\/ $\C$ is closed under the cokernels of injective
morphisms.
\end{enumerate}
\end{lem}

\begin{proof}
  See~\cite[Lemma~6.17]{Sto} (cf.\
Lemma~\ref{relatively-cotorsion-lemma} below).
\end{proof}

 A cotorsion theory in $R\modl$ is said to be \emph{hereditary} if
it satisfies the equivalent conditions of
Lemma~\ref{hereditary-cotorsion}.
 In particular, the flat cotorsion theory is hereditary, because
the condition~(iii) clearly holds.

\medskip

 Let $\F\subset R\modl$ be a class of $R$\+modules.
 An $R$\+module morphism $f\:F\rarrow M$ with $F\in\F$ is called
an \emph{$\F$\+precover} of an $R$\+module $M$ if any morphism
$f'\:F'\rarrow M$ with $F'\in\F$ factorizes through~$f$ (that is
for any such~$f'$ there exists a morphism $u\:F'\rarrow F$
such that $f'=fu$).
 A \emph{special\/ $\F$\+precover} of an $R$\+module $M$ is
a surjective morphism $f\:F\rarrow M$ with $F\in\F$ and
$\ker(f)\in\F^\perp$; in other words, it is a morphism that can be
included into a short exact sequence like~\eqref{special-precover}.

 An $\F$\+precover $f\:F\rarrow M$ is called an \emph{$\F$\+cover}
of the $R$\+module $M$ if for any endomorphism $u\:F\rarrow F$
the equation $fu=f$ implies that $u$~is an automorphism of~$F$,
i.~e., $u$~is invertible.
 Clearly, an $\F$\+cover of a given module $M$, if it exists, is
unique up to a (nonunique) isomorphism.

\begin{lem} \label{cover-wakamatsu}
\textup{(a)} Any special\/ $\F$\+precover is an\/ $\F$\+precover. \par
\textup{(b)} If the class\/ $\F$ is closed under extensions in
$R\modl$, then the kernel of any\/ $\F$\+cover belongs to\/~$\F^\perp$.
 In particular, any surjective\/ $\F$\+cover is special. \par
\textup{(c)} Assume that an $R$\+module $M$ admits an\/ $\F$\+cover.
 In this case, an\/ $\F$\+precover $f\:F\rarrow M$ is an\/ $\F$\+cover
if and only if the $R$\+module $F$ has no nonzero direct summands
contained in\/ $\ker f$.
\end{lem}

\begin{proof}
 Part~(a) is~\cite[Proposition~2.1.3]{Xu}.
 Part~(b) is known as Wakamatsu's lemma; this
is~\cite[Lemma~2.1.1]{Xu}.
 Part~(c) is~\cite[Corollary~1.2.8]{Xu}.
\end{proof}

 Dually, let $\C\subset R\modl$ be another class of $R$\+modules.
 An $R$\+module morphism $g\:M\rarrow C$ with $C\in\C$ is called
a \emph{$\C$\+preenvelope} of $M$ if any morphism $g'\:M\rarrow C'$
with $C'\in\C$ factorizes through~$g$ (that is for any~$g'$ there
exists a morphism $u\:C\rarrow C'$ such that $g'=ug$).
 A \emph{special\/ $\C$\+preenvelope} of an $R$\+module $M$ is
an injective morphism $g\:M\rarrow C$ with $C\in\C$ and
$\coker(f)\in{}^\perp\C$; in other words, it is a morphism that
can be included in a short exact sequence
like~\eqref{special-preenvelope}.

 A $\C$\+preenvelope $g\:M\rarrow C$ is called a \emph{$\C$\+envelope}
of the $R$\+module $M$ if for any endomorphism $u\:C\rarrow C$
the equation $ug=g$ implies that $u$~is an automorphism of~$C$.
 A $\C$\+envelope of a given module $M$, if it exists, is unique up to
a (nonunique) isomorphism.

\begin{lem} \label{envelope-wakamatsu}
\textup{(a)} Any special\/ $\C$\+preenvelope is a\/ $\C$\+preenvelope.
\par
\textup{(b)} If the class\/ $\C$ is closed under extensions in
$R\modl$, then the cokernel of any\/ $\C$\+envelope belongs
to\/~${}^\perp\C$.
 In particular, any injective\/ $\C$\+envelope is special. \par
\textup{(c)} Assume that an $R$\+module $M$ admits a\/
$\C$\+envelope.
 In this case, a\/ $\C$\+preenvelope $g\:M\rarrow C$ is
a\/ $\C$\+envelope if and only if the $R$\+module $C$ has no
proper direct summands containing\/ $\im g$.
\end{lem}

\begin{proof}
 Part~(a) is~\cite[Proposition~2.1.4]{Xu}.
 Part~(b) is Wakamatsu's lemma~\cite[Lemma~2.1.2]{Xu}.
 Part~(c) is~\cite[Corollary~1.2.3]{Xu}.
\end{proof}

\begin{lem} \label{direct-sum-cover-envelope}
\textup{(a)} Let\/ $\F\subset R\modl$ be a full subcategory closed
under finite direct sums.
 Then the direct sum of any two\/ $\F$\+precovers is
an\/ $\F$\+precover, the direct sum of any two special\/
$\F$\+precovers is a special\/ $\F$\+precover, and the direct sum of
any two\/ $\F$\+covers is an\/ $\F$\+cover. \par
\textup{(b)} Let\/ $\C\subset R\modl$ be a full subcategory closed
under finite direct sums.
 Then the direct sum of any two\/ $\C$\+preenvelopes is
a\/ $\C$\+preenvelope, the direct sum of any two special\/
$\C$\+preenvelopes is a special\/ $\C$\+preenvelope, and
the direct sum of any two\/ $\C$\+envelopes is a\/ $\C$\+envelope.
\end{lem}

\begin{proof}
 Part~(a): the assertions concerning precovers and special precovers
are straightforward.
 Concerning the covers, let $M_1$ and $M_2$ be two $R$\+modules and
$f_i\:F_i\rarrow M_i$ be their $\F$\+covers.
 We have to show that $f_1\oplus f_2\: F_1\oplus F_2\allowbreak\rarrow
M_1\oplus M_2$ is an $\F$\+cover.
 This is obvious when $\Hom_R(F_1,F_2)=0$ or $\Hom_R(F_2,F_1)=0$,
because endomorphisms of $F_1\oplus F_2$ are then represented by
triangular matrices, and a triangular matrix with invertible diagonal
entries is invertible.

 In the general case, one observes that any $2\times 2$ matrix with
an invertible diagonal entry can be naturally decomposed into
the product of an upper triangular and a lower triangular matrices
(by a kind of Gaussian elemination).
 The original matrix is invertible whenever the diagonal entries of
its two triangular factors are.
 In the situation at hand, the latter is guaranteed by the condition
that the two original morphisms~$f_i$ are
covers~\cite[Remark~1.4.2]{Xu}.
 The proof of part~(b) is similar.
\end{proof}

\begin{thm} \label{enochs-covers-envelopes}
\textup{(a)} Let\/ $\F\subset R\modl$ be a full subcategory closed
under filtered inductive limits.
 Suppose that an $R$\+module $M$ admits an\/ $\F$\+precover.
 Then $M$ also admits an\/ $\F$\+cover. \par
\textup{(b)} Let\/ $\F\subset R\modl$ be a full subcategory closed
under extensions and filtered inductive limits.
 Set\/ $\C=\F^\perp$.
 Suppose that an $R$\+module $M$ admits a special\/
$\C$\+preenvelope with the cokernel belonging to\/~$\F$.
 Then $M$ also admits a\/ $\C$\+envelope.
\end{thm}

\begin{proof} 
 These two results are due to Enochs.
 Part~(a) is~\cite[Theorem~2.2.8]{Xu} or~\cite[Theorem~1.2]{Ba}.
 Part~(b) is easily obtained from~\cite[Theorem~2.2.6]{Xu}.
 For generalizations, see~\cite[Theorem~2.7 or Corollary~4.17]{PR}
in the case of part~(a), and~\cite[Corollary~4.18 or Remark~4.19]{PR}
in the case of part~(b).
\end{proof}

 The following theorem is due to El~Bashir~\cite[Theorem~3.2]{Ba}
(for a generalization, see~\cite[Theorem~2.5 and
Proposition~2.6]{PR}).

\begin{thm} \label{el-bashir}
 Let\/ $\F\subset R\modl$ be a full subcategory closed under direct
sums and filtered inductive limits.
 Suppose that there is a set (not a proper class) of modules\/
$\sS\subset\F$ such that every module from\/ $\F$ is a filtered
inductive limit of modules from\/~$\sS$.
 Then every $R$\+module has an\/ $\F$\+cover. \qed
\end{thm}

\begin{cor} \label{flat-covers-cotorsion-envelopes-cor}
\textup{(a)} For any associative ring $R$, any left $R$\+module has
a flat cover. \par
\textup{(b)} For any associative ring $R$, any left $R$\+module has
a cotorsion envelope.
\end{cor}

\begin{proof}
 Since the class of all flat $R$\+modules $\F=R\modl_\fl$ is closed
under filtered inductive limits and consists precisely of
the filtered inductive limits of finitely generated free modules, 
Theorem~\ref{el-bashir} applies, proving part~(a).
 In view Lemma~\ref{cover-wakamatsu}(b) and Lemma~\ref{salce-lemma},
this provides another proof of completeness of the flat cotorsion
theory~\cite[Section~3]{BBE}.
 Conversely, one can deduce part~(a) from
Theorem~\ref{flat-cotorsion-theory-complete} using
Lemma~\ref{cover-wakamatsu}(a) and
Theorem~\ref{enochs-covers-envelopes}(a).
 Part~(b) of the corollary follows from completeness of the flat
cotorsion theory by means of Theorem~\ref{enochs-covers-envelopes}(b),
since the class of all flat left $R$\+modules is closed under
filtered inductive limits.
\end{proof}

 Now we return to our usual setting of a commutative ring~$R$.
 An $R$\+module $C$ is said to be \emph{contraadjusted} if it is
$s$\+contraadjusted for every $s\in R$.
 An $R$\+module $F$ is called \emph{very flat} if it is a direct
summand of a transfinitely iterated extension (in the sense of
the inductive limit) of $R$\+modules of the form $R[s^{-1}]$,
\ $s\in R$.
 We denote the class of all contraadjusted $R$\+modules by
$R\modl_\ctaa$ and the class of all very flat $R$\+modules
by $R\modl_\vfl\subset R\modl$.

\begin{cor} \label{very-flat-cotorsion-theory}
 For any commutative ring~$R$, the pair of full subcategories
$(R\modl_\vfl,\>R\modl_\ctaa)$ is a hereditary complete cotorsion
theory in $R\modl$.
\end{cor}

\begin{proof}
 This is the result of~\cite[Section~1.1]{Pcosh}.
 The assertion that $(R\modl_\vfl,\>R\modl_\ctaa)$ is a complete
cotorsion theory is provided by
Theorem~\ref{eklof-trlifaj} applied to the set of $R$\+modules
$\sS=\{\,R[s^{-1}]\mid s\in R\,\}$.
 Furthermore, any cotorsion theory generated by a class of
modules of projective dimension~$\le1$ is hereditary.
\end{proof}

 The pair of full subcategories $(R\modl_\vfl,\>R\modl_\ctaa)$ in
$R\modl$ is called the \emph{very flat cotorsion theory}.
 According to Corollary~\ref{very-flat-cotorsion-theory}, any
$R$\+module has a special very flat precover and a special
contraadjusted preenvelope.
 On the other hand, it is proved in the paper~\cite{ST} that for
any Noetherian commutative ring $R$ with infinite spectrum there
exist an $R$\+module having \emph{no} very flat cover and
an $R$\+module having \emph{no} contraadjusted envelope.

 For a Noetherian domain $R$ with finite spectrum, it is shown
in~\cite[Lemma~2.13]{ST} that the very flat cotorsion theory in
$R\modl$ coincides with the flat cotorsion theory, $R\modl_\vfl=
R\modl_\fl$ and $R\modl_\ctaa=R\modl_\cot$.
 In Section~\ref{krull-dim-1-secn} below, we will extend this
result to all Noetherian commutative rings with finite spectrum.
 On the other hand, for a von~Neumann regular commutative ring $R$,
every very flat $R$\+module is projective and every $R$\+module
is contraadjusted~\cite[Example~2.9]{ST}.

\Section{Contramodules and Relatively Cotorsion Modules}
\label{relatively-cotorsion-secn}

 According to Example~\ref{flat-cotorsion-example} and
Lemma~\ref{hereditary-cotorsion}, a left module $C$ over
an associative ring $A$ is said to be \emph{cotorsion} if for every
flat left $A$\+module $F$ one has $\Ext^1_A(F,C)=\nobreak0$,
or equivalently, for every flat left $A$\+module $F$ one has
$\Ext^q_A(F,C)=0$ for all $q\ge\nobreak1$.
 The class of cotorsion left $A$\+modules is closed under the cokernels
of injective morphisms, extensions, and infinite products 
(see Lemma~\ref{relatively-cotorsion-lemma} below for a discussion in
the somewhat greater generality of relatively cotorsion modules).

\begin{lem} \label{restrict-scalars-cotorsion}
 For any homomorphism of associative rings $A\rarrow A'$, any cotorsion
$A'$\+module is also a cotorsion $A$\+module (in the induced
$A$\+module structure).
\end{lem}

\begin{proof}
 For any flat left $A$\+module $F$, there is a natural isomorphism
$\Ext^q_A(F,C)\simeq\Ext^q_{A'}(A'\ot_AF,\>C)$ for all $q\ge0$;
and the $A'$\+module $A'\ot_AF$ is flat.
\end{proof}

 Let $R$ be a commutative ring and $I\subset R$ be an ideal.
 An $R$\+module $C$ is said to be an \emph{$I$\+contramodule} (or
an \emph{$I$\+contramodule $R$\+module}) if $C$ is an $s$\+contramodule
for every $s\in I$ (cf.\ Theorem~\ref{ideal-contramodule-thm}).
 The full subcategory of $I$\+contramodule $R$\+modules is denoted
by $R\modl_{I\ctra}\subset R\modl$.

\begin{rem} \label{contramodules-modules-over-localizations}
 Any $I$\+contramodule $R$\+module is actually a module over
the localization $(1+I)^{-1}R$ of the ring $R$ with respect to
the multiplicative set $(1+I)=\{1+s\mid s\in I\}$.
 Indeed, for any $s\in I$, the inverse map to the action of $1-s$
in an $I$\+contramodule $R$\+module $C$ can be constructed as
$$
 (1-s)^{-1}(c)=\sum\nolimits_{n=0}^\infty s^n c.
$$
 In particular, if $\m$ is a maximal ideal in a commutative ring $R$,
then any $\m$\+contramodule $R$\+module is actually a module over
the local ring $R_\m=(R\setminus\m)^{-1}R$.
 When $I$ is an ideal in a Noetherian ring $R$, one can show
that the category of $I$\+contramodule $R$\+modules is isomorphic
to the category of $\Lambda_I(I)$\+contramodule
$\Lambda_I(R)$\+modules~\cite[Theorem~B.1.1]{Pweak}.
\end{rem}

 The aim of this section is to prove the following theorem
(cf.~\cite[Proposition~B.10.1]{Pweak}
and~\cite[Proposition~1.3.7(a)]{Pcosh}).

\begin{thm} \label{noetherian-maximal-ideal-cotorsion}
 Let $\m$ be a maximal ideal in a Noetherian ring~$R$.
 Then every\/ $\m$\+contramodule $R$\+module is a cotorsion $R$\+module.
\end{thm}

 More generally, let $A\to A'$ be a homomorphism of associative rings.
 Let us call a left $A$\+module $C$ \emph{cotorsion relative to $A'$}
(or \emph{$A\.|A'$\+cotorsion}) if $\Ext^1_A(F,C)=0$ for every
flat $A$\+module $F$ such that the left $A'$\+module $A'\ot_A F$
is projective.

\begin{lem} \label{relatively-cotorsion-lemma}
\textup{(a)} For any $A\.|A'$\+cotorsion left $A$\+module $C$ and any
flat left $A$\+module $F$ such that the left $A'$\+module $A'\ot_AF$
is projective one has\/ $\Ext_A^q(F,C)=0$ for all\/ $q\ge1$. \par
\textup{(b)} The class of all $A\.|A'$\+cotorsion left $A$\+modules is
closed under the cokernels of injective morphisms, extensions, and
infinite products in $A\modl$.
\end{lem}

\begin{proof}
 Part~(a): arguing by induction in~$q\ge1$, choose a surjective
homomorphism $P\rarrow F$ onto the $A$\+module $F$ from
a projective left $A$\+module~$P$.
 Denoting the kernel of this $A$\+module morphism by $G$, we have
a short exact sequence of flat left $A$\+modules $0\rarrow G\rarrow P
\rarrow F\rarrow 0$ and a short exact sequence of left $A'$\+modules
$0\rarrow A'\ot_A G\rarrow A'\ot_A P\rarrow A'\ot_A F\rarrow0$.
 Since the $A'$\+modules $A'\ot_AP$ and $A'\ot_AF$ are projective,
the $A'$\+module $A'\ot_AG$ is projective, too.
 On the other hand, we have $\Ext_A^q(F,C)\simeq\Ext_A^{q-1}(G,C)$
for $q\ge2$, and now the right-hand side of this isomorphism
vanishes by the induction assumption.

 Part~(b): we will only check closedness under the cokernels of
injective morphisms.
 Let $0\rarrow C\rarrow D\rarrow E\rarrow0$ be a short exact
sequence of left $A$\+modules; suppose that the modules $C$ and $D$
are $A\.|A'$\+cotorsion.
 Let $F$ be a flat left $A$\+module such that $A'\ot_A F$ is
a projective left $A'$\+module.
 Applying part~(a) for $q=2$, we conclude from the exact sequence
$\Ext_A^1(F,D)\rarrow\Ext_A^1(F,E)\rarrow\Ext_A^2(F,C)$ that
$\Ext^1_A(F,E)=0$.
 So $E$ is an $A\.|A'$\+cotorsion left $A$\+module, too.
\end{proof}

 We will obtain Theorem~\ref{noetherian-maximal-ideal-cotorsion} as
a particular case of the following result, which extends the assertion
of~\cite[Proposition~B.10.1]{Pweak} to non-Noetherian rings.

\begin{thm} \label{I-contra-relatively-cotorsion}
 Let $R$ be a commutative ring and $I\subset R$ a finitely generated
ideal.
 Then every $I$\+contramodule $R$\+module is $R\.|(R/I)$\+cotorsion
(i.~e., cotorsion relative to the quotient ring $R/I$).
\end{thm}

 The proof of Theorem~\ref{I-contra-relatively-cotorsion} occupies
the rest of this section.
 We start with the following lemma, which can be found
in~\cite[Lemma~B.10.2]{Pweak}, but the argument is standard and goes
back, at least, to the famous~\cite[Theorem~P]{Bas}.

\begin{lem} \label{idempotent-lifting-argument}
 Let $A$ be an associative ring and $J\subset A$ an ideal such that
$J^n=0$ for a certain $n\ge1$.
 Suppose that $P$ is a flat left $A$\+module such that $P/JP$ is
a projective left $A/J$\+module.
 Then the $A$\+module $P$ is projective.

 Similarly, if the $P$ is a flat $A$\+module and $P/JP$ is a free
$A/J$\+module, then $P$ is a free $A$\+module.
\end{lem}

\begin{proof}
  Let $G$ be a free $A$\+module such that $P/JP$ is a direct summand
of $G/JG$.
 Then there exists an idempotent endomorphism~$e$ of the $A/J$\+module
$G/JG$ such that the $A/J$\+module $P/JP$ is isomorphic to $e(G/JG)$.
 The functor $A/J\ot_A{-}$ taking $G$ to $G/JG$ provides an associative
ring homomorphism $\pi\:\Hom_A(G,G)\rarrow\Hom_{A/J}(G/JG,\.G/JG)$.
 Since the $A$\+module $G$ is projective, the homomorphism~$\pi$ is
surjective.
 Its kernel $I=\ker \pi=\Hom_A(G,JG)\subset\Hom_A(G,G)$ satisfies
$I^n=0$; so one can lift idempotents modulo~$I$.
 Let $f\in\Hom_A(G,G)$ be an idempotent endomorphism for which
$\pi(f)=e$.
 The projective $A$\+module $fG$ is endowed with a natural homomorphism
onto the $A/J$\+module $fG/J(fG)\simeq e(G/JG)\simeq P/JP$, which
can lifted to an $A$\+module morphism $l\:fG\rarrow P$.

 We claim that $l$~is an isomorphism.
 Indeed, let $K$ and $L$ denote its kernel and cokernel, respectively.
 Then $L/JL = 0$, since the morphism $fG\rarrow P/JP$ was surjective.
 Given that $J^n=0$, it follows that $L=0$.
 Now we have a short exact sequence of $A$\+modules $0\rarrow K
\rarrow fG\rarrow P\rarrow 0$.
 The $A$\+module $P$ being flat by assumption, the sequence remains
exact after taking the tensor product with $A/J$ over $A$; so the short
sequence $0\rarrow K/JK\rarrow fG/J(fG)\rarrow P/JP\rarrow0$ is
also exact.
 We have shown that $K/JK=0$ and it follows that $K=0$.

 To prove the second assertion, it suffices to say that when $e=1$ one
can choose $f=1$.
\end{proof}

 The following lemma is a particular case of~\cite[Lemma~B.10.3]{Pweak}.

\begin{lem} \label{dual-eklof}
 Let $A$ be an associative ring, $F$ a left $A$\+module, and
$C_1\larrow C_2\larrow C_3\larrow\dotsb$ a projective system of left
$A$\+modules.
 Assume that the induced maps\/ $\Hom_A(F,C_{n+1})\rarrow\Hom_A(F,C_n)$
are surjective for all\/ $n\ge1$ and\/ $\Ext^1_A(F,C_n)=0$ for all\/
$n\ge1$.
 Then\/ $\Ext^1_A(F,\>\varprojlim_{n\ge1} C_n)=0$.
\end{lem}

\begin{proof}
 This is almost the dual version of (the ordinal~$\omega$ particular
case of) the Eklof lemma (Lemma~\ref{eklof-lemma} above;
cf.~\cite[Proposition~18]{ET} and~\cite[Lemma~4.5]{PR}).
  Suppose we are given a short exact sequence of left $A$\+modules
$$
 0\lrarrow\varprojlim\nolimits_{n\ge1} C_n\overset{i}\lrarrow
 M\overset{p}\lrarrow F\lrarrow0.
$$
 The push-forwards with respect to the projection morphisms
$\varprojlim_{m\ge1} C_m\rarrow C_n$ provide a projective system of
short exact sequences
$$
 0\lrarrow C_n\overset{i_n}\lrarrow M_n\overset{p_n}\lrarrow F\lrarrow0.
$$
 By assumption, all the short exact sequences in this projective system
split.
 We only have to show that one can choose the splittings $t_n\:M_n
\rarrow C_n$, \ $t_ni_n=\id_{C_n}$ in a compatible way, as then
composing with the morphisms $M\rarrow M_n$ and passing to
the projective limit will provide the desired splitting
$M\rarrow\varprojlim_{n\ge1} C_n$.

 Given a splitting $t_n\:M_n\rarrow C_n$ and a splitting
$t'_{n+1}:M_{n+1}\rarrow C_{n+1}$, the difference between
the two compositions $M_{n+1}\rarrow M_n\rarrow C_n$ and
$M_{n+1}\rarrow C_{n+1}\rarrow C_n$ is a morphism $M_{n+1}\rarrow C_n$
that vanishes in the composition with $i_{n+1}\:C_{n+1}\rarrow M_{n+1}$,
and therefore factorizes through $p_{n+1}\:M_{n+1}\rarrow F$.
 We have obtained a morphism $f\:F\rarrow C_n$.
 By assumption, it can be lifted to a morphism $g\:F\rarrow C_{n+1}$.
 Adding the composition $M_{n+1}\rarrow F\rarrow C_{n+1}$ to
the splitting $t'_{n+1}$, that is replacing $t'_{n+1}$ by
$t_{n+1}=t'_{n+1}+gp_{n+1}$, provides a splitting $t_{n+1}$
compatible with~$t_n$.
 Now one can proceed by induction in~$n$.
\end{proof}

\begin{cor} \label{separ-complete-relative-cotorsion}
 Let $I$ be a finitely generated ideal in a commutative ring~$R$.
 Then every $I$\+adically separated and complete $R$\+module is
$R\.|(R/I)$\+cotorsion.
\end{cor}

\begin{proof}
 Let $F$ be a flat $R$\+module such that the $R/I$\+module $F/IF$ is
projective.
 Applying Lemma~\ref{idempotent-lifting-argument} to the ring $A=R/I^n$
with the ideal $J=I/I^n$ and the $A$\+module $F/I^nF$, we conclude
that the $R/I^n$\+module $F/I^nF$ is projective.

 Let $C$ be an $I$\+adically complete and separated $R$\+module;
so $C=\varprojlim_{n\ge1}C/I^nC$.
 It suffices to check that the conditions of Lemma~\ref{dual-eklof}
are satisfied for the ring $A=R$, the module $F$, and the projective
system of modules $(C_n=C/I^nC)_{n=1}^\infty$.
 Since $F$ is a flat $R$\+module, we have
$$
 \Ext_R^1(F,C_n)=\Ext_{R/I^n}^1(R/I^n\ot_RF,\>C_n)=0.
$$
 Finally, any morphism $F\rarrow C_n$ factorizes through
the surjection $F\rarrow F/I^{n+1}F$, providing a morphism $F/I^{n+1}F
\rarrow C_n$, which can be lifted to a morphism
$F/I^{n+1}F\rarrow C_{n+1}$.
\end{proof}

\begin{lem} \label{derived-limit-cotorsion}
 Let $C_1\larrow C_2\larrow C_3\larrow\dotsb$ be a projective system
of $I$\+adically separated and complete $R$\+modules.
 Then the projective limit\/ $\varprojlim_{n\ge1} C_n$ and the derived
projective limit\/ $\varprojlim_{n\ge1}^1 C_n$ are
$R\.|(R/I)$\+cotorsion $R$\+modules.
\end{lem}

\begin{proof}
 The $R$\+modules $\varprojlim_{n\ge1} C_n$ and
$\varprojlim_{n\ge1}^1 C_n$ are computed as, respectively, the kernel
and the cokernel of the morphism
$$
 f=\id-\mathit{shift}\:\prod\nolimits_{n=1}^\infty C_n
 \lrarrow\prod\nolimits_{n=1}^\infty C_n.
$$
 Since the class of $I$\+contramodule $R$\+modules is closed under
the kernels, cokernels, and infinite products in $R\modl$, both
$\varprojlim_n C_n$ and $\varprojlim_n^1 C_n$ are $I$\+contramodules.
 Furthermore, the $R$\+module $\varprojlim_n C_n$ is $I$\+adically
separated as a submodule of the $I$\+adically separated module
$\prod_n C_n$.
 By Corollary~\ref{separ-complete-relative-cotorsion}, we can conclude
that $\varprojlim_n C_n$ and $\prod_n C_n$ are
$R\.|(R/I)$\+cotorsion $R$\+modules.

 Now we have short exact sequences of $R$\+modules
\begin{gather*}
 0\lrarrow \varprojlim\nolimits_{n\ge1}C_n\lrarrow
 \prod\nolimits_{n=1}^\infty C_n\lrarrow \im f\lrarrow 0 \\
 0\lrarrow \im f\lrarrow\prod\nolimits_{n=1}^\infty C_n
 \lrarrow \varprojlim\nolimits_{n\ge1}^1 C_n\lrarrow0.
\end{gather*}
  According to Lemma~\ref{relatively-cotorsion-lemma}, the class of
$R\.|(R/I)$\+cotorsion $R$\+modules is closed under the cokernels
of injective morphisms.
 Thus $\im f$ and $\varprojlim_n^1 C_n$ are also
$R\.|(R/I)$\+cotorsion.
\end{proof}

\begin{proof}[Proof of Theorem~\ref{I-contra-relatively-cotorsion}]
 We will show that the $R$\+module $\Delta_I(C)$ is
$R\.|(R/I)$\+cotorsion for every $R$\+module $C$; since one has
$\Delta_I(C)=C$ for any $I$\+contramodule $R$\+module $C$,
this is sufficient.
 The argument is based on the exact
sequence from Lemma~\ref{delta-I-lambda-I}.
 The elements $s_1^n$,~\dots, $s_m^n\in R$ act in the complex
$T^\bu_n(R;s_1,\dotsc,s_m)$ by contractible endomorphisms, so
all the homology modules of the complex
$\Hom_R(T^\bu_n(R;s_1,\dotsc,s_m),C)$ are annihilated by
each of these elements.
 Hence $H_*(\Hom_R(T_n^\bu(R;s_1,\dotsc,s_m),C))$ are $I$\+adically
separated and complete $R$\+modules.
 By Lemma~\ref{derived-limit-cotorsion}, both the leftmost and
the rightmost terms of our short exact sequence are
$R\.|(R/I)$\+cotorsion $R$\+modules; and it follows that
the middle term is $R\.|(R/I)$\+cotorsion, too.
\end{proof}

\Section{Flat, Projective, and Free Contramodules}
\label{free-contramodules-secn}

 The aim of this section and the next one is to discuss Enochs'
classification of flat cotorsion modules over Noetherian
rings~\cite{En} and explain the connection with free contramodules
over Noetherian local rings, as stated in~\cite[Theorem~1.3.8]{Pcosh}.

 We start with the following application of the contramodule
Nakayama lemma.

\begin{lem} \label{tor-lambda-nakayama}
 Let $R$ be a commutative ring, $I\subset R$ a finitely generated
ideal, $C$ an $R$\+module, and $\Lambda_I(C)$ its $I$\+adic
completion.
 Assume that\/ $\Tor^R_1(R/I,\.\Lambda_I(C))=0$.
 Then the natural morphism $\Delta_I(C)\rarrow\Lambda_I(C)$ is
an isomorphism.
\end{lem}

\begin{proof}
 The morphism $f\:\Delta_I(C)\rarrow\Lambda_I(C)$ is surjective by
Lemma~\ref{delta-I-lambda-I} (the proof of
Corollary~\ref{delta-s-to-lambda-s}(a) is also applicable, and shows
additionally that $\Lambda_I(C)=\Lambda_I(\Delta_I(C))$).
 Its kernel $K=\ker f$ is an $I$\+contramodule, because both
$\Delta_I(C)$ and $\Lambda_I(C)$ are.
 Furthermore, the maps
$$
 C/I^nC\lrarrow\Delta_I(C)/I^n\Delta_I(C)\lrarrow
 \Lambda_I(C)/I^n\Lambda_I(C)
$$
are isomorphisms for any $R$\+module $C$ and any $n\ge1$
(see the proof of Theorem~\ref{lambda-adjoint} or the argument
in Remark~\ref{lambda-s-torsion-free}).
 Now from the short exact sequence $0\rarrow K\rarrow \Delta_I(C)
\rarrow\Lambda_I(C)\rarrow0$ we obtain the long exact sequence
\begin{multline*}
 \dotsb\lrarrow\Tor^R_1(R/I,\.\Lambda_I(C))\lrarrow
 K/IK \\ \lrarrow\Delta_I(C)/I\Delta_I(C)\lrarrow
 \Lambda_I(C)/I\Lambda_I(C)\lrarrow0,
\end{multline*}
and the vanishing of $\Tor^R_1(R/I,\.\Lambda_I(C))$ implies
the vanishing of $K/IK$.
 Finally, it remains to apply the Nakayama Lemma~\ref{nakayama},
extended to any finite number of variables as mentioned in
Remark~\ref{many-variables-power-summations}, in order to
conclude that $K=IK$ implies $K=0$ for an $I$\+contramodule
$R$\+module~$K$.
\end{proof}

 The following lemma is specific to Noetherian rings.

\begin{lem} \label{noetherian-separ-complete-flat}
 Let $R$ be a Noetherian commutative ring, $I\subset R$ an ideal,
and $C$ an $I$\+adically separated and complete $R$\+module.
 Assume that the $R/I^n$\+module $C/I^nC$ is flat for every $n\ge1$.
 Then the $R$\+module $C$ is flat.
\end{lem}

\begin{proof}
 It suffices to show that the functor $M\longmapsto C\ot_RM$ is
exact on the category of finitely generated $R$\+modules~$M$.
 We consider two functors on the category of finitely generated
$R$\+modules:
$$
 M\longmapsto C\ot_RM \quad\text{and}\quad
 M\longmapsto \Lambda_I(C\ot_RM)=
 \varprojlim\nolimits_{n\ge1}\bigl(C/I^nC\ot_{R/I^n}M/I^nM\bigr).
$$
 First let us check that the functor $M\longmapsto\Lambda_I(C\ot_RM)$
is exact.
 Indeed, let $0\rarrow K\rarrow L\rarrow M\rarrow0$ be a short
exact sequence of finitely generated $R$\+modules.
 Then there are short exact sequences
\begin{equation} \label{artin-rees-short-sequence}
 0\lrarrow K/(K\cap I^nL)\lrarrow L/I^nL\lrarrow M/I^nM\lrarrow0,
 \quad n\ge1.
\end{equation}
 According to the Artin--Rees lemma, there exists $m\ge0$ such
that $K\cap I^nL=I^{n-m}(K\cap I^mL)$ for all $n\ge m$.
 Consequently, one has
$$
 I^nK\.\subset\. K\cap I^nL\.\subset\. I^{n-m}K \quad
 \text{for all $n\ge m$},
$$
and therefore
$$
 \varprojlim_{n\ge1}\bigl( C\ot_R K/(K\cap I^nL)\bigr)\simeq
 \varprojlim_{n\ge1}\bigl(C\ot_R K/I^nK\bigr).
$$
 Furthermore, the sequence~\eqref{artin-rees-short-sequence} is
a short exact sequence of $R/I^n$\+modules, so flatness of
the $R/I^n$\+module $C/I^nC$ implies exactness of the sequence
\begin{equation} \label{artin-rees-tensor-c}
 0\lrarrow C\ot_R K/(K\cap I^nL)\lrarrow C\ot_R L/I^nL
 \lrarrow C\ot_R M/I^nM\lrarrow0.
\end{equation}
 The sequences~\eqref{artin-rees-tensor-c} form a projective
system of short exact sequences and termwise surjective morphisms
between them, so the passage to the projective limit over $n\ge1$
preserves exactness of~\eqref{artin-rees-tensor-c}.
 We have constructed the desired short exact sequence
$$
 0\lrarrow\varprojlim_{n\ge1}(C\ot_R K/I^nK)\lrarrow
 \varprojlim_{n\ge1}(C\ot_R L/I^nL)\lrarrow\varprojlim_{n\ge1}
 (C\ot_RM/I^nM)\lrarrow0.
$$

 Now we have a morphism of functors
\begin{equation} \label{tensor-to-lambda-tensor}
 C\ot_RM\lrarrow \Lambda_I(C\ot_RM)
\end{equation}
of the argument $M$ running over the abelian category of
finitely generated $R$\+modules.
 We claim that this is an isomorphism of functors.
 The functor in the left-hand side is right exact (preserves
cokernels), while the functor in the right-hand side is exact,
as we have just shown.
 When $M$ is a finitely generated \emph{free} $R$\+module,
the map~\eqref{tensor-to-lambda-tensor} is an isomorphism,
because the map $C\rarrow\Lambda_I(C)$ is.
 In the general case, one can present a finitely generated
$R$\+module $M$ as the cokernel of a morphism of finitely
generated free $R$\+modules $f\:G\rarrow F$, and then
$$
 C\ot_R\coker f = \coker (C\ot_R f) = \coker\Lambda_I(C\ot_Rf)
 =\Lambda_I(C\ot\coker f)
$$
because $C\ot_Rf=\Lambda_I(C\ot_Rf)$.

 Finally, since the two functors are isomorphic and the functor
$M\longmapsto\Lambda_I(C\ot_RM)$ is exact, the functor
$M\longmapsto C\ot_RM$ is exact, too.
\end{proof}

 As a corollary, we obtain the following
result~\cite[Lemma~B.9.2]{Pweak}, many versions and generalizations
of which are known by now.
 See~\cite[Proposition~C.5.4]{Pcosh} for the noncommutative
Noetherian case; or \cite[Corollary~D.1.7]{Pcosh}
and~\cite[Corollary~6.15]{PR} for far-reaching generalizations to
contramodules over topological rings.

\begin{cor} \label{flat-contra-noetherian-cor}
 Let $R$ be a Noetherian commutative ring, $I\subset R$ an ideal,
and $C$ an $I$\+contramodule $R$\+module.  Then \par
\textup{(a)} $C$ is a flat $R$\+module if and only if $C/I^nC$ is
a flat $R/I^n$\+module for every $n\ge1$; \par
\textup{(b)} whenever either of the two conditions in\/~\textup{(a)}
is satisfied, $C$ is $I$\+adically separated.
\end{cor}

\begin{proof}
 Clearly, if $C$ is a flat $R$\+module then $C/I^nC$ is a flat
$R/I^n$\+module.
 Conversely, assume that $C/I^nC$ is a flat $R/I^n$\+module for
every~$n$.
 Then the $R$\+module $\Lambda_I(C)$ satisfies the assumptions
of Lemma~\ref{noetherian-separ-complete-flat}, since
$\Lambda_I(C)/I^n\Lambda_I(C)=C/I^nC$.
 Hence $\Lambda_I(C)$ is a flat $R$\+module.
 By Lemma~\ref{tor-lambda-nakayama}, it follows that $C=\Delta_I(C)
\rarrow\Lambda_I(C)$ is an isomorphism.
 Therefore, the $R$\+module $C$ is flat and $I$\+adically separated.
\end{proof}

\begin{cor} \label{flat-delta-lambda-noetherian-cor}
 Let $R$ be a Noetherian commutative ring, $I\subset R$ an ideal,
and $F$ a flat $R$\+module.  Then \par
\textup{(a)} the natural morphism $\Delta_I(F)\rarrow\Lambda_I(F)$
is an isomorphism; \par
\textup{(b)} the $R$\+module $\Delta_I(F)=\Lambda_I(F)$ is flat.
\end{cor}

\begin{proof}
 Since $\Delta_I(F)/I^n\Delta_I(F)=F/I^nF$ is a flat $R/I^n$\+module
for every $n\ge1$, the assertions follow from
Corollary~\ref{flat-contra-noetherian-cor} applied to
the $R$\+module $C=\Delta_I(F)$.
\end{proof}

 A generalization of the assertion of 
Corollary~\ref{flat-delta-lambda-noetherian-cor}(a) to the case of
a commutative ring with a weakly proregular finitely generated ideal
can be found in~\cite[Lemma~2.5]{Pmgm}.
 A partial generalization of the assertion~(b) is briefly
discussed in~\cite[Remark~5.6]{Pmgm}.

\medskip

 Let $R$ be a commutative ring and $I\subset R$ an ideal.
 For any $R$\+module $M$ and a set $X$, we denote by $M^{(X)}$
the direct sum of $X$ copies of $M$ and by $M^X$ the direct
product of $X$ copies of~$M$.
 The $R$\+module $\Delta_I(R^{(X)})$ is called the \emph{free
$I$\+contramodule $R$\+module} generated by~$X$, because for
any $I$\+contramodule $R$\+module $C$ one has
$$
 \Hom_R(\Delta_I(R^{(X)}),C)\simeq\Hom_R(R^{(X)},C)\simeq C^X.
$$
 Free $I$\+contramodule $R$\+modules are projective objects in
the abelian category $R\modl_{I\ctra}$, and every $I$\+contramodule
$R$\+module $C$ is a quotient object of some free $I$\+contramodule
$R$\+module (e.~g., one can take $X=C$).
 It follows that an $I$\+contramodule $R$\+module is a projective
object in $R\modl_{I\ctra}$ if and only if it is a direct summand
of some free $I$\+contramodule $R$\+module.

 The following result can be found in~\cite[Corollary~B.8.2]{Pweak}
(see also~\cite[Theorem~3.4]{Yek0} and~\cite[Corollary~1.8]{PSY2}).

\begin{thm} \label{projective-contramodules-charact}
 Let $R$ be a Noetherian commutative ring, $I\subset R$ an ideal,
and $C$ an $I$\+contramodule $R$\+module.  Then the following
conditions are equivalent:
\begin{enumerate}
\renewcommand{\theenumi}{\roman{enumi}}
\item $C$ is a projective $I$\+contramodule $R$\+module;
\item $C/I^nC$ is a projective $R/I^n$\+module for every $n\ge1$;
\item $C/I^nC$ is a flat $R/I^n$\+module for every $n\ge1$ and
$C/IC$ is a projective $R/I$\+module;
\item $C$ is a flat $R$\+module and $C/IC$ is a projective
$R/I$\+module.
\end{enumerate}
 Furthermore, any projective $I$\+contramodule $R$\+module is
$I$\+adically separated.
\end{thm}

\begin{proof}
 The implications (i) $\Longrightarrow$ (ii) $\Longrightarrow$ (iii)
$\Longleftarrow$ (iv) are obvious, while (iii) $\Longrightarrow$ (ii)
is provided by Lemma~\ref{idempotent-lifting-argument} and
(iii) $\Longrightarrow$ (iv) by
Corollary~\ref{flat-contra-noetherian-cor}(a).
 Any one of the conditions~(iii)\+-(iv) implies that $C$ is
$I$\+adically separated by
Corollary~\ref{flat-contra-noetherian-cor}(b).

 It remains to prove (iv) $\Longrightarrow$ (i).
 Assuming~(iv), by Theorem~\ref{I-contra-relatively-cotorsion}
we have $\Ext^1_R(C,D)=0$ for any $I$\+contramodule $R$\+module~$D$.
 Hence any short exact sequence of $I$\+contramodule $R$\+modules
$0\rarrow D\rarrow E\rarrow C$ splits.

 Alternatively, one can use an idempotent-lifting argument similar
to the proof of Lemma~\ref{idempotent-lifting-argument} in order
to show that (iv) implies~(i).
 Let $G=R^{(X)}$ be a free $R$\+module such that $C/IC$ is a direct
summand of $G/IG$.
 Then there is an idempotent endomorphism $e\:G/IG\rarrow G/IG$
such that the $R$\+module $C/IC$ is isomorphic to $e(G/IG)$.
 Proceeding by induction in~$n$, one lifts the idempotent element
$e\in\Hom_R(G/IG,\.G/IG)$ to a compatible sequence of idempotent
elements $e_n\in\Hom_R(G/I^nG,\.G/I^nG)$.
 Passing to the projective limit, we obtain an idempotent
endomorphism $f\:\Lambda_I(G)\rarrow\Lambda_I(G)$.
 By Corollary~\ref{flat-delta-lambda-noetherian-cor}(a),
$\,\Lambda_I(G)=\Delta_I(G)$ is a free $I$\+contramodule $R$\+module.
 Set $F=f\Lambda_I(G)$; then $F$ is a projective $I$\+contramodule
$R$\+module.
 Therefore, the isomorphism $F/IF\simeq e(G/IG)\simeq C/IC$ can be
lifted to an $R$\+module homomorphism $l\:F\rarrow C$.

 Let $K$ and $L$ denote the kernel and cokernel of~$l$.
 Then $L/IL=\coker(F/IF\to C/IC)=0$, so by Lemma~\ref{nakayama}
with Remark~\ref{many-variables-power-summations}, we have $L=0$
and the morphism~$l$ is surjective.
 The $R$\+module $C$ is flat by assumption, so from the exact
sequence $0\rarrow K\rarrow F\rarrow C\rarrow0$ we obtain
the exact sequence $0\rarrow K/IK\rarrow F/IF\rarrow C/IC\rarrow0$.
 Hence $K/IK=0$.
 Applying the Nakayama Lemma~\ref{nakayama} again, we conclude
that $K=0$.
\end{proof}

 Free $I$\+contramodule $R$\+modules can be characterized in
a way similar to the above characterization of the projective ones.

\begin{thm} \label{free-contramodules-charact}
 Let $R$ be a Noetherian commutative ring, $I\subset R$ an ideal,
and $C$ an $I$\+contramodule $R$\+module.  Then the following
conditions are equivalent:
\begin{enumerate}
\renewcommand{\theenumi}{\roman{enumi}}
\item $C$ is a free $I$\+contramodule $R$\+module;
\item $C/I^nC$ is a free $R/I^n$\+module for every $n\ge1$;
\item $C/I^nC$ is a flat $R/I^n$\+module for every $n\ge1$ and
$C/IC$ is a free $R/I$\+module;
\item $C$ is a flat $R$\+module and $C/IC$ is a free
$R/I$\+module.
\end{enumerate}
\end{thm}

\begin{proof}
 The implications (i) $\Longrightarrow$ (ii) $\Longrightarrow$ (iii)
$\Longleftarrow$ (iv) are obvious, while (iii) $\Longrightarrow$ (ii)
is provided by Lemma~\ref{idempotent-lifting-argument} (the second
assertion) and (iii) $\Longrightarrow$ (iv) by
Corollary~\ref{flat-contra-noetherian-cor}(a).
 Finally, the second (longer) proof of
Theorem~\ref{projective-contramodules-charact}\,%
(iv)$\,\Longrightarrow\,$(i) above is also a proof of
(iv) $\Longrightarrow$ (i) in the present theorem
(take $e_n=1$ and $f=1$ when $e=1$).
\end{proof}

 A version of the next result can be found
in~\cite[Corollary~4.5]{Yek0}; for a generalization,
see~\cite[Lemma~1.3.2]{Pweak}.

\begin{cor} \label{max-ideal-projective-contra-free}
 Let\/ $\m$ be a maximal ideal in a Noetherian commutative ring~$R$.
 Then the classes of projective\/ $\m$\+contramodule $R$\+modules
and free\/ $\m$\+contramodule $R$\+modules coincide.
\end{cor}

\begin{proof}
 Compare Theorem~\ref{projective-contramodules-charact}(iv) and
Theorem~\ref{free-contramodules-charact}(iv).
\end{proof}

 Let $R$ be a commutative ring and $I\subset R$ be a finitely
generated ideal such that the free $I$\+contramodule $R$\+modules
are $I$\+adically separated: e.~g., $R$ is Noetherian
(Corollary~\ref{flat-delta-lambda-noetherian-cor}(a) or
the last assertion of Theorem~\ref{projective-contramodules-charact}),
or $I$ is weakly proregular (\cite[Lemma~2.5]{Pmgm}).
 Then the free $I$\+contramodule $R$\+module generated by a set $X$
can be computed as
$$
 \Delta_I(R^{(X)})=\Lambda_I(R^{(X)})=\varprojlim\nolimits_{n\ge1}
 (R/I^n)^{(X)}=\Lambda_I(R)[[X]],
$$
where $\Lambda_I(R)[[X]]$ denotes the $R$\+module of all families
of elements $u_x\in\Lambda_I(R)$, \ $x\in X$ converging to~$0$ in
the $I$\+adic (\,$=$~projective limit) topology of $\Lambda_I(R)$.
 In other words, the $R$\+module $\Lambda_I(R)[[X]]$ consists of
all the maps $X\rarrow\Lambda_I(R)$, \ $x\longmapsto u_x$ such that
for every $n\ge1$ the set of all $x\in X$ for which
$$
 u_x\notin I^n\Lambda_I(R)=\ker(\Lambda_I(R)\to R/I^n)
$$
is finite (see the proof of Theorem~\ref{lambda-adjoint}; cf.\
Example~\ref{s-power-summation-examples}\,(1)).

\Section{Free Contramodules and Flat Cotorsion Modules}
\label{flat-cotorsion-secn}

 Throughout this section, $R$ is a Noetherian commutative ring.
 The following lemma is a dual version of
Corollary~\ref{flat-delta-lambda-noetherian-cor}(b).

\begin{lem} \label{injective-gamma-noetherian-lem}
 Let $I$ be an ideal in $R$ and $K$ an injective $R$\+module.
 Then the maximal $I$\+torsion submodule\/ $\Gamma_I(K)$ of $K$
is also an injective $R$\+module.
\end{lem}

\begin{proof}
 It suffices to check that for any finitely generated $R$\+module $M$
and a submodule $N\subset M$ any $R$\+module morphism $g\:N\rarrow
\Gamma_I(K)$ can be extended to a morphism $f\:M\rarrow\Gamma_I(K)$.
 Since the $R$\+module $N$ is finitely generated, there exists
$n\ge1$ such that $g$~annihilates $I^nN$.
 By the Artin--Rees lemma, there exists $m\ge1$ such that
$N\cap I^mM\subset I^nN$.
 Then we have $N/(N\cap I^mM)\subset M/I^mM$, and the morphism~$g$
factorizes through the surjection $N\rarrow N/(N\cap I^mM)$,
providing an $R$\+module morphism $g'\:N/(N\cap I^mM)\rarrow 
\Gamma_I(K)\subset K$.
 Since $K$ is an injective $R$\+module, the morphism~$g'$ can be
extended to an $R$\+module morphism $f'\:M/I^mM\rarrow K$.
 Obviously, the image of~$f'$ is contained in
$\Gamma_I(K)\subset K$.
\end{proof}

 It follows from Lemma~\ref{injective-gamma-noetherian-lem} that
every injective object of the category $R\modl_{I\tors}$ is at the same
time an injective $R$\+module (i.~e., an injective object in
$R\modl$).

\medskip

 For any $R$\+module $M$, we denote by $E_R(M)\supset M$ an injective
envelope of the $R$\+module~$M$ (cf.\ the definition of
a $\C$\+envelope for a class of objects $\C\subset R\modl$
in Section~\ref{covers-envelopes-secn}).
 It follows from Lemma~\ref{injective-gamma-noetherian-lem} that
whenever $M$ is an $I$\+torsion $R$\+module for some ideal
$I\subset R$, the $R$\+module $E_R(M)$ is also $I$\+torsion.

 On the other hand, suppose that $M$ is an $s$\+torsion-free for
each element~$s$ from a certain multiplicative subset $S\subset R$.
 Then one has $E_R(S^{-1}M)=E_R(M)$, because $M\subset S^{-1}M$ and any
nonzero submodule of $S^{-1}M$ has a nonzero intersection with~$M$.
 For any $R$\+module $N$, any endomorphism of the $R$\+module
$E_R(N)$ restricting to an automorphism of $N$ is an automorphism
of~$E_R(N)$.
 Applying this observation to $N=S^{-1}M$, we conclude that
the elements of $S$ act by automorphisms of $E_R(M)$, that is
$E_R(M)$ is an $(S^{-1}R)$\+module.

 Furthermore, every injective $(S^{-1}R)$\+module is an injective
$R$\+module, since $S^{-1}R$ is a flat $R$\+module.
 It follows that
$$
 E_R(M)=E_R(S^{-1}M)=E_{S^{-1}R}(S^{-1}M).
$$

 In particular, let $\p\subset R$ be a prime ideal.
 Denote by
$$
 k_R(\p)=((R/\p)\setminus0)^{-1}(R/\p)=
 \bigl((R\setminus\p)^{-1}R\bigr)\big/\bigl((R\setminus\p)^{-1}\p\bigr)
$$
the residue field of~$\p$.
 Then the injective $R$\+module $E_R(R/\p)=E_R(k_R(\p))$ is
$\p$\+torsion.
 It is also a module over the local ring $R_\p=(R\setminus\p)^{-1}R$,
and in fact, an injective $R_\p$\+module isomorphic to
$E_{R_\p}(k_R(\p))$.
 In particular, if $R$ is an integral domain, then $E_R(R)=k_R(0)$
is the field of fractions of~$R$.

\begin{lem} \label{hom-indecomp-injectives-vanishing}
 Let $\p$ and $\q\subset R$ be two prime ideals.
 Then
$$
 \Hom_R(E_R(R/\p),\.E_R(R/\q))=0
 \qquad\text{if\/ \,$\p\not\subset\q$}.
$$
\end{lem}

\begin{proof}
 Let $s\in\p$ be an element not belonging to~$\q$.
 Then the $R$\+module $E_R(R/\p)$ is $s$\+torsion, while
the action of~$s$ in $E_R(R/\q)$ is invertible, so there are
no nonzero $s$\+torsion elements in $E_R(R/\q)$.
\end{proof}

 The next result is sometimes called the \emph{Matlis duality}
\cite[Corollary~4.3]{Mat0} (not to be confused with
the covariant Matlis category equivalence of~\cite[Section~3]{Mat},
\cite[Section~VIII.2]{FS}).
 We will only need the simplest finite-length version.

\begin{lem} \label{contravariant-matlis-duality}
 Let\/ $\m\subset R$ be a maximal ideal.
 Then the functor
$$
 M\longmapsto\Hom_R(M,E_R(R/\m))
$$
is an involutive auto-anti-equivalence of the category of finitely
generated\/ $\m$\+torsion $R$\+modules.
 In other words, the $R$\+module\/ $\Hom_R(M,E_R(R/\m))$ is a finitely
generated\/ $\m$\+torsion $R$\+module for every finitely
generated\/ $\m$\+torsion $R$\+module~$M$, and the natural morphism
$$
 M\lrarrow\Hom_R(\Hom_R(M,E_R(R/\m)),E_R(R/\m))
$$
is an isomorphism.
\end{lem}

\begin{proof}
 Set $E=E_R(R/\m)$.
 The quotient module $R/\m=k_R(\m)$ is a field, and the $R$\+module
$\Hom_R(R/\m,E)$ is a vector space over this field.
 This vector space is one-dimensional, since, by the definition
of an injective envelope, $R/\m$ is an $R$\+submodule in $E$ and
any nonzero $R$\+submodule in $E$ has a nonzero intersection
with this particular submodule.
 Hence $\Hom_R(\Hom_R(R/\m,E),E)$ is a one-dimensional $R/\m$\+vector
space, too.
 Since there exists an injective morphism $R/\m\rarrow E$, the natural
map $R/\m\rarrow\Hom_R(\Hom_R(R/\m,E),E)$ is injective, and
therefore an isomorphism.

 Now an $R$\+module is finitely generated $\m$\+torsion if and only if
it is a finitely iterated extension of copies of the $R$\+module~$R/\m$.
 The functor $\Hom_R({-},E)$ is exact, so it takes extensions to
extensions.
 This proves that the functor $\Hom_R({-},E)$ takes finitely generated
$\m$\+torsion $R$\+modules to finitely generated $\m$\+torsion
$R$\+modules.
 Finally, if $0\rarrow L\rarrow M\rarrow N\rarrow 0$ is a short
exact sequence of $R$\+modules such that $L\rarrow\Hom_R(\Hom_R(L,E),E)$
and $N\rarrow\Hom_R(\Hom_R(N,E),E)$ are isomorphisms, then
$M\rarrow\Hom_R(\Hom_R(M,E),E)$ is an isomorphism, too.
\end{proof}

 The following lemma is a dual version of
Corollary~\ref{max-ideal-projective-contra-free}.

\begin{lem} \label{max-ideal-injective-torsion-modules}
 Let\/ $\m\subset R$ be a maximal ideal.
 Then the injective objects of the category of\/ $\m$\+torsion
$R$\+modules are precisely the direct sums of copies of
the $R$\+module $E_R(R/\m)$.
\end{lem}

\begin{proof}
 Direct sums of copies of $E=E_R(R/\m)$ are injective $R$\+modules,
because the class of injective left modules over a left
Noetherian ring is closed under infinite direct sums.
 Conversely, for any $\m$\+torsion $R$\+module $M$ denote by ${}_\m M$
the submodule of elements annihilated by~$\m$ in~$M$.
 The argument is based on the observation that ${}_\m M=0$ implies
$M=0$ (cf.~\cite[Lemma~2.1(a)]{Prev}).
 Let $K$ be an injective $\m$\+torsion $R$\+module.
 Choose a basis indexed by a set $X$ in the $R/\m$\+vector
space~${}_\m K$, consider the direct sum $E^{(X)}$ of $X$ copies
of $E$, and extend the isomorphism ${}_\m K\simeq {}_\m E^{(X)}$ to
an $R$\+module morphism $f\:K\rarrow E^{(X)}$.
 Then $\ker f=0$, since $K$ is $\m$\+torsion and $\ker f\subset K$
does not intersect ${}_\m K$.
 Since $K$ is injective, it follows that it is a direct summand
in $E^{(X)}$; hence $_\m\coker f=0$ and $\coker f=0$.
\end{proof}

 The following classification theorem is due to Matlis~\cite{Mat0}.

\begin{thm} \label{injectives-classification}
 An $R$\+module $K$ is injective if and only if it is isomorphic to
an $R$\+module of the form
\begin{equation} \label{injective-modules}
 \bigoplus\nolimits_{\p\in\Spec R} E_R(R/\p)^{(X_\p)},
\end{equation}
where\/ $\p\longmapsto X_\p$ is a correspondence assigning some set
$X_\p$ to every prime ideal\/ $\p\subset R$.
 The cardinality of every set $X_\p$ is uniquely determined by
the $R$\+module~$K$. \qed
\end{thm}

 Recall that we can use the notation $\Lambda_\p(R_\p)=
\varprojlim_{n\ge1}R_\p/\p^nR_\p$ for the completion of
the local ring $R_\p=(R\setminus\p)^{-1}R$ of a prime ideal
$\p\in\Spec R$.
 For a maximal ideal $\m\subset R$, we may write simply
$\Lambda_\m(R)=\Lambda_\m(R_\m)$.

 The following classical result can be viewed as a simple version of
the covariant Matlis category equivalence.
 At the same time, it is one of the simplest manifestations of
the underived co-contra correpondence
phenomenon~\cite[Sections~1.2 and~3.6]{Prev},
\cite[Proposition~1.5.1]{Pmgm}.

\begin{thm} \label{noetherian-maximal-ideal-covariant-duality}
 Let\/ $\m\subset R$ be a maximal ideal.  Then the pair of functors
$$
 M\longmapsto\Hom_R(E_R(R/\m),M) \quad\text{and}\quad
 C\longmapsto E_R(R/\m)\ot_R C
$$
provides an equivalence between the additive category of injective\/
$\m$\+torsion $R$\+mod\-ules $M$ and the additive category of
projective\/ $\m$\+contramodule $R$\+modules~$C$.
\end{thm}

\begin{proof}
 As in the previous proof, we set $E=E_R(R/\m)$.
 Obviously, $M\longmapsto\Hom_R(E,M)$ and $C\longmapsto E\ot_RC$
form a pair of adjoint functors from the category of $R$\+modules
to itself.
 According to Lemma~\ref{max-ideal-injective-torsion-modules},
injective $\m$\+torsion $R$\+modules are precisely the $R$\+modules
of the form $E^{(X)}$, where $X$ is a set.
 According to Corollary~\ref{max-ideal-projective-contra-free},
projectve $\m$\+contramodule $R$\+modules are precisely
the $R$\+modules $\Lambda_\m(R)[[X]]=\Lambda_\m(R^{(X)})$,
where $X$ is again an arbitrary set.

 For every integer $n\ge1$, denote by ${}_nE\subset E$
the submodule of all elements annihilated by $\m^n$ in~$E$.
 Then ${}_nE\simeq\Hom_R(R/\m^n,E)$ is a finitely generated
$\m$\+torsion $R$\+module by Lemma~\ref{contravariant-matlis-duality},
and $E=\varinjlim_{n\ge1}{}_nE$.
 Now we have
\begin{multline*}
 E\ot_R\Lambda_\m(R)[[X]] =
 \varinjlim\nolimits_{n\ge1} \.{}_nE\ot_R\Lambda_\m(R)[[X]] \\
 = \varinjlim\nolimits_{n\ge1} \.{}_nE\ot_{R/\m^n}
 \bigl(\Lambda_\m(R)[[X]]\big/\m^n\Lambda_\m(R)[[X]]\bigr) \\
 = \varinjlim\nolimits_{n\ge1} \.{}_nE\ot_{R/\m^n} (R/\m^n)^{(X)} 
 = \varinjlim\nolimits_{n\ge1}\,({}_nE)^{(X)} = E^{(X)}.
\end{multline*}
 Conversely,
\begin{multline*}
 \Hom_R(E,E^{(X)})=
 \varprojlim\nolimits_{n\ge1}\Hom_{R/\m^n}({}_nE,\.{}_nE^{(X)})
 =\varprojlim\nolimits_{n\ge1}\Hom_{R/\m^n}({}_nE,\.{}_nE)^{(X)} \\
 =\varprojlim\nolimits_{n\ge1}\Hom_{R/\m^n}(R/\m^n,R/\m^n)^{(X)}
 =\varprojlim\nolimits_{n\ge1}\,(R/\m^n)^{(X)}=\Lambda_\m(R)[[X]],
\end{multline*}
since $\Hom_R({}_nE,\.{}_nE)=\Hom_R(R/\m^n,R/\m^n)$
by Lemma~\ref{contravariant-matlis-duality}.
\end{proof}

 The next lemma is a dual version of
Lemma~\ref{hom-indecomp-injectives-vanishing}.

\begin{lem} \label{hom-prime-ideals-contramodules-vanishing}
 Let\/ $\q\subset R$ be a prime ideal.
 Denote by $P\subset\Spec R$ the set of all prime ideals\/
$\p\subset R$ such that\/ $\p\not\supset\q$.
 Let $X_\p$, \ $\p\in P$ and $X_\q$ be some sets.
 Then
$$
 \Hom_R\bigl({\textstyle\prod_{\p\in P}\Lambda_\p(R_\p)[[X_\p]]},\>
 \Lambda_\q(R_\q)[[X_\q]]\bigr)=0.
$$
\end{lem}

\begin{proof}
 More generally, let $M_\p$ be arbitrary $R_\p$\+modules and
$C_\q$ a $\q$\+contramodule $R$\+module.
 Then we claim that
$$
 \Hom_R\bigl({\textstyle\prod_{\p\in P}M_\p},\>C_\q\bigr)=0.
$$
 Indeed, let $s_1$,~\dots, $s_m$ be a finite set of generators of
the ideal~$\q$.
 For every $j=1$,~\dots, $m$, denote by $P_j\subset\Spec R$ the set of
all prime ideals $\p\subset R$ not containing~$s_j$.
 Then $P=\bigcup_{j=1}^mP_j$, hence $\prod_{\p\in P}M_\p$ is a direct
summand of $\bigoplus_{j=1}^m\prod_{\p\in P_j}M_\p$, and it suffices to
show that for every $1\le j\le m$ one has
$\Hom_R\bigl(\prod_{\p\in P_j}M_\p,\>C_q\bigr)=0$.
 Now, the action of~$s_j$ is invertible in $M_\p$ for every
$\p\in P_j$, hence also in $\prod_{\p\in P_j}M_\p$; while $C_\q$ is
an $s_j$\+contramodule, so it contains no $s_j$\+divisible submodules.
 (Cf.~\cite[Lemma~5.1.2(a)]{Pcosh}.)
\end{proof}

 The following version of
Lemma~\ref{hom-prime-ideals-contramodules-vanishing}
will be useful in the sequel.

\begin{lem} \label{hom-maximal-ideals-contramodules}
 Let\/ $I\subset R$ be an ideal.
 Denote by $P\subset\Spec R$ the set of all maximal ideals\/
$\m\subset R$ such that $I\not\subset\m$.
 Let\/ $C_\m$, \ $\m\in P$ be some\/ $\m$\+contramodule $R$\+modules
and $C_I$ an\/ $I$\+contramodule $R$\+module.
 Then
$$
 \Hom_R\bigl({\textstyle\prod_{\m\in P}C_\m},\>C_I\bigr)=0.
$$
\end{lem}

\begin{proof}
 Following the argument in the proof of 
Lemma~\ref{hom-prime-ideals-contramodules-vanishing},
we only have to check that $C_\m$ is a module over the local ring
$R_\m$ for every maximal ideal $\m\in P$.
 This is explained in
Remark~\ref{contramodules-modules-over-localizations}.
\end{proof}

 The next lemma is a classical result.

\begin{lem} \label{hom-cotorsion-flat}
 Let $N$ and $K$ be $R$\+modules.  Then\par
\textup{(a)} if $K$ is an injective $R$\+module, then\/
$\Hom_R(N,K)$ is a cotorsion $R$\+module; \par
\textup{(b)} if both $N$ and $K$ are injective $R$\+modules,
then\/ $\Hom_R(N,K)$ is a flat cotorsion $R$\+module.
\end{lem}

\begin{proof}
 Part~(a): for any left $R$\+module $F$, one has
$\Ext^q_R(F,\.\Hom_R(N,K))\simeq\Hom_R(\Tor^R_q(N,F),\.K)$,
and the right-hand side vanishes for $q>0$ when $F$ is flat.
 Part~(b): for any finitely generated $R$\+module $M$, one has
$M\ot_R\Hom_R(N,K)\simeq\Hom_R(\Hom_R(M,N),\.K)$, and the functor
in the right-hand side is exact.
\end{proof}

 Finally, we come to the following classification theorem due to
Enochs~\cite{En} (see also~\cite[Theorem~1.3.8]{Pcosh}).

\begin{thm} \label{flat-cotorsion-classification}
 An $R$\+module $C$ is flat and cotorsion if and only if it is
isomorphic to an $R$\+module of the form
\begin{equation} \label{flat-cotorsion-modules}
 \prod\nolimits_{\p\in\Spec R} \Lambda_\p(R_\p)[[X_\p]],
\end{equation}
where\/ $\p\longmapsto X_\p$ is a correspondence assigning some set
$X_\p$ to every prime ideal\/ $\p\subset R$.
 The cardinality of the set of generators $X_\p$ of a free\/
$(R_\p\p)$\+contramodule $R_\p$\+module $\Lambda_\p(R_\p)[[X_\p]]$
is uniquely determined by the $R$\+module~$C$.
\end{thm}

\begin{proof}[Brief sketch of proof]
 $R_\p\p$ is the maximal ideal of a local ring $R_\p$, so by
Theorem~\ref{noetherian-maximal-ideal-cotorsion} all
$R_\p\p$\+contramodule $R_\p$\+modules are cotorsion $R_\p$\+modules,
and by Lemma~\ref{restrict-scalars-cotorsion} they are also
cotorsion $R$\+modules.
 By Corollary~\ref{flat-delta-lambda-noetherian-cor}(b) or
Theorem~\ref{free-contramodules-charact}\,(i)$\Longleftrightarrow$(iv),
all free $(R_\p\p)$\+contramodule $R_\p$\+modules are flat
$R_\p$\+modules, hence also flat $R$\+modules (as $R_\p$ is
a flat $R$\+module).
 Alternatively, by (the proof of)
Theorem~\ref{noetherian-maximal-ideal-covariant-duality} one has
$$
 \Lambda_\p(R_\p)[[X]]\simeq
 \Hom_R\bigl(E_R(R/\p),\.E_R(R/\p)^{(X)}\bigr),
$$
and the right-hand side is a flat cotorsion $R$\+module by
Lemma~\ref{hom-cotorsion-flat}.

 Conversely, Enochs proves in~\cite{En} that every flat cotorsion
$R$\+module is a direct summand of an $R$\+module of
the form $\Hom_R(E,E')$, where $E$ and $E'$ are injective
$R$\+modules.
 He then proceeds to compute the $R$\+module $\Hom_R(E,E')$ using
the classification of injective
$R$\+modules~\eqref{injective-modules}, and shows that it has
the form~\eqref{flat-cotorsion-modules}.
 Finally, one can use, e.~g.,
Lemma~\ref{hom-prime-ideals-contramodules-vanishing} together with
Corollary~\ref{max-ideal-projective-contra-free} in order to
check that direct summands of $R$\+modules of
the form~\eqref{flat-cotorsion-modules} also have
the form~\eqref{flat-cotorsion-modules}, and that the factors in
the direct product~\eqref{flat-cotorsion-modules} can be recovered
from the product.
 (Cf.~\cite[Theorem~5.1.1]{Pcosh}.)
\end{proof}

\begin{rem}
 Matlis' classification of injective modules over Noetherian rings
was used by Hartshorne in his theory of injective quasi-coherent
sheaves over locally Noetherian schemes~\cite[\S\,II.7]{Hart}.
 Analogously, there is a theory of projective locally cotorsion
contraherent cosheaves over locally Noetherian
schemes~\cite[Section~5.1]{Pcosh} based on Enochs' classification
of flat cotorsion modules over Noetherian rings.
\end{rem}

\Section{Cotorsion and Contraadjusted Abelian Groups}
\label{abelian-groups-secn}

 We start with presenting several examples of flat covers and
cotorsion envelopes in the category $\Ab=\Z\modl$, before proceeding
to describe cotorsion abelian groups and discuss contraadjusted
abelian groups.
 First of all, we recall that an abelian group is flat if and
only if it is torsion-free.

\begin{ex} \label{cyclic-group-cover-example}
 Let $m\ge 2$ be a natural number.
 Then in the short exact sequence
$$
 0\lrarrow\bigoplus\nolimits_{p|m}\Z_p\overset m\lrarrow
 \bigoplus\nolimits_{p|m}\Z_p\lrarrow \Z/m\Z\lrarrow0
$$
the middle term is a flat $\Z$\+module, and the leftmost term
is a cotorsion $\Z$\+module.
 So it is a short exact sequence of the type~\eqref{special-precover}
for the group $\Z/m\Z$ (with respect to the flat cotorsion theory
in $\Z\modl$).
 Here the direct sums are taken over the finite set of all
the prime numbers $p$ dividing~$m$, and the leftmost
arrow acts by multiplication with~$m$ on the groups
of $p$\+adic integers~$\Z_p$.

 Indeed, the group $\bigoplus_{p|m}\Z_p$ is torsion-free, so it is
a flat $\Z$\+module.
 To show that the groups $\Z_p$ are cotorsion, one can apply
Theorem~\ref{noetherian-maximal-ideal-cotorsion}.
 Alternatively, it suffices to notice that $\Z_p=\Hom_\Z(\Q_p/\Z_p,\.
\Q_p/\Z_p)$, where the group $\Q_p/\Z_p$ is injective (because it
is divisible, see below), and use Lemma~\ref{hom-cotorsion-flat}(a).

 Moreover, the map $f\:\bigoplus_{p|m}\Z_p\to\Z/m\Z$ is a flat cover
of the cyclic group $\Z/m\Z$.
 Indeed, we have seen that $f$~is a special flat precover.
 By Lemma~\ref{cover-wakamatsu}(a), it follows that $f$~is a flat
precover.
 So it remains to check that any endomorphism $u\:\bigoplus_{p|m}
\Z_p\rarrow\bigoplus_{p|m}\Z_p$ for which $fu=f$ is an automorphism.

 One easily computes that $\Hom_\Z(\Z_p,\Z_q)=0$ for $p\ne q$
(e.~g., because $\Z_p$ is $q$\+divisible and there are no
$q$\+divisible subgroups in~$\Z_q$) and $\Hom_\Z(\Z_p,\Z_p)=\Z_p$
(e.~g., because $\Z_p=\Delta_p(\Z)=\Lambda_p(\Z)$, so
$\Hom_\Z(\Z_p,\Z_p)\rarrow\Hom_\Z(\Z,\Z_p)$ is an isomorphism).
 Thus
$\Hom_\Z({\textstyle\bigoplus_{p|m}\Z_p},\textstyle{\bigoplus_{p|m}
\Z_p})=\textstyle\bigoplus_{p|m}\Z_p$.
 Now, $fu=f$ for $u\in\bigoplus_{p|m}\Z_p$ means that $u-1$ is
divisible by~$m$ in $\bigoplus_{p|m}\Z_p$; hence $u\equiv 1\pmod p$
for every $p$ dividing~$m$ and $u$~is invertible in
$\bigoplus_{p|m}\Z_p$.
\end{ex}

\begin{ex} \label{q-p-mod-z-p-cover-example}
 Let $p$ be a prime number.
 Then in the short exact sequence
$$
 0\lrarrow \Z_p\lrarrow \Q_p\lrarrow \Q_p/\Z_p\lrarrow0
$$
the middle term $\Q_p$ is a flat $\Z$\+module, and the leftmost term
$\Z_p$ is a cotorsion $\Z$\+module (as we have seen in
Example~\ref{cyclic-group-cover-example}).
 So it is a short exact sequence of the type~\eqref{special-precover}
for the group $\Q_p/\Z_p=\Z[p^{-1}]/\Z$.

 Moreover, the map $f\:\Q_p\rarrow\Q_p/\Z_p$ is a flat cover of
the group~$\Q_p/\Z_p$.
 Indeed, by Corollary~\ref{flat-covers-cotorsion-envelopes-cor}(a),
a flat cover of $\Q_p/\Z_p$ exists in~$\Ab$.
 Applying Lemma~\ref{cover-wakamatsu}(a,c), we conclude that it
suffices to check that $\Q_p$ has no nonzero direct summands
contained in $\ker(f)$.
 However, such a direct summand would be also a direct summand
in $\ker(f)=\Z_p$, and $\Z_p$ has no nontrivial direct summands,
because there are nontrivial idempotents in the ring
$\Z_p=\Hom_\Z(\Z_p,\Z_p)$.
 It remains to point out that $\ker(f)$ itself is not a direct
summand in~$\Q_p$.
\end{ex}

\begin{ex}
 Consider the short exact sequence
$$
 0\lrarrow\prod\nolimits_p\Z_p\lrarrow\prod\nolimits
 ^{\textstyle\prime}_p\Q_p
 \lrarrow\bigoplus\nolimits_p\Q_p/\Z_p\lrarrow0,
$$
where the product in the leftmost term and the direct sum
in the rightmost term are taken over all the prime numbers~$p$.
 The middle term is the ``restricted product'' of the groups/fields
of $p$\+adic rationals $\Q_p$; by the definition, it consists of
all the collections $(a_p\in\Q_p)_p$ such that $a_p\in\Z_p$ for
all but a finite subset of the primes~$p$.
 Alternatively, the middle term can be defined as the tensor
product $\Q\ot_\Z\prod_p\Z_p$.
 This is what is called ``the ring of finite ad\`els'' in
the algebraic number theory.

 Clearly, the middle term $\prod^{\textstyle\prime}_p\Q_p$ is a flat
$\Z$\+module.
 The leftmost term $\prod_p\Z_p$ is a cotorsion $\Z$\+module,
since we already know that every factor $\Z_p$ is cotorsion and
the class of cotorsion modules is closed under infinite products.
 So our sequence is a short exact sequence of the
type~\eqref{special-precover} for the group $\bigoplus_p\Q_p/\Z_p
=\Q/\Z$ (with respect to the flat cotorsion theory in $\Z\modl$).

 Moreover, the map $f\:\prod^{\textstyle\prime}_p\Q_p\rarrow\Q/\Z$ is
a flat cover of the group~$\Q/\Z$.
 Indeed, let $u\:\prod^{\textstyle\prime}_p\Q_p\rarrow
\prod^{\textstyle\prime}_p\Q_p$ be a group homomorphism such
that $fu=f$.
 Then $f(u-\nobreak1)=\nobreak0$, hence the morphism $(u-1)\:
\prod^{\textstyle\prime}_p\Q_p\rarrow\prod^{\textstyle\prime}_p\Q_p$
factorizes through the embedding $\prod_p\Z_p\rarrow
\prod^{\textstyle\prime}_p\Q_p$.
 However, any group homomorphism from the $\Q$\+vector space
$\prod^{\textstyle\prime}_p\Q_p$ into the group $\prod_p\Z_p$ vanishes,
as there are no divisible subgroups in $\prod_p\Z_p$.
 We have shown that $u=1$.
\end{ex}

\begin{ex} \label{z-envelope-example}
 In the short exact sequence
$$
 0\lrarrow\Z\lrarrow\prod\nolimits_p\Z_p\lrarrow\prod\nolimits_p
 \Z_p\Big/\Z\lrarrow0
$$
the middle term $\prod_p\Z_p$ is a cotorsion $\Z$\+module, and
the rightmost term $\bigl(\prod_p\Z_p\bigr)\big/\Z$ is a flat
$\Z$\+module.
 In fact, the rightmost term is a $\Q$\+vector space isomorphic
to $\bigl(\prod^{\textstyle\prime}_p\Q_p\bigr)\big/\Q$.
 So this is a short exact sequence of
the type~\eqref{special-preenvelope} for the group~$\Z$.

 Moreover, the map $g\:\Z\rarrow\prod_p\Z_p$ is a cotorsion envelope
of the group~$\Z$.
 Indeed, we have seen that it is a special preenvelope, hence
(by Lemma~\ref{envelope-wakamatsu}(a)) a preenvelope.
 According to Lemma~\ref{hom-prime-ideals-contramodules-vanishing},
we have
$$
 \Hom_\Z\bigl({\textstyle\prod_p\Z_p},\>{\textstyle\prod_p\Z_p}\bigl)
 ={\textstyle\prod_p\Hom_\Z(\Z_p,\Z_p)=\prod_p\Z_p}=
 \Hom_\Z\bigl(\Z,\>{\textstyle\prod_p\Z_p}\bigl).
$$
 Thus the equation $ug=g$ for an endomorphism
$u\:\prod_p\Z_p\rarrow\prod_p\Z_p$ implies $u=1$.
\end{ex}

\begin{ex}
 All finite abelian groups are cotorsion, e.~g., by
Lemma~\ref{hom-cotorsion-flat}(a).
 Moreover, by Lemma~\ref{idempotent-lifting-argument} all flat
modules over an Artinian commutative ring $R$ are projective.
 Hence all $R$\+modules are cotorsion.
 Applying Lemma~\ref{restrict-scalars-cotorsion}, we conclude that
all the $\Z/m\Z$\+modules are cotorsion over~$\Z$, that is every
abelian group annihilated by some integer $m\ge2$ is cotorsion.
 Hence so are all the infinite products of such abelian groups.

 On the other hand, the groups $\bigoplus_{n=0}^\infty\Z_p$ and
$\bigoplus_{n=1}^\infty \Z/p^n\Z$ are \emph{not} cotorsion.
 We will see below in
Examples~\ref{contra-categorical-coproduct-examples}
what their cotorsion envelopes are.
\end{ex}

 An abelian group $B$ is said to be \emph{divisible} if it is
$s$\+divisible for every $0\ne s\in\Z$.
 Any abelian group $B$ has a unique maximal divisible subgroup
$B_\div$, which can be constructed as the sum of all divisible
subgroups in~$B$.

 An abelian group $B$ is said to be \emph{reduced} if it has
no nonzero divisible subgroups, i.~e., $B_\div=0$.
 For any abelian group $B$, the quotient group $B_\red=B/B_\div$
is reduced.
 It is the (unique) maximal reduced quotient group of~$B$.

 An abelian group is an injective object in $\Z\modl$ if and only
if it is divisible.
 Hence the natural short exact sequence
$$
 0\lrarrow B_\div\lrarrow B\lrarrow B_\red\lrarrow0
$$
is always (noncanonically) split.

\begin{lem} \label{hom-Q-reduced-groups}
 For any abelian group $B$, the subgroup $B_\div\subset B$ is
equal to the image of the map\/ $\Hom_\Z(\Q,B)\rarrow\Hom_\Z(\Z,B)
= B$.
 An abelian group $B$ is reduced if and only if\/ $\Hom_\Z(\Q,B)=0$.
\end{lem}

\begin{proof}
 Essentially, the claim is that for every element $b\in B_\div$
there is a homomorphism $f\:\Q\rarrow B$ such that $b=f(1)$.
 There are several ways to explain why this is true, the simplest of
them being, because $\Z$ is a countable integral domain
(for another approach, see Theorem~\ref{S-divisible}(a) below).
 E.~g., one can choose an element $b_1\in B_\div$ such that
$b=2\cdot b_1$, then an element $b_2\in B_\div$ such that $b_1=2\cdot 3
\cdot b_2$, an element $b_3\in B_\div$ such that $b_2=2\cdot 3\cdot 5
\cdot b_3$, etc.
 After all such choices have been made, one obtains the desired
homomorphism from the group
$$
 \Q=\varinjlim\, (\Z\overset 2\lrarrow\Z\overset{2\cdot3}\lrarrow\Z
 \overset{2\cdot3\cdot5}\lrarrow\Z\lrarrow\dotsb)
$$
into~$B$.
\end{proof}

 The classification of divisible abelian groups is provided by
Theorem~\ref{injectives-classification}: an abelian group $B$ is
divisible if and only if it is isomorphic to a group of the form
$$
 \Q^{(X)}\,\oplus\,\bigoplus\nolimits_p(\Q_p/\Z_p)^{(X_p)},
$$
where $X$ and $X_p$ are some sets (whose cardinalities are uniquely
determined by the group $B$) and the direct sum is taken over
all the prime numbers~$p$.

 For any abelian group $C$ and an integer $s\in\Z$, we consider
the group $\Delta_s(C)$ and the adjunction morphism
$\delta_{s,C}\:C\rarrow\Delta_s(C)$ constructed in
Theorem~\ref{delta-s-theorem}.

\begin{lem} \label{delta-abelian-groups-lemma}
 For any abelian group $C$, one has \par
\textup{(a)} $\Delta_0(C)=C$ and\/ $\delta_{0,C}=\id_C$; \par
\textup{(b)} $\Delta_1(C)=0$; \par
\textup{(c)} for any integer $s\ne0$, there is a natural isomorphism
$$
 \Delta_s(C)\simeq\bigoplus\nolimits_{p|s}\Delta_p(C),
$$
where the direct sum is taken over all the prime numbers~$p$
dividing~$s$.
 The components of the map\/ $\delta_{s,C}$ with respect to this
direct sum decomposition are equal to\/~$\delta_{p,C}$.
\end{lem}

\begin{proof}
 Part~(a): for any commutative ring $R$ and the element $s=0$, one
has $R[s^{-1}]=\nobreak0$, hence every $R$\+module is
a $0$\+contramodule and the reflector $\Delta_0\:R\modl\rarrow
R\modl_{0\ctra}$ is the identity functor.
 Part~(b): for any commutative ring $R$ and the element $s=1$, one
has $R[s^{-1}]=R$, hence the only $1$\+contramodule $R$\+module is
the zero module.

 Part~(c): according to Theorem~\ref{delta-s-theorem}(iii), we have
$$
 \Delta_s(C)=\Ext_\Z^1(\Z[s^{-1}]/\Z,\.C).
$$
 The assertion now follows from the isomorphism
$\Z[s^{-1}]/\Z = \bigoplus\nolimits_{p|s}\Z[p^{-1}]/\Z$.
\end{proof}

 For any abelian group $C$, we consider the natural group
homomorphism
$$
 \delta_{\Z,C}=(\delta_{p,C})_p\,\:
 C\lrarrow\prod\nolimits_p\Delta_p(C),
$$
where the product is taken over all the prime numbers~$p$.
 The following theorem goes back to Nunke~\cite[Theorem~7.1]{Nun}.

\begin{thm} \label{abelian-groups-nunke-thm}
 For any abelian group $C$, one has \par
\textup{(a)} $\ker(\delta_{\Z,C})=C_\div$; \par
\textup{(b)} $C$ is reduced if and only if the map\/ $\delta_{\Z,C}$
is injective; \par
\textup{(c)} $\coker(\delta_{\Z,C})$ is a\/ $\Q$\+vector space; \par
\textup{(d)} $C$ is cotorsion if and only if the map\/ $\delta_{\Z,C}$
is surjective; \par
\textup{(e)} if $C$ is reduced, then the morphism\/ $\delta_{\Z,C}$ is
a cotorsion envelope of~$C$.
\end{thm}

\begin{proof}
 From the short exact sequence $0\rarrow\Z\rarrow\Q\rarrow\Q/\Z
\rarrow0$ together with the isomorphisms $\Q/\Z=\bigoplus_p
\Z[p^{-1}]/\Z$ and $\Delta_p(C)=\Ext^1_\Z(\Z[p^{-1}]/\Z,C)$, we
obtain the long exact sequence
\begin{multline*}
 0\lrarrow\Hom_\Z(\Q/\Z,\.C)\lrarrow\Hom_\Z(\Q,C) \\ \lrarrow C
 \lrarrow\prod\nolimits_p\Delta_p(C)\lrarrow
 \Ext^1_\Z(\Q,C)\lrarrow0.
\end{multline*}
 Hence the kernel of $\delta_{\Z,C}$ is equal to the image of
the map $\Hom_\Z(\Q,C)\rarrow C$, which coincides with $C_\div$
by Lemma~\ref{hom-Q-reduced-groups}.
 This proves part~(a), and part~(b) follows trivially.
 Furthermore, the group $\coker(\delta_{\Z,C})$ is identified
with the group $\Ext^1_\Z(\Q,C)$, where the field $\Q$ acts
via its action in the first argument.
 This proves part~(c).

 If the group $C$ is cotorsion, then $\Ext^1_\Z(\Q,C)=0$, so
the map $\delta_{\Z,C}$ is surjective.
 Conversely, if $\delta_{\Z,C}$ is surjective, then we have
a short exact sequence
\begin{equation} \label{cotorsion-group-sequence}
 0\lrarrow C_\div\lrarrow C\lrarrow\prod\nolimits_p\Delta_p(C)
 \lrarrow0.
\end{equation}
 Now the groups $\Delta_p(C)$ are always cotorsion by
Theorem~\ref{noetherian-maximal-ideal-cotorsion}, the group
$C_\div$ is cotorsion since it is injective, and it follows that
$C$ is cotorsion, because the class of cotorsion modules is closed
under infinite products and extensions.

 Alternatively, if $\delta_{\Z,C}$ is surjective, then
$\Ext^1_\Z(\Q,C)=0$.
 Given a torsion-free abelian group $F$, one considers
the short exact sequence
$0\rarrow F\lrarrow\Q\ot_\Z F\rarrow \Q/\Z\ot_\Z F\rarrow0$.
 Since the category of abelian groups has homological dimension~$1$,
i.~e., $\Ext^2_\Z(A,B)=0$ for any abelian groups $A$ and $B$,
from the corresponding long exact sequence we see that the map
$$
 \Ext^1_\Z(\Q\ot_\Z F,\>C)\lrarrow\Ext^1_\Z(F,C)
$$
is surjective for any group~$C$.
 Since the group $\Q\ot_\Z F$ is a $\Q$\+vector space, that is
a direct sum of copies of $\Q$, we conclude that $\Ext^1_\Z(\Q,C)=0$
implies $\Ext^1_\Z(F,C)=0$.
 This proves part~(d).

 When the group $C$ is reduced, we have a short exact sequence
\begin{equation} \label{reduced-group-sequence}
 0\lrarrow C\lrarrow\prod\nolimits_p\Delta_p(C)\lrarrow
 \Ext^1_\Z(\Q,C)\lrarrow0,
\end{equation}
where the middle term $\prod_p\Delta_p(C)$ is cotorsion and
the rightmost term $\Ext^1_\Z(\Q,C)$ is flat.
 So the morphism $\delta_{\Z,C}$ is a special cotorsion preenvelope,
hence a cotorsion preenvelope, of the group~$C$.
 To check that it is an envelope, one computes
\begin{setlength}{\multlinegap}{0pt}
\begin{multline*}
 \Hom_\Z\bigl({\textstyle\prod_p\Delta_p(C)},\>
 {\textstyle\prod_p\Delta_p(C)}\bigr) =
 {\textstyle\prod_p\Hom_\Z(\Delta_p(C),\Delta_p(C))} \\ =
 {\textstyle\prod_p\Hom_\Z(C,\Delta_p(C))} =
 \Hom_\Z\bigl(C,\>{\textstyle\prod_p\Delta_p(C)}\bigr)
\end{multline*}
by Lemma~\ref{hom-maximal-ideals-contramodules}
(cf.\ Example~\ref{z-envelope-example}).
 Thus the equation $u\delta_{\Z,C}=\delta_{\Z,C}$ for an endomorphism
$u\:\prod_p\Delta_p(C)\rarrow\prod_p\Delta_p(C)$ implies $u=1$,
proving~(e).
\end{setlength}
\end{proof}

\begin{cor} \label{nonreduced-cotorsion-envelope-cor}
 Choosing a splitting $t_C\:C\rarrow C_\div$ of
the embedding $C_\div\rarrow C$, one can construct a cotorsion
envelope of an arbitrary abelian group $C$ as the map
$$
 (t_C,\.\delta_{\Z,C})\: C\lrarrow C_\div\,\oplus\,
 \prod\nolimits_p\Delta_p(C).
$$
\end{cor}

\begin{proof}
 Clearly, the map $(t_C,\delta_{\Z,C})$ is injective with the cokernel
isomorphic to $\Ext^1_\Z(\Q,C)$; so it is a special cotorsion
preenvelope of the group~$C$.
 To check that it is an envelope, one can use the computation of
$\Hom_\Z\bigl(\prod_p\Delta_p(C),\>\prod_p\Delta_p(C)\bigr)$ in
the above proof together with the observation that
$\Hom_\Z(C_\div,\>\prod_p\Delta_p(C))=0$.
 So endomorphisms~$u$ of the group $C_\div\oplus\prod_p\Delta_p(C)$
are represented by triangular matrices with the entries
$u_{dd}\:C_\div\rarrow C_\div$, \ $u_{cc}\:\prod_p\Delta_p(C)\rarrow
\prod_p\Delta_p(C)$, and $u_{dc}\:\prod_p\Delta_p(C)\rarrow C_\div$.
 The equation $u(t_C,\delta_{\Z,C})=(t_C,\delta_{\Z,C})$ implies
$u_{cc}=1$ and $u_{dd}=1$; and a triangular matrix with invertible
diagonal entries is invertible.

 More generally, the direct sum of any two cotorsion envelopes is
a cotorsion envelope (see Lemma~\ref{direct-sum-cover-envelope}).
\end{proof}

\begin{cor} \label{cotorsion-abelian-groups}
 An abelian group $C$ is cotorsion if and only if it is isomorphic
to a group of the form
$$
 D\,\oplus\,\prod\nolimits_p C_p,
$$
where the product is taken over all the prime numbers, $D$ is
a divisible group, and $C_p$ are $p$\+contramodule abelian groups.
 The groups $D$ and $C_p$ are uniquely determined by the group~$C$.
\end{cor}

\begin{proof}
 Follows from the short exact
sequence~\eqref{cotorsion-group-sequence} (which is even functorial,
so the groups $C_p$ and $D$ are functors of a cotorsion group~$C$)
together with Theorem~\ref{noetherian-maximal-ideal-cotorsion} and
the arguments in the (first) proof of
Theorem~\ref{abelian-groups-nunke-thm}(d).
\end{proof}

 Let $\C\subset\Ab$ denote the full subcategory of reduced cotorsion
abelian groups.

\begin{cor} \label{reduced-cotorsion-groups}
\textup{(a)} The full subcategory\/ $\C$ is closed under the kernels,
cokernels, extensions, and infinite products in\/~$\Ab$.
 In particular, $\C$ is an abelian category and its embedding\/
$\C\rarrow\Ab$ is an exact functor. \par
\textup{(b)} The functor $(C_p)_p\longmapsto\prod_pC_p$ establishes
an equivalence between the Cartesian product of the abelian
categories\/ $\Z\modl_{p\ctra}$ of $p$\+contramodule abelian groups,
taken over all the prime numbers~$p$, and the category\/~$\C$.
\end{cor}

\begin{proof}
 Part~(a): by Lemma~\ref{hom-Q-reduced-groups}, an abelian group $C$
is reduced if and only if $\Hom_\Z(\Q,C)=0$.
 We have also seen that an abelian group $C$ is cotorsion if and only
if $\Ext^1_\Z(\Q,C)=0$.
 Hence the assertions of part~(a) follow from
Theorem~\ref{ext-0-1-orthogonal}(a).
 Part~(b) claims that an abelian group $C$ is reduced cotorsion if
and only if it can be presented in the form $C=\prod_pC_p$ with
$C_p\in\Z\modl_{p\ctra}$, and
$$
 \Hom_\Z\bigl({\textstyle\prod_pB_p},\>{\textstyle\prod_pC_p}\bigr)
 =\textstyle\prod_p\Hom_\Z(B_p,C_p)
$$
for any $B_p$ and $C_p\in\Z\modl_{p\ctra}$.
 The former assertion is provided by
Corollary~\ref{cotorsion-abelian-groups}, and the latter one
by Lemma~\ref{hom-maximal-ideals-contramodules}.
\end{proof}

\begin{exs} \label{contra-categorical-coproduct-examples}
(1)~Fix a prime number~$p$, and let $B_\alpha$ be a family of
$p$\+contra\-module abelian groups indexed by some set of
indices~$\{\alpha\}$.
 Then the group
$$
 B=\bigoplus\nolimits_\alpha B_\alpha
$$
is \emph{not} a $p$\+contramodule in general, but it is
$q$\+divisible for all the primes $q\ne p$ and has no $p$\+divisible
subgroups.
 Hence one has $B_\div=0$ and $\Delta_q(B)=0$ for all $q\ne p$.
 Applying Theorem~\ref{abelian-groups-nunke-thm}(e), we conclude
that
$$
 \delta_{p,B}\:B\rarrow\Delta_p(B)
$$
is a cotorsion envelope of~$B$.

 The functor $\Delta_p$, being a left adjoint functor to the embedding
$\Z\modl_{p\ctra}\rarrow\Z\modl$, preserves categorical coproducts.
 Hence it takes the coproduct of abelian groups $B=\bigoplus_\alpha
B_\alpha$ to the coproduct of the objects $B_\alpha=\Delta_p(B_\alpha)$
taken in the category of $p$\+contramodule abelian groups
$\Z\modl_{p\ctra}$.
 We have explained that the group $\Delta_p(B)$ is
the coproduct of the $p$\+contramodule abelian groups $B_\alpha$
taken in the category of $p$\+contramodule abelian groups.

\medskip

(2)~Consider the particular case when all the groups $B_\alpha$
are copies of~$\Z_p$.
 As it was explained at the end of
Section~\ref{free-contramodules-secn}, we have
$\Delta_p(\Z_p^{(X)})=\Z_p[[X]]$.
 So $\Z_p^{(X)}\rarrow\Z_p[[X]]$ is a cotorsion envelope of
the group $\Z_p^{(X)}$.

 In particular, $\Delta_p(\bigoplus_{n=0}^\infty\Z_p)$ is
the group $C$ of all the sequences of $p$\+adic integers converging
to zero in the topology of $\Z_p$ from
Example~\ref{s-contra-adjusted-counterex}\,(1).
 So $\bigoplus_{n=0}^\infty \Z_p\rarrow C$ is a cotorsion envelope
of the group $\bigoplus_{n=0}^\infty\Z_p$.

\medskip

(3)~Now let us consider the group $B=\bigoplus_{n=1}^\infty\Z/p^n\Z$.
 It is explained in~\cite[Section~1.5]{Prev} that the coproduct
of the groups $\Z/p^n\Z$, \ $n\ge1$ in the category $\Z\modl_{p\ctra}$
is the non-$p$-separated $p$\+contramodule abelian group $C/E$
from Example~\ref{s-contra-adjusted-counterex}\,(1).
 So we have $\Delta_p(B)=C/E$, and $B\rarrow C/E$ is a cotorsion
envelope of the group~$B$.
\end{exs}

\begin{rem}
 It is probably impossible to describe cotorsion modules over
commutative rings much more complicated than~$\Z$.
 The intuition seems to be that torsion-free abelian groups are
``many'', so cotorsion groups are relatively ``few''.
 But over more complicated rings, flat modules are ``few'', so
cotorsion modules must be ``many''.
 Nevertheless, we will see in the next section that much of
the theory of this section can be extended to Noetherian rings
of Krull dimension~$1$.
\end{rem}

 It remains to say a few words about contraadjusted abelian groups.
 Here we follow~\cite[Section~4]{Sl} and~\cite[Example~5.2]{ST}.

 Clearly, an abelian group $C$ is contraadjusted if and only if
the group $C_\red$ is contraadjusted.
 The group $C_\red$ is a subgroup in the product $\prod_p\Delta_p(C)$
over all the prime numbers~$p$.
 According to Corollary~\ref{s-contraadjusted-criterion}(b),
a group $C$ is $s$\+contraadjusted if and only if the map
$\delta_{s,C}\:C\rarrow\Delta_s(C)$ is surjective for every $s\in\Z$.
 By Lemma~\ref{delta-abelian-groups-lemma}(c), we conclude that
an abelian group $C$ is contraadjusted if and only if
the composition
$$
 C_\red\lrarrow\prod\nolimits_p\Delta_p(C)\lrarrow\bigoplus
 \nolimits_{p\in P}\Delta_p(C)
$$
is surjective for any finite set of prime numbers~$P$.

\begin{exs}
(1)~The groups $\bigoplus_p\Z_p$ and $\bigoplus_p\Z/p\Z$, where
the direct sums are taken over all (or any infinite subset of)
the prime numbers~$p$, are contraadjusted, but not cotorsion.

\medskip

(2)~According to~\cite[Lemma~1.7.3 or Theorem~1.7.6]{Pcosh}
(see also~\cite[Lemma~2.3(ii)]{ST}), for any multiplicative set
$S\subset\Z$ containing infinitely many prime numbers,
the group $S^{-1}\Z$ is flat, but not very flat.
\end{exs}

\begin{rem}
 We are not aware of any explicit nontrivial example of
a short exact sequence~\eqref{special-precover}
or~\eqref{special-preenvelope} in the very flat cotorsion theory
in the category of abelian groups.
 Neither do we know anything about how the very flat contraadjusted
abelian groups might look like
(cf.\ Theorem~\ref{flat-cotorsion-classification}).
\end{rem}

\Section{Noetherian Rings of Krull Dimension~$1$}
\label{krull-dim-1-secn}

 In this section, $R$ is a Noetherian ring of Krull dimension~$\le1$.
 This means that every prime ideal in $R$ is either minimal or
maximal.
 As in any Noetherian ring, the set of all minimal prime ideals in
$R$ is finite.

 It is possible, however, that some prime ideals in $R$ are
simultaneously minimal and maximal.
 Such prime ideals correspond to isolated points of the topological
space $\Spec R$, i.~e., to Artinian ring direct summands of
the ring~$R$.
 In order to apply the theory developed in this section, one has
to mark such prime ideals in $R$ for being considered either on par
with the other minimal ideals, or on par with the other maximal ideals.
 In other words, we presume that the set of all prime ideals in $R$
has been divided into two disjoint subsets,
$$
 \Spec R = P_0\sqcup P_1,
$$
where all the prime ideals $\q\in P_0$ are minimal and all the prime
ideals $\p\in P_1$ are maximal.
 Given a prime ideal that is both minimal and maximal, one has to decide
whether to put it into $P_0$ or into $P_1$, but not into both.

 Generally speaking, the set of all zero-divisors in a Noetherian
ring is the union of the associated primes of the zero ideal.
 All the minimal primes belong to this set of associated
primes $\Ass_R(0)$, and there may be also a finite set of
nonminimal prime ideals belonging to $\Ass_R(0)$.
 We denote the intersection $P_1\cap\Ass_R(0)$
by $P_a\subset P_1$.
 In particular, all the minimal prime ideals belonging to $P_1$
are in~$P_a$.
 So the set of all zero-divisors in $R$ is
$$
 \bigcup\nolimits_{\q\in P_0}\q\,\cup\,\bigcup\nolimits_{\p\in P_a}\p,
$$
while the nilradical of $R$ is
$$
 \bigcap\nolimits_{\q\in P_0}\q\,\cap\,\bigcap\nolimits_{\p\in P_a}\p
$$
(where it, of course, suffices to intersect the minimal
prime ideals of~$R$).

 Let $S\subset R$ denote the multiplicative set
$$
 S=R\,\setminus\,\bigcup\nolimits_{\q\in P_0}\q.
$$
 All nonzero-divisors in $R$ always belong to~$S$; conversely, all
the elements of $S$ are nonzero-divisors if and only if
$P_a=\varnothing$.
 When $P_0$ is the set of all minimal prime ideals in $R$, one can
describe $S$ is the set of all elements in $R$ whose images are
nonzero-divisors in the quotient ring of $R$ by its nilradical.

 For every prime ideal $\m\in P_1$, there exists $s\in S$ such that
$s\in\m$.
 So prime ideals of the ring $S^{-1}R$ correspond bijectively to
the prime ideals $\q\in P_0$ in the ring~$R$.
 Hence the ring $S^{-1}R$ is zero-dimensional Noetherian, and
consequently Artinian.
 When $S$ consists of nonzero-divisors in $R$, or in other words, when
the map $R\rarrow S^{-1}R$ is injective, one says that \emph{$S^{-1}R$
is an Artinian classical ring of fractions of~$R$}.

 For every element $s\in S$, prime ideals of the quotient ring $R/(s)$
correspond bijectively to those prime ideals $\m\subset R$ that
contain~$s$.
 All such prime ideals in $R$ belong to~$P_1$ (in particular, they
are maximal), so the ring $R/(s)$ is Artinian, too.
 We denote the finite set of all prime ideals in $R$
containing~$s$ by $P(s)\subset P_1$.

 As usually, for any $R$\+module $M$ and a prime ideal $\p\subset R$,
we denote by $M_\p$ the localization $(R\setminus\p)^{-1}M=
R_\p\ot_R M$ of the module $M$ at the prime ideal~$\p$.
 Similarly, we use the notation $S^{-1}M=S^{-1}R\ot_RM$ for
the localization with respect to~$S$.
 We start with the following decomposition lemma
(cf.~\cite[Theorem~3.1]{Mat2}).

\begin{lem} \label{maximal-ideals-decomposition-lemma}
 Let $M$ be an $R$\+module such that $S^{-1}M=0$.
 Then the map $M\rarrow\prod_{\p\in P_1}M_\p$ is injective with
the image coinciding with\/ $\bigoplus_{\p\in P_1}M_\p\subset
\prod_{\p\in P_1}M_\p$, so we have a natural isomorphism
$$
 M\simeq\bigoplus\nolimits_{\p\in P_1} M_\p.
$$
 The $R$\+module $M_\p$ is\/ $\p$\+torsion
(in the sense of the definition in
Section~\ref{functor-delta-I-second-secn}).
\end{lem}

\begin{proof}
 By assumption, for every element $m\in M$ there exists $s\in S$
such that $sm=0$.
 We have a finite set $P(s)\subset P_1$ of all prime ideals
$\p\subset R$ such that $s\in\p$.
 For every $\p\in P_1\setminus P(s)$, the image of the element~$m$
is zero in $M_\p$.
 Hence the image of the map $M\rarrow\prod_{\p\in P_1}M_\p$ is
contained in $\bigoplus_{\p\in P_1}M_\p$.

 Furthermore, if $m\ne0$ then there exists a maximal ideal
$\p\in P_1$ containing the annihilator ideal $\Ann_R(m)\subset R$
of the element~$m$.
 The image of the element $m$ in $M_\p$ is nonzero.
 Hence the map $M\rarrow\bigoplus_{\p\in P_1}M_\p$ is injective.

 In order to check that this map is surjective, it suffices to
consider the case of a finitely generated module~$M$.
 Then the annihilator ideal $I$ of $M$ has a nonempty intersection
with~$S$.
 Hence $I$ is not contained in any prime ideal $\q\in P_0$,
and the prime ideals of the ring $R/I$ correspond to a finite
subset of prime ideals $P(I)\subset P_1$ in~$R$.
 The ring $R/I$ is Artinian, so it is a finite direct sum of
the Artinian local rings indexed by the maximal ideals of $R/I$,
that is $R/I=\bigoplus_{\p\in P(I)}R_\p/R_\p I$.
 Thus every $R/I$\+module $M$ is the direct sum of the
$R_\p/R_\p I$\+modules $M_\p$.

 Finally, we have $S^{-1}M_\p=0$, hence for every $m\in M_\p$
there exists $s\in S$ such that $sm=0$.
 Clearly, $s\in\p$ if $m\ne 0$.
 The quotient ring $R_\p/R_\p s$ is an Artinian local ring with
the maximal ideal $R_\p\p/R_\p s$, so every module over it is
$\p$\+torsion.
\end{proof}

 Denote by $K^\bu=K^\bu_R$ the two-term complex of $R$\+modules
$R\rarrow S^{-1}R$ with the term $R$ placed in the cohomological
degree~$-1$ and the term $S^{-1}R$ in the cohomological degree~$0$.
 When $S^{-1}R$ is an Artinian classical ring of fractions of $R$,
one can use the quotient module $(S^{-1}R)/R$ in lieu of
the complex~$K^\bu$.

 The following theorem is a key technical result.

\begin{thm} \label{complex-k-decomposition}
 The complex $K^\bu$ is naturally isomorphic to the direct sum\/
$\bigoplus_{\p\in P_1}K^\bu_\p$ in the derived category of
complexes of $R$\+modules\/ $\D^\b(R\modl)$.
\end{thm}

\begin{proof}
 Let $H\subset R$ denote the kernel of the map $R\rarrow S^{-1}R$.
 Then $S^{-1}H=0$ and $H$ is a finitely generated $R$\+module, so
$H$ is isomorphic to the direct sum of its localizations $H_\p$
over some finite set of prime ideals $\p\in P_1$ in~$R$.
 In fact, for a prime ideal $\p\in P_1$ one has $H_\p\ne0$ if
and only if $\p\in P_a=\Ass_R(0)\cap P_1$, so
$$
 H=\bigoplus\nolimits_{\p\in P_a} H_\p.
$$
 Besides, let $G$ denote the cokernel of the map $R\rarrow S^{-1}R$.
 Then, of course, we have $S^{-1}G=0$.
 The assertion of the theorem now reduces to the next lemma.
\end{proof}

\begin{lem} \label{two-term-complexes-decomposition}
 Let $f\:M\rarrow N$ be a two-term complex of $R$\+modules such
that $S^{-1}\ker(f)=0=S^{-1}\coker(f)$ and\/
$\ker(f)_\p=0$ for all but a finite set of ideals\/ $\p\in P_1$.
 Then the complex $M\rarrow N$ is naturally isomorphic to
the complex\/ $\bigoplus_{\p\in P_1}M_\p\rarrow\bigoplus_{p\in P_1} N_\p$
in the derived category\/ $\D^\b(R\modl)$.
\end{lem}

\begin{proof}
 Denote by $N^+$ the fibered product of the pair of morphisms
$\prod_\p N_\p\rarrow\prod_\p\coker(f_\p)$ and
$\bigoplus_\p\coker(f_\p)\rarrow\prod_\p\coker(f_\p)$.
 In other words, $N^+\subset\prod_\p N_\p$ consists of all
the collections of elements $(n_\p\in N_\p)_{\p\in P_1}$ such that
$n_\p$ belongs to the image of the morphism $M_\p\rarrow N_\p$
for all but a finite subset of the ideals $\p\in P_1$.
 Then we have a pair of morphisms of two-term complexes of
$R$\+modules
$$\textstyle
 (M\to N)\lrarrow\bigl(\prod_{\p\in P_1} M_\p\to N^+\bigr)
 \llarrow
 \bigl(\bigoplus_{\p\in P_1}M_\p\to\bigoplus_{\p\in P_1}N_\p\bigr).
$$
 Here a natural map $N\rarrow N^+$ exists, because the image
of the composition $N\rarrow\coker(f)\rarrow\prod_\p\coker(f_\p)$
is contained in $\bigoplus_\p\coker(f_\p)$.

 We claim that both these morphisms of complexes are
quasi-isomorphisms.
 Indeed, the kernel of the morphism $\prod_\p M_\p\rarrow N^+$
is equal to $\prod_\p\ker(f_\p)$, which coincides with
$\bigoplus_\p\ker(f_\p)$ by assumption.
 On the other hand, the cokernel of the morphism $\prod_\p M_\p
\rarrow N^+$ is equal to $\bigoplus_\p\coker(f_\p)$ by
construction.
\end{proof}

\begin{rem} \label{derived-direct-sum-decomposition-remark}
 A much more general version of Theorem~\ref{complex-k-decomposition}
can be obtained by combining
Lemma~\ref{maximal-ideals-decomposition-lemma} with the result
of~\cite[Theorem~6.6(a)]{PMat}.
 In fact, any complex $C^\bu\in\D(R\modl)$ for which the complex
$S^{-1}C^\bu$ is acyclic is naturally isomorphic to the direct sum
$\bigoplus_{\p\in P_1} C^\bu_\p$ as an object of $\D(R\modl)$.
 Indeed, according to~\cite[Theorem~6.6(a) and Remark~6.8]{PMat},
the full subcategory $\D_{S\tors}(R\modl)$ of complexes with
$S$\+torsion cohomology modules in $\D(R\modl)$ is equivalent to
the derived category $\D(R\modl_{S\tors})$ of the abelian category
of $S$\+torsion $R$\+modules (where an $R$\+module $M$ is said
to be \emph{$S$\+torsion} if $S^{-1}M=0$).
 This holds for any multiplicative subset $S$ in a Noetherian
commutative ring $R$, because the $S$\+torsion in $R$ is always
bounded in the Noetherian case.
 Now, in the situation at hand, the abelian category of
$S$\+torsion $R$\+modules is equivalent to the Cartesian product
of the categories of $\p$\+torsion $R_\p$\+modules over
the prime ideals $\p\in P_1$, the equivalence being provided
by the direct sum decomposition of
Lemma~\ref{maximal-ideals-decomposition-lemma}.
 This argument was suggested to the author by the anonymous referee.
\end{rem}

\begin{lem} \label{semilocal}
 Suppose that $R$ is a semilocal ring.  Then \par
\textup{(a)} there exists an element $s\in S$ such that $s\in\m$
for every maximal ideal\/ $\m\in P_1$; \par
\textup{(b)} the radical $\sqrt{(s)}$ of the principal ideal generated
by~$s$ in $R$ is equal to the intersection\/ $\bigcap_{\m\in P_1}\m$;
\par
\textup{(c)} one has $S^{-1}R=R[s^{-1}]$.
\end{lem}

\begin{proof}
 Part~(a): as we have already mentioned, by prime avoidance for every
maximal ideal $\m\in P_1$ there exists an element $s_\m\in S\cap\m$.
 Take $s=\prod_{\m\in P_1}s_\m$.
 Part~(b): the radical $\sqrt{(s)}$ is the intersection of all
the prime ideals $\p\in\Spec R$ containing~$s$.
 Now one has $s\notin\q$ for every $\q\in P_0$ and $s\in\m$ for
every $\m\in P_1$, so the assertion follows.
 Part~(c): for every $t\in S$, we have $\sqrt{(s)}\subset
\sqrt{(t)}$.
 Hence there exists $n\ge1$ for which $s^n\in(t)$.
 Thus inverting~$s$ in $R$ also gets~$t$ inverted.
\end{proof}

\begin{cor} \label{complex-r-s-minus-1-decomposition}
 For every element $s\in S$, the complex $R\rarrow R[s^{-1}]$ is
naturally isomorphic to the direct sum\/
$\bigoplus_{\m\in P(s)}K_\m^\bu$ in the derived category\/
$\D^\b(R\modl)$.
\end{cor}

\begin{proof}
 The conditions of Lemma~\ref{two-term-complexes-decomposition}
are clearly satisfied for the complex $R\rarrow R[s^{-1}]$, so
we have a natural isomorphism
$$
 (R\rarrow R[s^{-1}])\,\simeq\,\bigoplus\nolimits_{\m\in P_1}
 (R_\m\rarrow R_\m[s^{-1}]).
$$
in $\D^\b(R\modl)$.
 Now the complex $R_\m\rarrow R_\m[s^{-1}]$ is acyclic when $s\notin\m$,
so we are reduced to a direct sum over the finite set of ideals
$\m\in P(s)$.
 Finally, for $\m\in P(s)$ we have $R_\m[s^{-1}]=S^{-1}R_\m$ by
Lemma~\ref{semilocal}(c) applied to the local ring~$R_\m$; so
the complex $R_\m\rarrow R_\m[s^{-1}]$ is isomorphic to~$K_\m^\bu$.
\end{proof}

 A finite complex of $R$\+modules $M^\bu$ is said to have
\emph{projective dimension~$\le d$} if one has
$\Hom_{\D^\b(R\modl)}(M^\bu,N[n])=0$ for all $R$\+modules $N$
and all the integers $n>d$.
 We will denote the projective dimension of a finite complex
$M^\bu$ by $\pd_RM^\bu$.
 For a nonacyclic complex $M^\bu$, one has $\pd_R M^\bu\in
\Z\cup\{\infty\}$; and the projective dimension of
an acyclic complex is equal to~$-\infty$.

 The next result is an extension of the classical theory of
\emph{Matlis domains}~\cite[Section~2]{Mat1},
\cite[Section~IV.4]{FS} to Noetherian rings of Krull dimension~$1$.

\begin{cor} \label{s-minus-1-r-projective-dimension}
 One has \par
\textup{(a)} $\pd_R(S^{-1}R)\le 1$; \par
\textup{(b)} $\pd_R K^\bu\le 1$.
\end{cor}

\begin{proof}
 Part~(a) follows from part~(b), because $\pd_RR=0$.
 To prove part~(b), in view of Theorem~\ref{complex-k-decomposition},
it suffices to show that the projective dimension of the complex of
$R$\+modules $K_\m^\bu$ does not exceed~$1$ for every ideal $\m\in P_1$.
 Choose an element $s\in S\cap\m$.
 By Corollary~\ref{complex-r-s-minus-1-decomposition}, the complex
$K_\m^\bu$ is a direct summand of the complex $R\rarrow R[s^{-1}]$
as an object of the derived category $\D^\b(R\modl)$.
 Finally, the projective dimension of $R\rarrow R[s^{-1}]$ does
not exceed~$1$, e.~g., because $\pd_RR=0$ and $\pd_RR[s^{-1}]\le1$
(cf.\ the discussion of the complex $T^\bu(R;s)$ in
Sections~\ref{functor-delta-s-first-secn}\+-%
\ref{functor-delta-I-second-secn}).
\end{proof}

 Let $A$ be a commutative ring and $T\subset A$ a multiplicative set.
 An $A$\+module $B$ is called \emph{$T$\+divisible} if it is
$t$\+divisible for every $t\in T$.
 The class of $T$\+divisible $A$\+modules is closed under
the passages to quotient objects, extensions, infinite direct sums,
and infinite products.
 An $A$\+module $B$ is said to be \emph{$T$\+reduced} if it has no
$T$\+divisible submodules.
 The class of $T$\+reduced $A$\+modules is closed under
subobjects, extensions, infinite direct sums, and infinite products.
 Clearly, every $A$\+module $B$ has a unique maximal $T$\+divisible
submodule $B_\div$, equal to the sum of all the $T$\+divisible
submodules in~$B$.
 The quotient module $B_\red=B/B_\div$ is $T$\+reduced; it is
the (unique) maximal $T$\+reduced quotient module of~$B$.

 We will be interested in $S$\+divisible $R$\+modules.
 The notation $B_\div$ and $B_\red$ for an $R$\+module $B$ will stand
for the maximal $S$\+divisible submodule and the maximal
$S$\+reduced quotient module of $B$, respectively.
 Notice that for every maximal ideal $\m\in P_1$, every
$\m$\+contramodule $R$\+module is $S$\+reduced (because
there exists $s\in S\cap\m$, and $\m$\+contramodules contain
no $s$\+divisible submodules).

 Part~(a) of the following theorem is a version of the classical theory
of \emph{$h$\+divisible modules}~\cite{Mat1}, \cite[Section~VII.2]{FS}
for Noetherian rings of Krull dimension~$1$.

\begin{thm} \label{S-divisible}
\textup{(a)} Every $S$\+divisible $R$\+module is a quotient module
of an $(S^{-1}R)$\+module. \hbadness=1250 \par
\textup{(b)} Every $S$\+divisible $R$\+module is cotorsion.
\end{thm}

\begin{proof}
 Let $B$ be an $S$\+divisible $R$\+module.
 There is a natural distinguished triangle
$$
 R\lrarrow S^{-1}R\lrarrow K^\bu\lrarrow R[1]
$$
in the derived category $\D^\b(R\modl)$.
 Applying the functor $\Hom_{\D^\b(R\modl)}({-},B[*])$, we get
a fragment of the long exact sequence
\begin{multline*}
 \dotsb\lrarrow0\lrarrow \Hom_{\D^\b(R\modl)}(K^\bu,B)\lrarrow
 \Hom_R(S^{-1}R,B)\\ \lrarrow B\lrarrow
 \Hom_{\D^\b(R\modl)}(K^\bu,B[1])\lrarrow\dotsb
\end{multline*}

 According to Theorem~\ref{complex-k-decomposition}, we have
$K^\bu\simeq\bigoplus_{\m\in P_1}K^\bu_\m$ in $\D^\b(R\modl)$.
 By the last assertion of
Lemma~\ref{maximal-ideals-decomposition-lemma}, the cohomology
modules $H^{-1}(K^\bu_\m)$ and $H^0(K^\bu_\m)$ are $\m$\+torsion.
 Applying Lemma~\ref{tor-ext-torsion-contra-lem}, we conclude
that for every $n\in\Z$ the $R$\+module
$\Hom_{\D^\b(R\modl)}(K^\bu,B[n])$ is a product of $\m$\+contramodule
$R$\+modules over the maximal ideals $\m\in P_1$.

 In particular, the $R$\+module $\Hom_{\D^\b(R\modl)}(K^\bu,B[1])$ is
$S$\+reduced.
 Since $B$ is $S$\+divisible, it follows that the morphism
$B\rarrow\Hom_{D^\b(R\modl)}(K^\bu,B[1])$ vanishes and the morphism
$\Hom_R(S^{-1}R,B)\rarrow B$ is surjective.
 We have proved part~(a).

 Furthermore, every $(S^{-1}R)$\+module is cotorsion, because
$S^{-1}R$ is an Artinian ring.
 By Lemma~\ref{restrict-scalars-cotorsion}, every $(S^{-1}R)$\+module
is also a cotorsion $R$\+module.
 By Theorem~\ref{noetherian-maximal-ideal-cotorsion},
the $R$\+module $\Hom_{\D^\b(R\modl)}(K^\bu,B)$ is cotorsion, too.
 Since the class of cotorsion $R$\+modules is closed under
the cokernels of injective morphisms, it follows that
the $R$\+module $B$ is cotorsion.
\end{proof}

 Let us call an $R$\+module $C$ \emph{weakly cotorsion} (or
\emph{Matlis cotorsion}~\cite{Mat}) if $\Ext^1_R(S^{-1}R,C)=0$.
 Since the $R$\+module $S^{-1}R$ is flat, every cotorsion $R$\+module
is weakly cotorsion.
 Part~(b) of the next theorem provides the inverse implication,
extending to Noetherian rings of Krull dimension~$1$ the theory
of \emph{almost perfect domains} of Bazzoni and
Salce~\cite[Section~4]{BS}.

\begin{thm} \label{weakly-cotorsion}
\textup{(a)} The class of weakly cotorsion $R$\+modules is closed
under quotients (that is any quotient $R$\+module of a weakly
cotorsion $R$\+module is weakly cotorsion). \par
\textup{(b)} The classes of cotorsion $R$\+modules and
weakly cotorsion $R$\+modules coincide. \par
\textup{(c)} An $R$\+module $C$ is cotorsion if and only if
the $R$\+module $C_\red$ is cotorsion.
\end{thm}

\begin{proof}
 Part~(a) follows immediately from
Corollary~\ref{s-minus-1-r-projective-dimension}(a).
 To prove part~(b), let $C$ be a weakly cotorsion $R$\+module.
 Consider another fragment of the long exact sequence from
the proof of Theorem~\ref{S-divisible}
\begin{multline*}
 \dotsb\lrarrow\Hom_R(S^{-1}R,C)\lrarrow C \\ \lrarrow
 \Hom_{\D^\b(R\modl)}(K^\bu,C[1])\lrarrow\Ext^1_R(S^{-1}R,C)
 \lrarrow\dotsb
\end{multline*}
 Since $\Ext^1_R(S^{-1}R,C)=0$, the $R$\+module $C$ is an extension
of the $R$\+module $\Hom_{\D^\b(R\modl)}(K^\bu,C)$, which is a product
of $\m$\+contramodule $R$\+modules over the maximal ideals
$\m\in P_1$, and a quotient module of an $(S^{-1}R$)\+module
$\Hom_R(S^{-1}R,C)$.
 By Theorem~\ref{noetherian-maximal-ideal-cotorsion},
every product of $\m$\+contramodule $R$\+modules is cotorsion.
 By Theorem~\ref{S-divisible}(b), every quotient $R$\+module of
an $(S^{-1}R)$\+module is cotorsion (being clearly
$S$\+divisible).
 The assertion follows.

 Part~(c) holds because the $R$\+module $C_\div$ is always cotorsion
by Theorem~\ref{S-divisible}(b), and the class of cotorsion
$R$\+modules is closed under extensions and the cokernels of
injective morphisms.
\end{proof}

\begin{rem}
 By the famous result of Raynaud--Gruson~\cite[Corollaire~II.3.3.2]{RG},
the projective dimension of any flat module over a Noetherian
commutative ring $R$ does not exceed the Krull dimension of~$R$.
 This covers the result of our
Corollary~\ref{s-minus-1-r-projective-dimension}(a), and also
implies that the class of cotorsion $R$\+modules is closed
under quotients.
 We prefer not to use the difficult result of~\cite{RG} here, but
rather to have a self-contained exposition in our generality.
 In fact, we have proved in Theorem~\ref{weakly-cotorsion}(a\+b)
that the class of cotorsion $R$\+modules is closed under quotients,
and one can easily deduce from this the assertion that
the projective dimension of every flat $R$\+module does not
exceed~$1$.
 So we have obtained an independent proof of the Krull dimension~$1$
case of the Raynaud--Gruson theorem with our methods.

 On the other hand, the assertions of Theorem~\ref{S-divisible}(a\+b),
Theorem~\ref{weakly-cotorsion}(b), and
Corollaries~\ref{cotorsion-modules-described}\+-%
\ref{reduced-cotorsion-modules} below
do not seem to follow from known results.
\end{rem}

 Now we can deduce a corollary promised at the end of
Section~\ref{covers-envelopes-secn}.

\begin{cor}
 Let $R$ be a Noetherian ring with finite spectrum.
 Then the classes of cotorsion $R$\+modules and contraadjusted
$R$\+modules coincide.
 The classes of flat $R$\+modules and very flat $R$\+modules
coincide.
\end{cor}

\begin{proof}
 As explained in~\cite[Theorem~144]{Kap} (cf.~\cite[Lemma~2.10]{ST}),
any Noetherian ring with finite spectrum has Krull dimension~$\le1$.
 Hence Theorem~\ref{weakly-cotorsion}(b) applies, and it suffices to
prove that the $R$\+module $S^{-1}R$ is very flat.
 The ring $R$ being semilocal, it remains to use
Lemma~\ref{semilocal}(a,c) in order to show that there exists
an element $s\in S$ such that $S^{-1}R=R[s^{-1}]$.
\end{proof}

\begin{cor} \label{cotorsion-modules-described}
 An $R$\+module $C$ is cotorsion if and only if it can be included
into a short exact sequence
$$
 0\lrarrow D\lrarrow C\lrarrow\prod\nolimits_{\m\in P_1} C_\m\lrarrow0
$$
where $D$ is an $S$\+divisible $R$\+module and $C_\m$ are\/
$\m$\+contramodule $R$\+modules.
 Both the short exact sequence and the direct product decomposition
of the rightmost term are uniquely defined and depend functorially
on a cotorsion $R$\+module~$C$.
\end{cor}

\begin{proof}
 We have already seen in the proof of Theorem~\ref{weakly-cotorsion}
that every cotorsion $R$\+module can be included into a such
exact sequence.
 The exact sequence is unique, because products of $\m$\+cotorsion
$R$\+modules are $S$\+reduced, so $D=C_\div$.
 The direct product decomposition is unique by
Lemma~\ref{hom-maximal-ideals-contramodules}.
\end{proof}

 We denote by $\C_R\subset R\modl$ the full subcategory of
$S$\+reduced cotorsion $R$\+modules.

\begin{cor} \label{reduced-cotorsion-modules}
\textup{(a)} The full subcategory\/ $\C_R$ is closed under the kernels,
cokernels, extensions, and infinite products in $R\modl$.
 In particular, $\C_R$ is an abelian category and its embedding\/
$\C_R\rarrow R\modl$ is an exact functor. \par
\textup{(b)} The functor\/ $(C_\m)_{\m\in P_1}\longmapsto
\prod_\m C_\m$ establishes an equivalence between the Cartesian
product of the abelian categories $R\modl_{\m\ctra}$ of\/
$\m$\+contramodule $R$\+modules, taken over the maximal ideals\/
$\m\in P_1$ of the ring $R$, and the category\/~$\C_R$.
\end{cor}

\begin{proof}
 Part~(a): by Theorem~\ref{S-divisible}(a), an $R$\+module $C$ is
reduced if and only if $\Hom_R(S^{-1}R,C)=0$.
 By Theorem~\ref{weakly-cotorsion}(b), an $R$\+module $C$ is
cotorsion if and only if $\Hom_R(S^{-1}R,C)=0$.
 It remains to apply Theorem~\ref{ext-0-1-orthogonal}(a).
 Part~(b) follows from Corollary~\ref{cotorsion-modules-described};
see the proof of Corollary~\ref{reduced-cotorsion-groups}(b)
for a discussion.
\end{proof}

\begin{lem} \label{ext-from-torsion-maximal-ideal}
 Let\/ $\m$ be a maximal ideal in a Noetherian commutative ring $A$
and let $M$ be an\/ $\m$\+torsion $A$\+module.
 Then for every $A$\+module $C$ and all\/ $n\ge0$ the natural maps
$$
 \Ext_A^n(M,C)\lrarrow\Ext_A^n(M,C_\m)\lrarrow\Ext_{A_\m}^n(M,C_\m)
$$
are isomorphisms.
\end{lem}

\begin{proof}
 First of all, $M$ is an $A_\m$\+module, so the rightmost $\Ext$
module is well-defined.
 All the $\Ext$ modules in question can be viewed as the derived
functors of the functor $\Hom$ with respect to its second argument with
the first argument $M$ fixed.
 Notice that the localization functor $C\longmapsto C_\m$ takes
injective $A$\+modules to injective $A_\m$\+modules, which are also
injective $A$\+modules.
 Hence it suffices to consider the case $n=0$.

 Now the map $\Hom_A(M,C_\m)\rarrow\Hom_{A_\m}(M,C_\m)$ is clearly
an isomorphism, and the map $\Hom_A(M,C)\rarrow\Hom_A(M,C_\m)$ is
an isomorphism because the map $\Gamma_\m(C)\rarrow\Gamma_\m(C_\m)$ is.
 To check the latter claim, one applies the functor of localization
at~$\m$ to the short exact sequence $0\rarrow\Gamma_\m(C)\rarrow
C\rarrow C/\Gamma_\m(C)\rarrow0$ and notices that $\Gamma_\m(C)$ is
an $A_\m$\+module, while the $A$\+module $C/\Gamma_\m(C)$ and
the $A_\m$\+module $(C/\Gamma_\m(C))_\m$ are $\m$\+torsion-free.
\end{proof}

\begin{thm} \label{delta-m-theorem}
 For any maximal ideal\/ $\m\in P_1$, there is an isomorphism
of functors
\begin{equation} \label{delta-m-identified}
 \Hom_{\D^\b(R\modl)}(K^\bu_\m,{-}[1])\simeq\Delta_\m\:
 R\modl\lrarrow R\modl_{\m\ctra}.
\end{equation}
 For every $R$\+module $C$, the product of the isomorphisms\/
$\Hom_{\D^\b(R\modl)}(K^\bu_\m,C[1])\simeq\Delta_\m(C)$ over
all\/ $\m\in P_1$ together with the isomorphism\/
$\bigoplus_{\m\in P_1}K^\bu_\m\simeq K^\bu$ transform the morphism
$C=\Hom_R(R,C)\rarrow\Hom_{\D^\b(R\modl)}(K^\bu,C[1])$ into
the morphism $C\rarrow\prod_\m\Delta_\m(C)$ whose components are
the adjunction morphisms\/ $\delta_{\m,C}\:C\rarrow\Delta_\m(C)$.
\end{thm}

\begin{proof}
 By Remark~\ref{contramodules-modules-over-localizations},
the category of $\m$\+contramodule $R$\+modules is a full subcategory
of the category of $R_\m$\+modules, which is a full subcategory of
the category of all $R$\+modules,
$$
 R\modl_{\m\ctra}\subset R_\m\modl\subset R\modl.
$$
 The localization functor $C\longmapsto C_\m$ is left adjoint to
the embedding $R_\m\modl\rarrow R\modl$.
 Furthermore, according to Lemma~\ref{semilocal} applied to
the local ring $R_\m$, there exists an element $s\in S$ such that
$S^{-1}R_\m=R_\m[s^{-1}]$ and the radical of the ideal $(s)\subset R_\m$
is equal to~$R_\m\m$.
 So the complex $R_\m\rarrow R_\m[s^{-1}]$ is isomorphic to~$K_\m^\bu$.
 In view of Remark~\ref{I-C-radical-ideal} (or the second proof of
Theorem~\ref{ideal-contramodule-thm} in
Section~\ref{functor-delta-I-second-secn}), the full subcategories of
$s$\+contramodule $R_\m$\+modules and $(R_\m\m)$\+contramodule
$R_\m$\+modules in $R_\m\modl$ coincide.
 Therefore, the functor left adjoint to the embedding
$R\modl_{\m\ctra}=R_\m\modl_{(R_\m\m)\ctra}\rarrow R_\m\modl$ can be
computed as
$$
 D\longmapsto \Hom_{\D^\b(R_\m\modl)}(K^\bu_\m,D[1]), \quad
 D\in R_\m\modl
$$
(see Theorem~\ref{delta-s-theorem}(iii) and
Remark~\ref{gamma-delta-zero-divisor-remark}).
 It follows that the functor~$\Delta_\m$ left adjoint to the composition
of the two embeddings of categories can be obtained as the composition
$$
 C\longmapsto\Hom_{\D^\b(R_\m\modl)}(K^\bu_\m,C_\m[1]), \quad
 C\in R\modl,
$$
and the adjunction morphism is the composition
$$
 C\lrarrow C_\m\lrarrow\Hom_{\D^\b(R_\m\modl)}(K^\bu_\m,C_\m[1]).
$$

 Now, by Lemma~\ref{ext-from-torsion-maximal-ideal} we have
$$
 \Hom_{\D^\b(R\modl)}(K^\bu_\m,C[1])\simeq
 \Hom_{\D^\b(R_\m\modl)}(K^\bu_\m,C_\m[1]),
$$
which provides the desired isomorphism of
functors~\eqref{delta-m-identified}.
 We leave it to the reader to finish the proof of the second assertion
of the theorem.
\end{proof}

 In view of the result of Theorem~\ref{delta-m-theorem},
the long exact sequence from the proofs of
Theorems~\ref{S-divisible}\+-\ref{weakly-cotorsion} takes the form
\begin{equation}
 \dotsb\lrarrow\Hom_R(S^{-1}R,C)\lrarrow C\lrarrow
 \prod_{\m\in P_1}\Delta_\m(C)\lrarrow\Ext^1_R(S^{-1}R,C)
 \lrarrow0
\end{equation}
for every $R$\+module~$C$.
 The image of the morphism $\Hom_R(S^{-1}R,C)\rarrow C$ is the submodule
$C_\div\subset C$.
 So every $S$\+reduced $R$\+module $C$ is naturally a submodule in
the product $\prod_{\m\in P_1}\Delta_\m(C)$.

\medskip

 Recall that a ring is called \emph{reduced} if it contains no nonzero
nilpotent elements.
 In our setting, if the ring $R$ is reduced, then $S^{-1}R$ is a reduced
Artinian ring, that is a direct sum of a finite number of fields.

\begin{cor}
 If the ring $R$ is reduced and the $R$\+module $C$ is $S$\+reduced,
then $(\delta_{\m,C})_{\m\in P_1}\: C\rarrow\prod_{\m\in P_1}\Delta_\m(C)$
is a cotorsion envelope of the $R$\+module~$C$.
\end{cor}

\begin{proof}
 In the assumptions of the corollary, the morphism
$\delta_{S,C}=(\delta_{\m,C})_{\m\in P_1}$ is injective and its cokernel
$\Ext^1_R(S^{-1}R,C)$ is an $(S^{-1}R)$\+module.
 Furthermore, every $(S^{-1}R)$\+module is a flat (and even projective)
$(S^{-1}R)$\+module and a flat $R$\+module.
 Hence the morphism~$\delta_{S,C}$ is a special cotorsion preenvelope.
 To prove that it is an envelope, argue as in the proof
of Theorem~\ref{abelian-groups-nunke-thm}(e): using
Lemma~\ref{hom-maximal-ideals-contramodules}, compute
\begin{setlength}{\multlinegap}{0pt}
\begin{multline*}
 \Hom_R\bigl({\textstyle\prod_\m\Delta_\m(C)},\>
 {\textstyle\prod_\m\Delta_\m(C)}\bigr) =
 {\textstyle\prod_\m\Hom_R(\Delta_\m(C),\Delta_\m(C))} \\ =
 {\textstyle\prod_\m\Hom_R(C,\Delta_\m(C))} =
 \Hom_R\bigl(C,\>{\textstyle\prod_\m\Delta_\m(C)}\bigr)
\end{multline*}
and conclude that the equation $u\delta_{S,C}=\delta_{S,C}$ for
an endomorphism $u\:\prod_\m\Delta_\m(C)\rarrow\prod_\m\Delta_\m(C)$
implies $u=1$.
\end{setlength}
\end{proof}

\begin{rem}
 It would be interesting to know how to construct cotorsion envelopes
of nonreduced modules over Noetherian domains of Krull dimension~$1$.
 The construction of Corollary~\ref{nonreduced-cotorsion-envelope-cor}
requires the embedding $C_\div\rarrow C$ to be split, so it only
seems to work for Dedekind domains.
\end{rem}

 Now we return to our usual setting of an arbitrary Noetherian ring
$R$ of Krull dimension~$1$.
 Our aim is to characterize $s$\+contraadjusted $R$\+modules.

\begin{cor} \label{delta-s-delta-m-cor}
 For any element $s\in S$ and every $R$\+module $C$, there is
a natural isomorphism
$$
 \Delta_s(C)\simeq\bigoplus\nolimits_{\m\in P(s)}\Delta_\m(C),
$$
where the direct sum is taken over the set $P(s)\subset P_1$ of all
prime ideals\/ $\m\subset R$ containing~$s$.
 The components of the map\/ $\delta_{s,C}\:C\rarrow\Delta_s(C)$ with
respect to this direct sum decomposition are equal to\/~$\delta_{\m,C}$.
\end{cor}

\begin{proof}
 Follows from Corollary~\ref{complex-r-s-minus-1-decomposition} and
Theorem~\ref{delta-m-theorem}.
\end{proof}

\begin{cor} \label{krull-dim-1-s-contraadjusted}
 Let $C$ be an $R$\+module and $s\in S$ be an element.  Then \par
\textup{(a)} an $R$\+module $C$ is $s$\+contraadjusted if and only
if the $R$\+module $C_\red$ is $s$\+contraadjusted; \par
\textup{(b)} an $R$\+module $C$ is $s$\+contraadjusted if and only if
the map $C\rarrow\bigoplus_{\m\in P(s)}\Delta_\m(C)$ is surjective.
\end{cor}

\begin{proof}
 Part~(b) follows from Corollaries~\ref{s-contraadjusted-criterion}
and~\ref{delta-s-delta-m-cor}.
 Part~(a) holds, because the $R$\+module $C_\div$ is $s$\+divisible,
hence $\Delta_s(C_\div)=0$ and $\Delta_s(C)=\Delta_s(C_\red)$, so
the map $C\rarrow\Delta_s(C)$ is surjective if and only if
the map $C_\red\rarrow\Delta_s(C_\red)$ is.
\end{proof}

\begin{cor} \label{S-contraadjusted}
 An $R$\+module $C$ is $s$\+contraadjusted for every element $s\in S$
if and only if for every finite subset $P\subset P_1$ the map
$C\rarrow\bigoplus_{\m\in P}\Delta_\m(C)$ is surjective.
\end{cor}

\begin{proof}
 Follows from Corollary~\ref{krull-dim-1-s-contraadjusted}(b),
because for every finite subset $P\subset P_1$ there exists
an element $s\in S$ such that $P\subset P(s)$.
\end{proof}

\begin{rem}
 For a Noetherian domain $R$ of Krull dimension~$1$, the result
of Corollary~\ref{S-contraadjusted} says that an $R$\+module $C$
is contraadjusted if and only if for every finite set of maximal
ideals $P\subset P_1$ in $R$ the map $C\rarrow\bigoplus_{\m\in P}
\Delta_\m(C)$ is surjective.
 Indeed, in this case $S=R\setminus\{0\}$, and every $R$\+module
is $0$\+contraadjusted.
 This extends the characterization of contraadjusted modules over
Dedekind domains provided by~\cite[Corollary~4.13]{Sl} to
Noetherian domains of Krull dimension~$1$.
 It would be interesting to extend this kind of characterization
of contraadjusted modules further to Noetherian rings of Krull
dimension~$1$, but we do not know how to do it.
\end{rem}


\bigskip
\begin{thebibliography}{99}
\smallskip

\bibitem{Bas}
 H.~Bass.
   Finitistic dimension and a homological generalization of
semi-primary rings.
\textit{Transactions of the American Math.\ Society} \textbf{95},
p.~466--488, 1960.

\bibitem{BS}
 S.~Bazzoni, L.~Salce.
   Strongly flat covers.
\textit{Journ.\ of the London Math.\ Society} \textbf{66}, \#2,
p.~276--294, 2002.

\bibitem{BBE}
 L.~Bican, R.~El Bashir, E.~Enochs.
   All modules have flat covers.
\textit{Bulletin of the London Math.\ Society} \textbf{33}, \#4,
p.~385--390, 2001.

\bibitem{Bueh}
 T.~B\"uhler.
   Exact categories.
\textit{Expositiones Math.}\ \textbf{28}, \#1, p.~1--69, 2010.
\texttt{arXiv:0811.1480 [math.HO]}

\bibitem{ET}
 P.~C.~Eklof, J.~Trlifaj.
   How to make Ext vanish.
\textit{Bulletin of the London Math.\ Society} \textbf{33}, \#1,
p.~41--51, 2001.

\bibitem{Ba}
 R.~El~Bashir.
   Covers and directed colimits.
\textit{Algebras and Representation Theory} \textbf{9}, \#5,
p.~423--430, 2006.

\bibitem{En}
 E.~Enochs.
   Flat covers and flat cotorsion modules.
\textit{Proceedings of the American Math. Society} \textbf{92}, \#2,
p.~179--184, 1984.

\bibitem{Fuc}
 L.~Fuchs.
   Infinite abelian groups, vol.~I.
Academic Press, New York--San Francisco--London, 1970.

\bibitem{FS}
 L.~Fuchs, L.~Salce.
   Modules over non-Noetherian domains.
Mathematical Surveys and Monographs 84, American Math. Society, 2001.

\bibitem{GL}
 W.~Geigle, H.~Lenzing.
   Perpendicular categories with applications to representations
and sheaves.
\textit{Journ.\ of Algebra} \textbf{144}, \#2, p.~273--343, 1991.

\bibitem{GT}
 R.~G\"obel, J.~Trlifaj.
   Approximations and endomorphism algebras of modules.
Second Revised and Extended Edition.
De Gruyter Expositions in Mathematics 41,
De Gruyter, Berlin--Boston, 2012.

\bibitem{HRS}
 D.~Happel, I.~Reiten, S.~O.~Smal\o.
   Tilting in abelian categories and quasitilted algebras.
\textit{Memoirs of the American Math.\ Society} \textbf{120},
\#575, 1996. 

\bibitem{Harr}
 D.~K.~Harrison.
   Infinite abelian groups and homological methods.
\textit{Annals of Math.}\ \textbf{69}, \#2, p.~366--391, 1959.

\bibitem{Hart}
 R.~Hartshorne.
   Residues and duality.
\textit{Lecture Notes in Math.}\ \textbf{20}, Springer, 1966.

\bibitem{Kap}
 I.~Kaplansky.
   Commutative rings.  Revised edition.
University of Chicago Press, 1974.

\bibitem{Mat0}
 E.~Matlis.
   Injective modules over Noetherian rings.
\textit{Pacific Journ.\ of Math.}\ \textbf{8}, \#3, p.~511--528, 1958.

\bibitem{Mat1}
 E.~Matlis.
   Divisible modules.
\textit{Proceedings of the American Math. Society} \textbf{11}, \#3,
p.~385--391, 1960.

\bibitem{Mat}
 E.~Matlis.
   Cotorsion modules.
\textit{Memoirs of the American Math.\ Society} \textbf{49}, 1964.

\bibitem{Mat2}
 E.~Matlis.
   Decomposable modules.
\textit{Transactions of the American Math.\ Society} \textbf{125},
\#1, p.~147--179, 1966.

\bibitem{Nun}
 R.~J.~Nunke.
   Modules of extensions over Dedekind rings.
\textit{Illinois Journ.\ of Math.} \textbf{3}, \#2, p.~222--241, 1959.

\bibitem{PSY}
 M.~Porta, L.~Shaul, A.~Yekutieli.
   On the homology of completion and torsion.
\textit{Algebras and Representation Theory} \textbf{17}, \#1, p.~31--67,
2014.  \texttt{arXiv:1010.4386 [math.AC]}.
Erratum in \textit{Algebras and Representation Theory} \textbf{18},
\#5, p.~1401--1405, 2015.  \texttt{arXiv:1506.07765 [math.AC]}

\bibitem{PSY2}
 M.~Porta, L.~Shaul, A.~Yekutieli.
   Cohomologically cofinite complexes.
\textit{Communications in Algebra} \textbf{43}, \#2, p.~597--615, 2015.
\texttt{arXiv:1208.4064 [math.AC]}

\bibitem{Psemi}
 L.~Positselski.
   Homological algebra of semimodules and semicontramodules:
Semi-infinite homological algebra of associative algebraic structures.
 Appendix~C in collaboration with D.~Rumynin; Appendix~D in
collaboration with S.~Arkhipov.
 Monografie Matematyczne vol.~70, Birkh\"auser/Springer Basel, 2010. 
xxiv+349~pp. \texttt{arXiv:0708.3398 [math.CT]}

\bibitem{Pweak}
 L.~Positselski.
   Weakly curved A${}_\infty$-algebras over a topological local ring.
Electronic preprint \texttt{arXiv:1202.2697 [math.CT]}.

\bibitem{Pcosh}
 L.~Positselski.
   Contraherent cosheaves.
Electronic preprint \texttt{arXiv:1209.2995 [math.CT]}.

\bibitem{Prev}
 L.~Positselski.
   Contramodules.
Electronic preprint \texttt{arXiv:1503.00991 [math.CT]}.

\bibitem{Pmgm}
 L.~Positselski.
   Dedualizing complexes and MGM duality.
\textit{Journ.\ of Pure and Appl.\ Algebra} \textbf{220}, \#12,
p.~3866--3909, 2016.  \texttt{arXiv:1503.05523 [math.CT]}

\bibitem{PMat}
 L.~Positselski.
   Triangulated Matlis equivalence.
Electronic preprint \texttt{arXiv:1605.08018 [math.CT]},
to appear in \textit{Journ.\ of Algebra and its Appl.}
{\hbadness=2100\par}

\bibitem{PR}
 L.~Positselski, J.~Rosick\'y.
   Covers, envelopes, and cotorsion theories in locally presentable
abelian categories and contramodule categories.
\textit{Journ.\ of Algebra} \textbf{483}, p.~83--128, 2017.
\texttt{arXiv:1512.08119 [math.CT]}

\bibitem{RG}
 M.~Raynaud, L.~Gruson.
   Crit\`eres de platitude et de projectivit\'e: Techniques
de ``platification'' d'un module.
\textit{Inventiones Math.} \textbf{13}, \#1--2, p.~1--89, 1971.

\bibitem{Sal}
 L.~Salce.
   Cotorsion theories for abelian groups.
\textit{Symposia Math.}\ \textbf{XXIII},
Academic Press, London--New York, 1979, p.~11--32.

\bibitem{Sim}
 A.-M.~Simon.
   Approximations of complete modules by complete big
Cohen--Macaulay modules over a Cohen--Macaulay local ring.
\textit{Algebras and Representation Theory} \textbf{12}, \#2--5,
p.~385--400, 2009.

\bibitem{Sl}
 A.~Sl\'avik.
   Classes of modules arising in contemporary algebraic geometry.
Master Thesis, Department of Algebra, Faculty of Mathematics
and Physics, Charles University, Prague, 2015.

\bibitem{ST}
 A.~Sl\'avik, J.~Trlifaj.
   Very flat, locally very flat, and contraadjusted modules.
\textit{Journ.\ of Pure and Appl.\ Algebra} \textbf{220}, \#12,
p.~3910--3926, 2016.  \texttt{arXiv:1601.00783 [math.AC]}

\bibitem{Sto}
 J.~\v St'ov\'\i\v cek.
   Exact model categories, approximation theory, and cohomology
of quasi-coherent sheaves.
 In: Advances in Representation Theory of Algebras, EMS Series of
Congress Reports, European Math.\ Society, Zurich, 2013, 297--367.
\texttt{arXiv:1301.5206 [math.CT]}

\bibitem{Xu}
 J.~Xu.
   Flat covers of modules.
\textit{Lecture Notes in Math.} \textbf{1634}, Springer, 1996.

\bibitem{Yek0}
 A.~Yekutieli.
   On flatness and completion for infinitely generated modules over
noetherian rings.
\textit{Communications in Algebra} \textbf{39}, \#11,
p.~4221--4245, 2011.  \texttt{arXiv:0902.4378 [math.AC]}

\bibitem{Yek}
 A.~Yekutieli.
   A separated cohomologically complete module is complete.
\textit{Communications in Algebra} \textbf{43}, \#2, p.~616--622, 2015.
\texttt{arXiv:1312.2714 [math.AC]}

\end{thebibliography}
\end{document}